\documentclass{amsart}
\makeatletter
\@namedef{subjclassname@2020}{\textup{2020} Mathematics Subject Classification}
\makeatother
\usepackage[T1]{fontenc}

\usepackage{amsmath}
\usepackage{amssymb}
\usepackage{bm}
\usepackage{mathrsfs}
\usepackage{dsfont}
\usepackage{accents}

\usepackage{float}
\usepackage{graphicx,color}
\usepackage{ascmac}
\usepackage{harpoon}

\usepackage{tabularx}

\usepackage{tikz}
\usetikzlibrary{calc,decorations.markings}
\usetikzlibrary{arrows.meta}

\usepackage{amsthm}
\theoremstyle{definition}
\newtheorem{THM}{Theorem}
\newtheorem{LEM}[THM]{Lemma}
\newtheorem{PROP}[THM]{Proposition}

\newtheorem{DEF}[THM]{Definition}
\newtheorem{RMK}[THM]{Remark}

\newtheorem*{THM*}{Theorem}
\newtheorem*{LEM*}{Lemma}
\newtheorem*{PROP*}{Proposition}
\newtheorem*{COR*}{Corollary}
\newtheorem*{DEF*}{Definition}
\newtheorem*{RMK*}{Remark}
\newtheorem*{EX*}{Example}

\numberwithin{figure}{section}
\numberwithin{equation}{section}
\numberwithin{THM}{section}

\newcommand{\TL}{\operatorname{\mathsf{TL}}^{A_2}}
\newcommand{\JW}[2]{{JW}_{{+}^{#1}{-}^{#2}}}

\title[Twist formulas and $\mathfrak{sl}_3$~tails of $(2,2m)$-torus links]{Twist formulas for one-row colored $A_2$ webs and $\mathfrak{sl}_3$~tails of $(2,2m)$-torus links}
\author{Wataru Yuasa}
\date{}
\address{Research Institute for Mathematical Sciences\\
  Kyoto University\\
  Kitashirakawa Oiwake-cho, Sakyo-ku, Kyoto 606-8502, Japan}
\email[]{wyuasa@kurims.kyoto-u.ac.jp}

\subjclass[2020]{57K14, 57K16, 11F27}
\keywords{colored Jones polynomial; false theta series}

\begin{document}
\begin{abstract}
The $\mathfrak{sl}_3$ colored Jones polynomial $J_{\lambda}^{\mathfrak{sl}_3}(L)$ is obtained by coloring the link components with two-row Young diagram $\lambda$. 
Although it is difficult to compute $J_{\lambda}^{\mathfrak{sl}_3}(L)$ in general, we can calculate it by using Kuperberg's $A_2$ skein relation.
In this paper, 
we show some formulas for twisted two strands colored by one-row Young diagram in $A_2$ web space and compute $J_{(n,0)}^{\mathfrak{sl}_3}(T(2,2m))$ for an oriented $(2,2m)$-torus link. 
These explicit formulas derives the $\mathfrak{sl}_3$ tail of $T(2,2m)$. 
They also gives explicit descriptions of the $\mathfrak{sl}_3$ false theta series with one-row coloring because the $\mathfrak{sl}_2$ tail of $T(2,2m)$ is known as the false theta series.
\end{abstract}
\maketitle

\tikzset{->-/.style={decoration={
  markings,
  mark=at position #1 with {\arrow[thick, black]{>}}},postaction={decorate}}}
\tikzset{-<-/.style={decoration={
  markings,
  mark=at position #1 with {\arrow[thick, black]{<}}},postaction={decorate}}}
\tikzset{-|-/.style={decoration={
  markings,
  mark=at position #1 with {\arrow[black, semithick]{|}}},postaction={decorate}}}
\tikzset{
    triple/.style args={[#1] in [#2] in [#3]}{
        #1,preaction={preaction={draw,#3},draw,#2}
    }
}

\section{Introduction}
In this paper, 
we will make tools to calculate the quantum invariant of knots and links obtained from $\mathfrak{sl}_3$ by using the Kuperberg's linear skein theory. 
Especially, we will prove a full twist formula which is useful to compute explicitly the $\mathfrak{sl}_3$ colored Jones polynomial.
The quantum invariants of knots are obtained through a functor from the category of framed oriented tangles to the representation category of a quantum group of a simple Lie algebra.
Thus, one can define an invariant $J_{V}^{\mathfrak{g}}(K)$ of a knot $K$ for each simple Lie algebra $\mathfrak{g}$ and its representation $V$, that is, there exist infinitely many quantum invariants of knots.
However, it is difficult to compute $J_{V}^{\mathfrak{g}}(K)$ explicitly. 
The most simple case is $\mathfrak{g}=\mathfrak{sl}_2$ with the $2$-dimensional irreducible representation $V=\mathbb{C}^2$.
This quantum invariant reconstructs the Jones polynomial $J(K)$.
We know that $J(K)$ can be calculated from a knot diagram of $K$ by using the Kauffman bracket skein relation, for example~\cite{Kauffman87}. 
More generally, one can compute the colored Jones polynomial $J_{n}(K)$ obtained from the $(n+1)$-dimensional irreducible representation of $\mathfrak{sl}_2$ by applying the Kauffman bracket skein relation to a knot diagram colored by the Jones-Wenzl projector.
The Jones-Wenzl projector~\cite{Jones83, Wenzl87} corresponds to the projector onto the $(n+1)$-dimensional irreducible representation in $V^{{}\otimes n}$.
In this case, $J_{n}(K)$ is explicitly calculated for many knots and there are many useful formulas which decompose a tangle diagram to the linear sum of the web basis by the skein relation.
This method to calculate quantum invariants from knot diagrams is called the linear skein theory, see for example \cite{Lickorish97}.
The linear skein theory is constructed for some other $\mathfrak{g}$ than $\mathfrak{sl}_2$.
In a calculational point of view,
one of the most reasonable quantum invariants next after $\mathfrak{sl}_2$ is the $\mathfrak{sl}_3$ colored Jones polynomial $J_{(m,n)}^{\mathfrak{sl}_3}(K)$ which is obtained from the irreducible representation with the highest weight $(m,n)$ of $\mathfrak{sl}_3$.
We use the linear skein theory for $\mathfrak{sl}_3$ constructed by Kuperberg~\cite{Kuperberg94, Kuperberg96},
see Definition~\ref{A2skein} of the $A_2$ skein relation.
However, there are few non-trivial examples of explicit formulas of $J_{(m,n)}^{\mathfrak{sl}_3}(K)$.
For example, Lawrence~{\cite{Lawrence03}} calculated $J_{(m,n)}^{\mathfrak{sl}_3}(K)$ for the trefoil knot, more generally, In {\cite{GaroufalidisMortonVuong13, GaroufalidisVuong17} for the $(2,2m+1)$- and $(4,5)$-torus knots by using a representation theoretical method.
In~\cite{Yuasa17}, the author calculated $J_{(n,0)}^{\mathfrak{sl}_3}(K)$ for the two-bridge links $K$ with one-row coloring $(n,0)$ by the Kuperberg's linear skein theory. 
Important formulas used in \cite{Yuasa17} is the following full twist formulas:

\begin{THM}[anti-parallel full twist formula~\cite{Yuasa17}]\label{antifull}
Let $d=\min\{s,t\}$ and $\delta=\left|s-t\right|$.
\begin{align*}
\mathord{ \tikz[baseline=-.6ex, scale=.1]{
\draw (-6,5) -- (-9,5);
\draw (-6,-5) -- (-9,-5);
\draw (6,5) -- (9,5);
\draw (6,-5) -- (9,-5);
\draw[->-=.9, white, double=black, double distance=0.4pt, ultra thick] (-6,-5) to[out=east, in=west] (0,5);
\draw[-<-=.9, white, double=black, double distance=0.4pt, ultra thick] (-6,5) to[out=east, in=west] (0,-5);
\draw[white, double=black, double distance=0.4pt, ultra thick] (0,-5) to[out=east, in=west] (6,5);
\draw[white, double=black, double distance=0.4pt, ultra thick] (0,5) to[out=east, in=west] (6,-5);
\draw[fill=white] (-7,-7) rectangle (-6,-3);
\draw[fill=white] (-7,7) rectangle (-6,3);
\draw[fill=white] (7,-7) rectangle (6,-3);
\draw[fill=white] (7,7) rectangle (6,3);
\node at (-7,-7)[left]{$\scriptstyle{s}$};
\node at (-7,7)[left]{$\scriptstyle{t}$};
\node at (7,-7)[right]{$\scriptstyle{s}$};
\node at (7,7)[right]{$\scriptstyle{t}$};
}\ }
&=q^{\frac{st}{3}}
\sum_{l=0}^{\infty}
q^{l^2-l}q^{-(s+t)l}
(q)_{l}{s \choose l}_{q}{t \choose l}_{q}
\mathord{\ \tikz[baseline=-.6ex, scale=.1]{
\draw (-6,5) -- (-9,5);
\draw (-6,-5) -- (-9,-5);
\draw (6,5) -- (9,5);
\draw (6,-5) -- (9,-5);
\draw[-<-=.5] (-6,6) -- (6,6);
\draw[->-=.5] (-6,-6) -- (6,-6);
\draw[-<-=.5] (-6,4) to[out=east, in=east] (-6,-4);
\draw[->-=.5] (6,4) to[out=west, in=west] (6,-4);
\draw[fill=white] (-7,-7) rectangle (-6,-3);
\draw[fill=white] (-7,7) rectangle (-6,3);
\draw[fill=white] (7,-7) rectangle (6,-3);
\draw[fill=white] (7,7) rectangle (6,3);
\node at (-7,-7)[left]{$\scriptstyle{s}$};
\node at (-7,7)[left]{$\scriptstyle{t}$};
\node at (7,-7)[right]{$\scriptstyle{s}$};
\node at (7,7)[right]{$\scriptstyle{t}$};
\node at (0,6)[above]{$\scriptstyle{t-l}$};
\node at (0,-6)[below]{$\scriptstyle{s-l}$};
\node at (-6,0){$\scriptstyle{l}$};
\node at (6,0){$\scriptstyle{l}$};
}\ }\\
&=q^{-\frac{2}{3}d(d+\delta)-d}
\sum_{k=0}^{d}
q^{k(k+\delta)+k}
\frac{(q)_{d+\delta}}{(q)_{k+\delta}}{d \choose k}_{q}
\mathord{\ \tikz[baseline=-.6ex, scale=.1]{
\draw (-6,5) -- (-9,5);
\draw (-6,-5) -- (-9,-5);
\draw (6,5) -- (9,5);
\draw (6,-5) -- (9,-5);
\draw[-<-=.5] (-6,6) -- (6,6);
\draw[->-=.5] (-6,-6) -- (6,-6);
\draw[-<-=.5] (-6,4) to[out=east, in=east] (-6,-4);
\draw[->-=.5] (6,4) to[out=west, in=west] (6,-4);
\draw[fill=white] (-7,-7) rectangle (-6,-3);
\draw[fill=white] (-7,7) rectangle (-6,3);
\draw[fill=white] (7,-7) rectangle (6,-3);
\draw[fill=white] (7,7) rectangle (6,3);
\node at (-7,-7)[left]{$\scriptstyle{s}$};
\node at (-7,7)[left]{$\scriptstyle{t}$};
\node at (7,-7)[right]{$\scriptstyle{s}$};
\node at (7,7)[right]{$\scriptstyle{t}$};
\node at (0,6)[above]{$\scriptstyle{k+t-d}$};
\node at (0,-6)[below]{$\scriptstyle{k+s-d}$};
\node at (-4,0)[left]{$\scriptstyle{d-k}$};
\node at (4,0)[right]{$\scriptstyle{d-k}$};
}\ }.
\end{align*}
\end{THM}
From this theorem, the author obtained an $m$-full twist formula for anti-parallel two strands.
\begin{THM}[anti-parallel $m$-full twist formula~\cite{Yuasa17}]\label{antimfull}
Let $\underline{k}=(k_{1},\dots,k_{m})$ be an $m$-tuple of integers, $k_{0}=d=\min\{s,t\}$, and $\delta=\left| s-t \right|$.
\begin{align*}
\mathord{\ \tikz[baseline=-.6ex]{
\begin{scope}[xshift=-1cm]
\draw 
(-.5,.4) -- +(-.2,0)
(-.5,-.4) -- +(-.2,0);
\draw[->-=1, white, double=black, double distance=0.4pt, ultra thick] 
(-.5,-.4) to[out=east, in=west] (.0,.4);
\draw[-<-=1, white, double=black, double distance=0.4pt, ultra thick] 
(-.5,.4) to[out=east, in=west] (.0,-.4);
\draw[white, double=black, double distance=0.4pt, ultra thick] 
(0,-.4) to[out=east, in=west] (.5,.4);
\draw[white, double=black, double distance=0.4pt, ultra thick] 
(0,.4) to[out=east, in=west] (.5,-.4);
\draw[fill=white] (-.5,-.6) rectangle +(-.1,.4);
\draw[fill=white] (-.5,.6) rectangle +(-.1,-.4);
\node at (-.5,-.6)[left]{$\scriptstyle{s}$};
\node at (-.5,.6)[left]{$\scriptstyle{t}$};
\end{scope}
\node at (.0,.0){$\cdots$};
\node at (.0,-.4)[below]{$\scriptstyle{m\text{ full twists}}$};
\begin{scope}[xshift=1cm]
\draw
(.5,-.4) -- +(.2,0)
(.5,.4) -- +(.2,0);
\draw[->-=1, white, double=black, double distance=0.4pt, ultra thick] 
(-.5,-.4) to[out=east, in=west] (.0,.4);
\draw[-<-=1, white, double=black, double distance=0.4pt, ultra thick] 
(-.5,.4) to[out=east, in=west] (.0,-.4);
\draw[white, double=black, double distance=0.4pt, ultra thick] 
(0,-.4) to[out=east, in=west] (.5,.4);
\draw[white, double=black, double distance=0.4pt, ultra thick] 
(0,.4) to[out=east, in=west] (.5,-.4);
\draw[fill=white] (.5,-.6) rectangle +(.1,.4);
\draw[fill=white] (.5,.6) rectangle +(.1,-.4);
\node at (.5,-.6)[right]{$\scriptstyle{s}$};
\node at (.5,.6)[right]{$\scriptstyle{t}$};
\end{scope}
}\ }
&=q^{-\frac{2m}{3}k_{0}(k_{0}+\delta)-2mk_{0}}
\sum_{k_{0}\geq k_{1}\geq\cdots\geq k_{m}\geq 0} C(\underline{k})
\mathord{\ \tikz[baseline=-.6ex, scale=.1]{
\draw (-6,5) -- (-9,5);
\draw (-6,-5) -- (-9,-5);
\draw (6,5) -- (9,5);
\draw (6,-5) -- (9,-5);
\draw[-<-=.5] (-6,6) -- (6,6);
\draw[->-=.5] (-6,-6) -- (6,-6);
\draw[-<-=.5] (-6,4) to[out=east, in=east] (-6,-4);
\draw[->-=.5] (6,4) to[out=west, in=west] (6,-4);
\draw[fill=white] (-7,-7) rectangle (-6,-3);
\draw[fill=white] (-7,7) rectangle (-6,3);
\draw[fill=white] (7,-7) rectangle (6,-3);
\draw[fill=white] (7,7) rectangle (6,3);
\node at (-7,-7)[left]{$\scriptstyle{s}$};
\node at (-7,7)[left]{$\scriptstyle{t}$};
\node at (7,-7)[right]{$\scriptstyle{s}$};
\node at (7,7)[right]{$\scriptstyle{t}$};
\node at (0,6)[above]{$\scriptstyle{k_{m}+(t-k_{0})}$};
\node at (0,-6)[below]{$\scriptstyle{k_{m}+(s-k_{0})}$};
\node at (-4,0)[left]{$\scriptstyle{k_{0}-k_{m}}$};
\node at (4,0)[right]{$\scriptstyle{k_{0}-k_{m}}$};
}\ },
\end{align*}
where
\[
C(\underline{k})=q^{\sum_{i=1}^{m}k_{i}(k_{i}+\delta)+2k_{i}}q^{k_{0}-k_{m}}\frac{(q)_{k_{0}+\delta}}{(q)_{k_{m}+\delta}}{k_{0} \choose k_{1}',k_{2}',\dots,k_{m}',k_{m}}_{q}
\]
 and $k_{i+1}'=k_i-k_{i+1}$ for $i=0,1,\dots,m-1$.
\end{THM}

The above formula plays an important role in the study of $\mathfrak{sl}_3$ tails of knots and links.
The tail of a knot $K$ is a $q$-series which is a limit of the colored Jones polynomials $\{J_{n}(K;q)\}_n$. 
Independently, the existence of the tails for alternating knots was shown in \cite{DasbachLin06, DasbachLin07} and \cite{GaroufalidisLe15}, more generally, for adequate links in \cite{Armond13}. 
In \cite{GaroufalidisLe15}, Garoufalidis and L\^{e} showed a more general stability, the existence of the tail is the zero-stability, for alternating knots.
Some explicit descriptions of tails are known for $T(2,2m+1)$ in \cite{ArmondDasbach11A}, 
for $T(2,2m)$ in \cite{Hajij16}, 
for a pretzel knot $P(2k+1,2,2l+1)$ in \cite{ElhamdadiHajij17}, 
for knots with small crossing numbers in \cite{KeilthyOsburn16}, \cite{BeirneOsburn17}, and \cite{GaroufalidisLe15}.
Especially, the tail of $T(2,2m+1)$ is given by the theta series and one of $T(2,2m)$ is the false theta series.

One can consider a tail for the $\mathfrak{sl}_3$ colored Jones polynomial. 
We call it the $\mathfrak{sl}_3$ tail.
However, there are many problems such that how do we define a limit for $\{J_{(k,l)}^{\mathfrak{sl}_3}(K)\}_{k,l}$ and the existence of the $\mathfrak{sl}_3$ tail etc.
As a case study, the author gave an explicit formula of the $\mathfrak{sl}_3$ tail for the $(2,2m)$-torus link $T_{\leftrightarrows}(2,2m)$ with anti-parallel orientation by using Theorem~\ref{antimfull} in \cite{Yuasa18}, see also Theomrem~\ref{antiparalleltail}.
This $\mathfrak{sl}_3$ tail can be considered a special type of the $\mathfrak{sl}_3$ false theta series.
In fact, Bringmann-Kaszian-Milas~\cite{BringmannKaszianMilas19} commented that the $\mathfrak{sl}_3$ tail of $T_{\leftrightarrows}(2,2m)$ coincides with the diagonal part of the $\mathfrak{sl}_3$ false theta series defined through the study of vertex operator algebras in~\cite{BringmannMilas15,BringmannMilas17,CreutzigMilas14,CreutzigMilas17}.

This paper is organized as follows.
In section $2$, we review the Kuperberg's linear skein theory fo $A_2$ and give a diagrammatic definition of the Jones-Wenzl projector for $A_2$.
In section $3$, we derive formulas which decompose an $m$-full twist of two-strands with one-row coloring into the linear sum of $A_2$ basis webs.
We use the generating function for Young diagrams by involving the decompositions of the $A_2$ web with lattice paths in proofs of these formulas.
As an application, we compute the $\mathfrak{sl}_3$ colored Jones polynomials and the $\mathfrak{sl}_3$ tails of oriented $(2,m)$-torus links with one-row coloring in sectoin~$4$.

\section{Definition and properties of the $A_2$ clasp}
We use the following notation.
\begin{itemize}
\item $\left[n\right]=\frac{q^{\frac{n}{2}}-q^{-\frac{n}{2}}}{q^{\frac{1}{2}}-q^{\frac{1}{2}}}$ is a {\em quantum integer} for $n\in\mathbb{Z}_{{}\geq 0}$.
\item ${n \brack k}=\frac{\left[n\right]}{\left[k\right]\left[n-k\right]}$ for $0\leq k\leq n$ and ${n \brack k}=0$ for $k>n$.
\item $(q)_{n}=\prod_{i=1}^{n} \left(1-q^i\right)$ is a {\em $q$-Pochhammer symbol}.
\item $\binom{n}{k}_{q}=\frac{(q)_{n}}{(q)_{k}(q)_{n-k}}$ for $0\leq k\leq n$ and $\binom{n}{k}_{q}=0$ for $k>n$.
\item $\binom{n}{k_{1},k_{2},\dots,k_{m}}_{q}=\frac{(q)_{n}}{(q)_{k_{1}}(q)_{k_{2}}\cdots(q)_{k_{m}}}$ for positive integers $k_i$'s such that $\sum_{i=1}^{m}k_{i}=n$.
\end{itemize}

Let us define $A_2$ web spaces based on \cite{Kuperberg96}. 
We consider a disk $D$ with signed marked points $(P,\epsilon)$ on its boundary where $P\subset\partial D$ is a finite set and $\epsilon\colon P\to\{{+},{-}\}$ a map. 

A {\em tangled bipartite uni-trivalent graph} on $D$ is an immersion of a directed graph into $D$ satisfying (1) -- (4):
\begin{enumerate}
\item the valency of a vertex of underlying graph is $1$ or $3$,
\item all crossing points are transversal double points of two edges with under/over information,
\item the set of univalent vertices coincides with $P$,
\item a neighborhood of a vertex is one of the followings:
\begin{itemize}
\item[] 
{\ \tikz[baseline=-.6ex, scale=.1]{
\draw[dashed, fill=white] (0,0) circle [radius=7];
\draw[-<-=.5] (0:0) -- (90:7); 
\draw[-<-=.5] (0:0) -- (210:7); 
\draw[-<-=.5] (0:0) -- (-30:7);
}\ }
,
{\ \tikz[baseline=-.6ex, scale=.1]{
\draw[dashed, fill=white] (0,0) circle [radius=7];
\draw[->-=.5] (0:0) -- (90:7); 
\draw[->-=.5] (0:0) -- (210:7); 
\draw[->-=.5] (0:0) -- (-30:7);
}\ }
,
{\ \tikz[baseline=-.6ex, scale=.1]{
\draw[dashed] (0,0) circle [radius=7];
\draw[-<-=.5] (0,0) -- (7,0);
\node at (0,0) [left]{${+}$};
\draw[thick] (0,7) -- (0,-7);
\draw[fill=cyan] (0,0) circle [radius=.5];
}\ }
,
{\ \tikz[baseline=-.6ex, scale=.1]{
\draw[dashed] (0,0) circle [radius=7];
\draw[->-=.5] (0,0)--(7,0);
\node at (0,0) [left]{${-}$};
\draw[thick] (0,7) -- (0,-7);
\draw[fill=cyan] (0,0) circle [radius=.5];
}\ }
.
\end{itemize}
\end{enumerate}

\begin{DEF}[$A_2$ web space~\cite{Kuperberg96}]\label{A2skein}
Let $G(\epsilon;D)$ be the set of boundary fixing isotopy classes of tangled trivalent graphs on $D$.
The $A_2$ web space $W(\epsilon;D)$ is the quotient of the $\mathbb{C}[q^{\frac{1}{6}},q^{-\frac{1}{6}}]$-vector space on $G(\epsilon;D)$ by the following {\em $A_2$ skein relation}:
\begin{itemize}
\item
$
\mathord{\ \tikz[baseline=-.6ex, scale=.1]{
\draw[thin, dashed, fill=white] (0,0) circle [radius=7];
\draw[->-=.8] (-45:7) -- (135:7);
\draw[->-=.8, white, double=black, double distance=0.4pt, line width=2.4pt] (-135:7) -- (45:7);
}\ }
=q^{\frac{1}{3}}
\mathord{\ \tikz[baseline=-.6ex, scale=.1]{
\draw[thin, dashed, fill=white] (0,0) circle [radius=7];
\draw[->-=.5] (-45:7) to[out=north west, in=south] (3,0) to[out=north, in=south west] (45:7);
\draw[->-=.5] (-135:7) to[out=north east, in=south] (-3,0) to[out=north, in=south east] (135:7);
}\ }
-q^{-\frac{1}{6}}
\mathord{\ \tikz[baseline=-.6ex, scale=.1]{
\draw[thin, dashed, fill=white] (0,0) circle [radius=7];
\draw[->-=.5] (-45:7) -- (0,-3);
\draw[->-=.5] (-135:7) -- (0,-3);
\draw[-<-=.5] (45:7) -- (0,3);
\draw[-<-=.5] (135:7) -- (0,3);
\draw[-<-=.5] (0,-3) -- (0,3);
}\ }$,
\item
$\mathord{\ \tikz[baseline=-.6ex, scale=.1]{
\draw[thin, dashed, fill=white] (0,0) circle [radius=7];
\draw[->-=.8] (-135:7) -- (45:7);
\draw[->-=.8, white, double=black, double distance=0.4pt, line width=2.4pt] (-45:7) -- (135:7);
}\ }
=q^{-\frac{1}{3}}
\mathord{\ \tikz[baseline=-.6ex, scale=.1]{
\draw[thin, dashed, fill=white] (0,0) circle [radius=7];
\draw[->-=.5] (-45:7) to[out=north west, in=south] (3,0) to[out=north, in=south west] (45:7);
\draw[->-=.5] (-135:7) to[out=north east, in=south] (-3,0) to[out=north, in=south east] (135:7);
}\ }
-q^{\frac{1}{6}}
\mathord{\ \tikz[baseline=-.6ex, scale=.1]{
\draw[thin, dashed, fill=white] (0,0) circle [radius=7];
\draw[->-=.5] (-45:7) -- (0,-3);
\draw[->-=.5] (-135:7) -- (0,-3);
\draw[-<-=.5] (45:7) -- (0,3);
\draw[-<-=.5] (135:7) -- (0,3);
\draw[-<-=.5] (0,-3) -- (0,3);
}\ }$
\item 
$\mathord{\ \tikz[baseline=-.6ex, scale=.1]{
\draw[thin, dashed, fill=white] (0,0) circle [radius=7];
\draw[-<-=.6] (-45:7) -- (-45:3);
\draw[->-=.6] (-135:7) -- (-135:3);
\draw[->-=.6] (45:7) -- (45:3);
\draw[-<-=.6] (135:7) -- (135:3);
\draw[-<-=.5] (45:3) -- (135:3);
\draw[->-=.5] (-45:3) -- (-135:3);
\draw[-<-=.5] (45:3) -- (-45:3);
\draw[->-=.5] (135:3) -- (-135:3);
}\ }
=
\mathord{\ \tikz[baseline=-.6ex, scale=.1]{
\draw[thin, dashed, fill=white] (0,0) circle [radius=7];
\draw[-<-=.5] (-45:7) to[out=north west, in=south] (3,0) to[out=north, in=south west] (45:7);
\draw[->-=.5] (-135:7) to[out=north east, in=south] (-3,0) to[out=north, in=south east] (135:7);
}\ }
+
\mathord{\ \tikz[baseline=-.6ex, rotate=90, scale=.1]{
\draw[thin, dashed, fill=white] (0,0) circle [radius=7];
\draw[-<-=.5] (-45:7) to[out=north west, in=south] (3,0) to[out=north, in=south west] (45:7);
\draw[->-=.5] (-135:7) to[out=north east, in=south] (-3,0) to[out=north, in=south east] (135:7);
}\ }
$,\vspace{1ex}
\item 
$\mathord{\ \tikz[baseline=-.6ex, scale=.1]{
\draw[thin, dashed, fill=white] (0,0) circle [radius=7];
\draw[->-=.5] (0,-7) -- (0,-3);
\draw[->-=.5] (0,3) -- (0,7);
\draw[-<-=.5] (0,-3) to[out=east, in=south] (3,0) to[out=north, in=east] (0,3);
\draw[-<-=.5] (0,-3) to[out=west, in=south] (-3,0) to[out=north, in=west] (0,3);
}\ }
=
\left[2\right]
\mathord{\ \tikz[baseline=-.6ex, scale=.1]{
\draw[thin, dashed, fill=white] (0,0) circle [radius=7];
\draw[->-=.5] (0,-7) -- (0,7);
}\ }
$,\vspace{1ex}
\item 
$
\mathord{\ \tikz[baseline=-.6ex, scale=.1]{
\draw[thin, dashed, fill=white] (0,0) circle [radius=7];
\draw[->-=.5] (0,0) circle [radius=3];
}\ }
=
\left[3\right]
\mathord{\ \tikz[baseline=-.6ex, scale=.1]{
\draw[thin, dashed, fill=white] (0,0) circle [radius=7];
}\ }
=
\mathord{\ \tikz[baseline=-.6ex, scale=.1]{
\draw[thin, dashed, fill=white] (0,0) circle [radius=7];
\draw[-<-=.5] (0,0) circle [radius=3];
}\ }
.
$
\end{itemize}
An element in $W(\epsilon;D)$ is called {\em web} and an element in $G(\epsilon;D)$ without crossings which has no internal $0$-, $2$-, $4$-gons a {\em basis web}. 
Any web is described as the sum of basis webs.
\end{DEF}

The $A_2$ skein relation realize the {\em Reidemeister moves} (R1) -- (R4), that is,
we can show that webs represent diagrams in the left side and right side is the same web in $W(D;\epsilon)$.
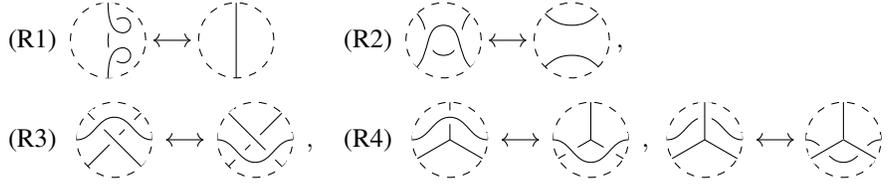
\begin{align*}
&\text{(R1)}~
\mathord{\ \tikz[baseline=-.6ex, scale=.1]{
\draw [thin, dashed] (0,0) circle [radius=5];
\draw (3,-2)
to[out=south, in=east] (2,-3)
to[out=west, in=south] (0,0)
to[out=north, in=west] (2,3)
to[out=east, in=north] (3,2);
\draw[white, double=black, double distance=0.4pt, line width=2.4pt] (0,-5) 
to[out=north, in=west] (2,-1)
to[out=east, in=north] (3,-2);
\draw[white, double=black, double distance=0.4pt, line width=2.4pt] (3,2)
to[out=south, in=east] (2,1)
to[out=west, in=south] (0,5);
}\ }
\tikz[baseline=-.6ex, scale=.5]{
\draw [<->, xshift=1.5cm] (1,0)--(2,0);
} 
\mathord{\ \tikz[baseline=-.6ex, scale=.1, xshift=3cm]{
\draw[thin, dashed] (0,0) circle [radius=5];
\draw (90:5) to (-90:5);
}\ }
&&\text{(R2)}~
\mathord{\ \tikz[baseline=-.6ex, scale=.1]{
\draw[thin, dashed] (0,0) circle [radius=5];
\draw (135:5) to[out=south east, in=west] (0,-2) to[out=east, in=south west](45:5);
\draw[white, double=black, double distance=0.4pt, line width=2.4pt]
(-135:5) to[out=north east, in=west] (0,2) to[out=east, in=north west] (-45:5);
}\ }
\tikz[baseline=-.6ex, scale=.5]{
\draw [<->, xshift=1.5cm] (1,0)--(2,0);
}
\mathord{\ \tikz[baseline=-.6ex, scale=.1, xshift=3cm]{
\draw[thin, dashed] (0,0) circle [radius=5];
\draw (135:5) to[out=south east, in=west](0,2) to[out=east, in=south west](45:5);
\draw (-135:5) to[out=north east, in=west](0,-2) to[out=east, in=north west] (-45:5);
}\ },\\
&\text{(R3)}~
\mathord{\ \tikz[baseline=-.6ex, scale=.1]{
\draw [thin, dashed] (0,0) circle [radius=5];
\draw (-135:5) -- (45:5);
\draw[white, double=black, double distance=0.4pt, line width=2.4pt] (135:5) -- (-45:5);
\draw[white, double=black, double distance=0.4pt, line width=2.4pt] 
(180:5) to[out=east, in=west](0,3) to[out=east, in=west] (-0:5);
}\ }
\tikz[baseline=-.6ex, scale=.5]{
\draw[<->, xshift=1.5cm] (1,0)--(2,0);
}
\mathord{\ \tikz[baseline=-.6ex, scale=.1, xshift=3cm]{
\draw [thin, dashed] (0,0) circle [radius=5];
\draw (-135:5) -- (45:5);
\draw [white, double=black, double distance=0.4pt, line width=2.4pt] (135:5) -- (-45:5);
\draw[white, double=black, double distance=0.4pt, line width=2.4pt] 
(180:5) to[out=east, in=west] (0,-3) to[out=east, in=west] (-0:5);
}\ },
&&\text{(R4)}~
\mathord{\ \tikz[baseline=-.6ex, scale=.1]{
\draw [thin, dashed] (0,0) circle [radius=5];
\draw (0:0) -- (90:5); 
\draw (0:0) -- (210:5); 
\draw (0:0) -- (-30:5);
\draw[white, double=black, double distance=0.4pt, line width=2.4pt]
(180:5) to[out=east, in=west] (0,3) to[out=east, in=west] (0:5);
}\ }
\tikz[baseline=-.6ex, scale=.5]{
\draw[<->, xshift=1.5cm] (1,0)--(2,0);
}
\mathord{\ \tikz[baseline=-.6ex, scale=.1]{
\draw [thin, dashed] (0,0) circle [radius=5];
\draw (0:0) -- (90:5); 
\draw (0:0) -- (210:5); 
\draw (0:0) -- (-30:5);
\draw[white, double=black, double distance=0.4pt, line width=2.4pt]
(180:5) to[out=east, in=west] (0,-3) to[out=east, in=west](0:5);
}\ }, 
\mathord{\ \tikz[baseline=-.6ex, scale=.1]{
\draw [thin, dashed] (0,0) circle [radius=5];
\draw[white, double=black, double distance=0.4pt, line width=2.4pt]
(180:5) to[out=east, in=west] (0,3) to[out=east, in=west] (0:5);
\draw[white, double=black, double distance=0.4pt, line width=2.4pt] (0:0) -- (90:5); 
\draw (0:0) -- (210:5); 
\draw (0:0) -- (-30:5);
}\ }
\tikz[baseline=-.6ex, scale=.5]{
\draw[<->, xshift=1.5cm] (1,0)--(2,0);
}
\mathord{\ \tikz[baseline=-.6ex, scale=.1]{
\draw [thin, dashed] (0,0) circle [radius=5];
\draw[white, double=black, double distance=0.4pt, ultra thick]
(180:5) to[out=east, in=west] (0,-3) to[out=east, in=west](0:5);
\draw (0:0) -- (90:5); 
\draw[white, double=black, double distance=0.4pt, line width=2.4pt] (0:0) -- (210:5); 
\draw[white, double=black, double distance=0.4pt, line width=2.4pt] (0:0) -- (-30:5);
}\ }.
\end{align*}

We review a diagrammatic definition of an $A_2$ clasp introduced in \cite{Kuperberg96, OhtsukiYamada97, Kim07} and its properties. 
The $A_2$ clasp gives a coloring of strands in a web by pairs of non-negative integers. 
It plays an important role as is the case with the Jones-Wenzl projector.

We construct a projector called the $A_2$ clasp of type $(n,m)$ in a special web space $\TL({+}^{n}{-}^{m},{+}^{n}{-}^{m})$. 

\begin{DEF}[the Temperley-Lieb category for $A_2$]
Let $D=\left[0,1\right]\times\left[0,1\right]$ and $\mathbf{n}$ denotes a set of $n$ points on $I=\left[0,1\right]$ dividing it into $n+1$ equal parts.
The {\em Temperley-Lieb category $\TL$} is a linear category over $\mathbb{C}(q^{\frac{1}{6}})$ is defined as follows:
\begin{itemize}
\item an object is a word (finite sequence) over $\{{+},{-}\}$,
\item the tensor product of two words is defined by the product (concatenation),
\item the space of morphisms $\TL(\alpha,\beta)$ is the web space $W(\bar{\alpha}\sqcup\beta;D)$ where $\bar{\alpha}$ is the word consisting of opposite signs of $\alpha$. 
\end{itemize}
We identify an object $\alpha$ with length $n$ as a map $\mathbf{n}\to\{{+},{-}\}$ by using the order on $\mathbf{n}$. 
A map $\bar{\alpha}\sqcup\beta$ means that the domain $\mathbf{n}\subset I$ of the map $\bar{\alpha}$ is identified with the top edge $\left[0,1\right]\times\{0\}$ and $\beta$ identified with the bottom edge $\left[0,1\right]\times\{1\}$. 
\begin{itemize}
\item The composition $GF\in G(\bar{\alpha}\sqcup\gamma;D)$ of $F\in G(\bar{\alpha}\sqcup\beta;D)$ and $G\in G(\bar{\beta}\sqcup\gamma;D)$ is given by gluing the top edge of $F$ and the bottom edge of $G$.
\item The ``tensor product'' $F\otimes G\in G(\bar{\alpha}_{1}\bar{\alpha}_{2},\beta_{1}\beta_{2};D)$ of $F\in G(\bar{\alpha}_1\sqcup\beta_1;D)$ and $G\in G(\alpha_2\sqcup\beta_2;D)$ by gluing the right edge $\{1\}\times\left[0,1\right]$ of $F$ and the left edge $\{0\}\times\left[0,1\right]$ of $G$.
\end{itemize}
They define the composition and the tensor product on the space of morphisms by linearization. 
The diagrammatic description of them is in Figure~{\ref{TLcat_fig}}.
\end{DEF}
\begin{figure}
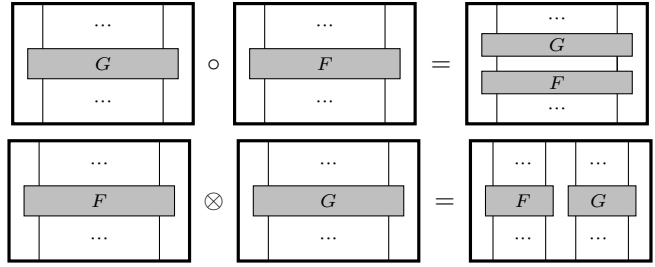

\begin{center}
\[
\mathord{\ \tikz[baseline=-.6ex, scale=.1]{
\draw (-8,8) -- (-8,-8);
\draw (8,8) -- (8,-8);
\node at (0,5) {$\scriptstyle{\cdots}$};
\node at (0,-5) {$\scriptstyle{\cdots}$};
\draw (8,-2) -- (8,2);
\draw[fill=lightgray] (-10,-2) rectangle (10,2);
\node at (0,0) {$\scriptstyle{G}$};
\draw[very thick] (-12,-8) rectangle (12,8);
}\ }
\circ
\mathord{\ \tikz[baseline=-.6ex, scale=.1]{
\draw (-8,8) -- (-8,-8);
\draw (8,8) -- (8,-8);
\node at (0,5) {$\scriptstyle{\cdots}$};
\node at (0,-5) {$\scriptstyle{\cdots}$};
\draw (8,-2) -- (8,2);
\draw[fill=lightgray] (-10,-2) rectangle (10,2);
\node at (0,0) {$\scriptstyle{F}$};
\draw[very thick] (-12,-8) rectangle (12,8);
}\ }
=
\mathord{\ \tikz[baseline=-.6ex, scale=.1]{
\draw (-8,8) -- (-8,-8);
\draw (8,8) -- (8,-8);
\node at (0,6) {$\scriptstyle{\cdots}$};
\node at (0,-6) {$\scriptstyle{\cdots}$};
\draw (8,-2) -- (8,2);
\draw[fill=lightgray] (-10,1) rectangle (10,4);
\node at (0,2.5) {$\scriptstyle{G}$};
\draw[fill=lightgray] (-10,-1) rectangle (10,-4);
\node at (0,-2.5) {$\scriptstyle{F}$};
\draw[very thick] (-12,-8) rectangle (12,8);
}\ }
\]
\[
\mathord{\ \tikz[baseline=-.6ex, scale=.1]{
\draw (-8,8) -- (-8,-8);
\draw (8,8) -- (8,-8);
\node at (0,5) {$\scriptstyle{\cdots}$};
\node at (0,-5) {$\scriptstyle{\cdots}$};
\draw (8,-2) -- (8,2);
\draw[fill=lightgray] (-10,-2) rectangle (10,2);
\node at (0,0) {$\scriptstyle{F}$};
\draw[very thick] (-12,-8) rectangle (12,8);
}\ }
\otimes
\mathord{\ \tikz[baseline=-.6ex, scale=.1]{
\draw (-8,8) -- (-8,-8);
\draw (8,8) -- (8,-8);
\node at (0,5) {$\scriptstyle{\cdots}$};
\node at (0,-5) {$\scriptstyle{\cdots}$};
\draw (8,-2) -- (8,2);
\draw[fill=lightgray] (-10,-2) rectangle (10,2);
\node at (0,0) {$\scriptstyle{G}$};
\draw[very thick] (-12,-8) rectangle (12,8);
}\ }
=
\mathord{\ \tikz[baseline=-.6ex, scale=.1]{
\draw (-9,8) -- (-9,-8);
\draw (-2,8) -- (-2,-8);
\node at (-5,5) {$\scriptstyle{\cdots}$};
\node at (-5,-5) {$\scriptstyle{\cdots}$};
\draw (9,8) -- (9,-8);
\draw (2,8) -- (2,-8);
\node at (5,5) {$\scriptstyle{\cdots}$};
\node at (5,-5) {$\scriptstyle{\cdots}$};
\draw[fill=lightgray] (-10,-2) rectangle (-1,2);
\node at (-5,0) {$\scriptstyle{F}$};
\draw[fill=lightgray] (10,-2) rectangle (1,2);
\node at (5,0) {$\scriptstyle{G}$};
\draw[very thick] (-12,-8) rectangle (12,8);
}\ }
\]
\end{center}
\caption{the composition and the tensor product in $\TL$}
\label{TLcat_fig}
\end{figure}

$\TL$ is generated by identity morphisms $\mathrm{id}_{+}\in \TL({+},{+})$, $\mathrm{id}_{-}\in \TL({-},{-})$ and the following morphisms.
\begin{align*}
& t_{-}^{++}
=
\mathord{\ \tikz[baseline=-.6ex, scale=.1]{
\draw[thick] (-8,-5) rectangle (8,5);
\draw[->-=.5] (0,0) -- (-5,5);
\draw[->-=.5] (0,0) -- (5,5);
\draw[->-=.5] (0,0) -- (0,-5);
}\ }, 
t_{+}^{--}
=
\mathord{\ \tikz[baseline=-.6ex, scale=.1]{
\draw[thick] (-8,-5) rectangle (8,5);
\draw[-<-=.5] (0,0) -- (-5,5);
\draw[-<-=.5] (0,0) -- (5,5);
\draw[-<-=.5] (0,0) -- (0,-5);
}\ }, 
t_{--}^{+}
=
\mathord{\ \tikz[baseline=-.6ex, scale=.1, yscale=-1]{
\draw[thick] (-8,-5) rectangle (8,5);
\draw[->-=.5] (0,0) -- (-5,5);
\draw[->-=.5] (0,0) -- (5,5);
\draw[->-=.5] (0,0) -- (0,-5);
}\ }, 
t_{++}^{-}
=
\mathord{\ \tikz[baseline=-.6ex, scale=.1, yscale=-1]{
\draw[thick] (-8,-5) rectangle (8,5);
\draw[-<-=.5] (0,0) -- (-5,5);
\draw[-<-=.5] (0,0) -- (5,5);
\draw[-<-=.5] (0,0) -- (0,-5);
}\ },\\
&b^{-+}
=
\mathord{\ \tikz[baseline=-.6ex, scale=.1]{
\draw[thick] (-8,-5) rectangle (8,5);
\draw[->-=.5] (-5,5) to[out=south, in=west] (0,1) to[out=east, in=south] (5,5);
}\ }, 
b^{+-}
=
\mathord{\ \tikz[baseline=-.6ex, scale=.1]{
\draw[thick] (-8,-5) rectangle (8,5);
\draw[-<-=.5] (-5,5) to[out=south, in=west] (0,1) to[out=east, in=south] (5,5);
}\ }, 
d_{+-}
=
\mathord{\ \tikz[baseline=-.6ex, scale=.1, yscale=-1]{
\draw[thick] (-8,-5) rectangle (8,5);
\draw[->-=.5] (-5,5) to[out=south, in=west] (0,1) to[out=east, in=south] (5,5);
}\ }, 
d_{-+}
=
\mathord{\ \tikz[baseline=-.6ex, scale=.1, yscale=-1]{
\draw[thick] (-8,-5) rectangle (8,5);
\draw[-<-=.5] (-5,5) to[out=south, in=west] (0,1) to[out=east, in=south] (5,5);
}\ }. 
\end{align*}

Let us define one colored $A_2$ clasp $\JW{m}{0}$ before defining $\JW{m}{n}$. 
We do not describe the boundary of $D$ in what follows.

\begin{DEF}[The $A_2$ clasp in $\TL({+}^{m},{+}^{m})$]\label{sclasp}
The $A_{2}$ clasp $\JW{m}{0}$ described by a white box with $m\in\mathbb{Z}_{{}\geq 0}$ is defined as follows.
\begin{enumerate}
\item
$\mathord{\ \tikz[baseline=-.6ex, scale=.1]{
\draw[->-=.8] (0,-6) -- (0,6);
\draw[fill=white] (-3,-1) rectangle (3,1);
\node at (0,5) [left]{${\scriptstyle 0}$};
}\ }
=\emptyset$ (the empty diagram),
$\mathord{\ \tikz[baseline=-.6ex, scale=.1]{
\draw[->-=.8] (0,-6) -- (0,6);
\draw[fill=white] (-3,-1) rectangle (3,1);
\node at (0,5) [left]{${\scriptstyle 1}$};
}\ }
=
\,\tikz[baseline=-.6ex, scale=.1]{
\draw[->-=.5] (0,-6) -- (0,6); 
}\,
=\mathrm{id}_{+}\in \TL({+},{+})$

\item
$\mathord{\ \tikz[baseline=-.6ex, scale=.1]{
\draw[->-=.8] (0,-6) -- (0,6);
\draw[fill=white] (-3,-1) rectangle (3,1);
\node at (0,5) [left]{${\scriptstyle m+1}$};
}\ }
= 
\mathord{\ \tikz[baseline=-.6ex, scale=.1]{
\draw[->-=.8] (0,-6) -- (0,6);
\draw[->-=.5] (5,-6) -- (5,6);
\draw[fill=white] (-2,-1) rectangle (2,1);
\node at (0,5) [left]{${\scriptstyle m}$};
}\ }
-\frac{\left[n\right]}{\left[n+1\right]}
\mathord{\ \tikz[baseline=-.6ex, scale=.1]{
\draw[->-=.6] (0,4) -- (0,7);
\draw[->-=.6] (0,-7) -- (0,-4);
\draw[->-=.5] (-2,-3) -- (-2,3);
\draw[-<-=.7] (5,-1) -- (5,1);
\draw[-<-=.7] (2,3) to[out=south, in=west] (5,1);
\draw[-<-=.3] (8,7) -- (8,3) to[out=south, in=east] (5,1);
\draw[->-=.7, yscale=-1] (2,3) to[out=south, in=west] (5,1);
\draw[->-=.3, yscale=-1] (8,7) -- (8,3) to[out=south, in=east] (5,1);
\draw[fill=white] (-3,3) rectangle (3,4);
\draw[fill=white] (-3,-3) rectangle (3,-4);
\node at (0,6) [left]{${\scriptstyle m}$};
\node at (0,-6) [left]{${\scriptstyle m}$};
\node at (-2,0) [left]{${\scriptstyle m-1}$};
}\ }\in\TL({+}^{n+1},{+}^{n+1})$
\end{enumerate}
The $A_2$ clasp in $\JW{0}{m}\in\TL({-}^m,{-}^m)$ is also defined in the same way.
\end{DEF}
We introduce the $A_2$ clasp $\JW{m}{n}$ in $\TL({+}^{m}{-}^{n},{+}^{m}{-}^{n})$ based on~\cite{OhtsukiYamada97}.
\begin{DEF}[The $A_2$ clasp in $\TL({+}^{m}{-}^{n},{+}^{m}{-}^{n})$]\label{dclasp}
\[
\JW{m}{n}
=
\mathord{\ \tikz[baseline=-.6ex, scale=.1]{
\draw[->-=.2, ->-=.8] (-2,-5) -- (-2,5);
\draw[-<-=.2, -<-=.8] (2,-5) -- (2,5);
\draw[fill=white] (-4,-1) rectangle (4,1);
\node at (-2,5) [left]{${\scriptstyle m}$};
\node at (-2,-5) [left]{${\scriptstyle m}$};
\node at (2,5) [right]{${\scriptstyle n}$};
\node at (2,-5) [right]{${\scriptstyle n}$};
}\ }
=
\sum_{i=0}^{\min\{m,n\}}
(-1)^i
\frac{{m\brack i}{n\brack i}}{{m+n+1\brack i}}
\mathord{\ \tikz[baseline=-.6ex, scale=.1]{
\draw[->-=.5] (-4,5) -- (-4,8);
\draw[-<-=.5] (4,5) -- (4,8);
\draw[-<-=.5, yscale=-1] (-4,5) -- (-4,8);
\draw[->-=.5, yscale=-1] (4,5) -- (4,8);
\draw[->-=.5] (-5,-4) -- (-5,4);
\draw[-<-=.5] (5,-4) -- (5,4);
\draw[-<-=.5] (-3,4) to[out=south, in=south] (3,4);
\draw[->-=.5, yscale=-1] (-3,4) to[out=south, in=south] (3,4);
\draw[fill=white] (-6,4) rectangle (-2,5);
\draw[fill=white, xscale=-1] (-6,4) rectangle (-2,5);
\draw[fill=white, yscale=-1] (-6,4) rectangle (-2,5);
\draw[fill=white, xscale=-1, yscale=-1] (-6,4) rectangle (-2,5);
\node at (-5,0) [left]{${\scriptstyle m-i}$};
\node at (5,0) [right]{${\scriptstyle n-i}$};
\node at (0,5) {${\scriptstyle i}$};
\node at (0,-5) {${\scriptstyle i}$};
\node at (-4,7) [left]{${\scriptstyle m}$};
\node at (4,7) [right]{${\scriptstyle n}$};
\node at (-4,-7) [left]{${\scriptstyle m}$};
\node at (4,-7) [right]{${\scriptstyle n}$};
}\ }.\]
\end{DEF}

The $A_2$ clasp has the following well-known properties
\begin{PROP}
Let $m$, $n$ are non-negative integers.
\begin{enumerate}
\item $\left(\JW{k}{l}\right)\left(\JW{m}{n}\right)=\left(\JW{m}{n}\right)\left(\JW{k}{l}\right)=\JW{m}{n}$ for $0\leq k\leq m$ and $0\leq l\leq n$.
\item $F\left(\JW{m}{n}\right)=0$ if $F\in\{t_{--}^{+}, t_{++}^{-}, d_{+-}, d_{-+}\}$.
\item $\left(\JW{m}{n}\right)F=0$ if $F\in\{t_{-}^{++}, t_{+}^{--}, b^{-+}, b^{+-}\}$.
\end{enumerate}
In the above, we omit tensor components of identity morphisms in $F$ and $\JW{k}{l}$.
\end{PROP}

Next, we construct a special web ${JW}_{\alpha}^{\beta}\in\TL(\alpha,\beta)$  which has a similar properties to the $A_2$ clasp where $\alpha$ and $\beta$ are obtained by rearranging the order of ${+}^{m}{-}^{n}$.

Let us introduce two basis webs called {\em $H$-webs}:
\[
H_{{+}{-}}^{{-}{+}}
=
\mathord{\ \tikz[baseline=-.6ex, scale=.1]{
\draw[->-=.3,-<-=.8] (0,-4) -- (0,4);
\draw[-<-=.3,->-=.8] (5,-4) -- (5,4);
\draw[-<-=.5] (0,0) -- (5,0);
}\ }
,
\quad H_{{-}{+}}^{{+}{-}}
=
\mathord{\ \tikz[baseline=-.6ex, scale=.1]{
\draw[-<-=.3,->-=.8] (0,-4) -- (0,4);
\draw[->-=.3,-<-=.8] (5,-4) -- (5,4);
\draw[->-=.5] (0,0) -- (5,0);
}\ }.
\]

Let $\alpha\in\TL$ be a object obtained by rearranging the order of ${+}^{m}{-}^{n}$.
Then, one can construct a morphism from ${+}^{m}{-}^{n}$ to $\alpha$ by composing $H$-webs.
$\sigma(\alpha)$ denotes such a web minimizing the number of $H$-webs and it is uniquely determined.
In the same way, 
one can construct the morphism from $\alpha$ to ${+}^{m}{-}^{n}$ by pre-composing $H$-webs and denote it by $\tau(\alpha)$.
We remark that $\sigma(\alpha)$ (resp. $\tau(\alpha)$) does not contain $H_{-+}^{+-}$ (resp.$H_{+-}^{-+}$).
\begin{DEF}
Let $\alpha$ and $\beta$ be words obtained by rearranging ${+}^{m}{-}^{n}$. 
We define the following webs:
\begin{align*}
\JW{m}{n}^{\alpha}&=\sigma(\alpha)\left(\JW{m}{n}\right)\in\TL({+}^{m}{-}^{n},\alpha), \\
{JW}_{\beta}^{{+}^{m}{-}^{n}}&=\left(\JW{m}{n}\right)\tau(\beta)\in\TL(\beta,{+}^{m}{-}^{-}), \\
{JW}_{\beta}^{\alpha}&=\JW{m}{n}^{\alpha}{JW}_{\beta}^{{+}^{m}{-}^{n}}\in\TL(\beta,\alpha).
\end{align*}
Let us describe ${JW}_{\beta}^{\alpha}$ diagrammatically by $\mathrm{id}_{\beta}\circ(\text{the white box})\circ\mathrm{id}_{\alpha}$.
\end{DEF}
\begin{LEM}\label{vanish}
$F\left(\JW{m}{n}^{\alpha}\right)=0$ if $F\in\{t_{--}^{+}, t_{++}^{-}, d_{+-}, d_{-+}\}$ and $\left({JW}_{\beta}^{{+}^{m}{-}^{n}}\right)F=0$ if $F\in\{t_{-}^{++}, t_{+}^{--}, b^{-+}, b^{+-}\}$.
\end{LEM}
\begin{proof}
We only show the first case and denote $\sigma(\alpha)$ by $\sigma$ for simplicity
We give a proof by induction on the number $h(\sigma)$ of $H_{+-}^{-+}$ contained in $\sigma$.
If $h(\sigma)=0$, it is clear since $\JW{m}{n}^{\alpha}=\JW{m}{n}$. 
If $h(\sigma)=1$, then $\sigma=\mathrm{id}_{{+}^{m-1}}\otimes H_{+-}^{-+}\otimes\mathrm{id}_{{-}^{n-1}}$ and one can prove easily.
When $h(\sigma)=k+1$, 
we choose $\alpha'\in\TL$ such that $\sigma(\alpha)=\left(\mathrm{id}\otimes H_{+-}^{-+}\otimes\mathrm{id}\right)\sigma(\alpha')$.
Then,
$\sigma$ is one of the following webs:
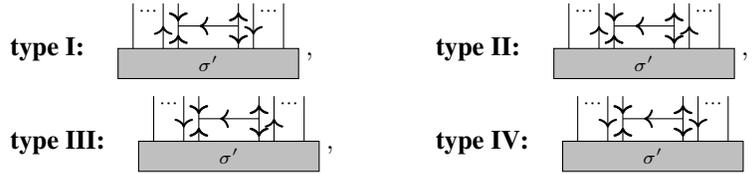
\begin{align*}
&\text{\bf type~{I}:}\quad
\mathord{\ \tikz[baseline=-.6ex, scale=.1]{
\draw (10,0) -- (10,6);
\draw (-10,0) -- (-10,6);
\draw[-<-=.5] (6,0) -- (6,6);
\draw[->-=.5] (-6,0) -- (-6,6);
\draw[->-=.3,-<-=.8] (-4,0) -- (-4,6);
\draw[-<-=.3,->-=.8] (4,0) -- (4,6);
\draw[-<-=.5] (-4,3) -- (4,3);
\draw[fill=lightgray] (-12,0) rectangle (12,-4);
\node at (8,5) {$\scriptstyle{\cdots}$};
\node at (-8,5) {$\scriptstyle{\cdots}$};
\node at (0,-2) {$\scriptstyle{\sigma'}$};
}\ },
&&
\text{\bf type~{II}:}\quad
\mathord{\ \tikz[baseline=-.6ex, scale=.1]{
\draw (10,0) -- (10,6);
\draw (-10,0) -- (-10,6);
\draw[->-=.5] (6,0) -- (6,6);
\draw[->-=.5] (-6,0) -- (-6,6);
\draw[->-=.3,-<-=.8] (-4,0) -- (-4,6);
\draw[-<-=.3,->-=.8] (4,0) -- (4,6);
\draw[-<-=.5] (-4,3) -- (4,3);
\draw[fill=lightgray] (-12,0) rectangle (12,-4);
\node at (8,5) {$\scriptstyle{\cdots}$};
\node at (-8,5) {$\scriptstyle{\cdots}$};
\node at (0,-2) {$\scriptstyle{\sigma'}$};
}\ },\\
&
\text{\bf type~{III}:}\quad
\mathord{\ \tikz[baseline=-.6ex, scale=.1]{
\draw (10,0) -- (10,6);
\draw (-10,0) -- (-10,6);
\draw[->-=.5] (6,0) -- (6,6);
\draw[-<-=.5] (-6,0) -- (-6,6);
\draw[->-=.3,-<-=.8] (-4,0) -- (-4,6);
\draw[-<-=.3,->-=.8] (4,0) -- (4,6);
\draw[-<-=.5] (-4,3) -- (4,3);
\draw[fill=lightgray] (-12,0) rectangle (12,-4);
\node at (8,5) {$\scriptstyle{\cdots}$};
\node at (-8,5) {$\scriptstyle{\cdots}$};
\node at (0,-2) {$\scriptstyle{\sigma'}$};
}\ },
&&
\text{\bf type~{IV}:}\quad
\mathord{\ \tikz[baseline=-.6ex, scale=.1]{
\draw (10,0) -- (10,6);
\draw (-10,0) -- (-10,6);
\draw[-<-=.5] (6,0) -- (6,6);
\draw[-<-=.5] (-6,0) -- (-6,6);
\draw[->-=.3,-<-=.8] (-4,0) -- (-4,6);
\draw[-<-=.3,->-=.8] (4,0) -- (4,6);
\draw[-<-=.5] (-4,3) -- (4,3);
\draw[fill=lightgray] (-12,0) rectangle (12,-4);
\node at (8,5) {$\scriptstyle{\cdots}$};
\node at (-8,5) {$\scriptstyle{\cdots}$};
\node at (0,-2) {$\scriptstyle{\sigma'}$};
}\ },
\end{align*}
where $\sigma'=\sigma(\alpha')$.
Because of $h(\sigma')=k$, 
$\JW{m}{n}^{\alpha'}$ satisfies the statement in Lemma~\ref{vanish} by the induction hypothesis.
We only need to calculate a web obtained by multiplication of $F$ at the part we pictured in the above.

In the case of type~{I}, the possibility of $F$ is $d_{+-}$ or $d_{-+}$. 
Then, the vanishing of $F\left(\sigma\JW{m}{n}\right)$ is clear.

In the case of type~{II}, 
it is similar to type~{I} if $F=d_{+-}$ or $d_{-+}$.
We only have to consider $F=t_{++}^{-}$, that is, show
\[
\mathord{\ \tikz[baseline=-.6ex, scale=.1]{
\draw (4,6) -- (6,6);
\draw[-<-=.7] (5,6) -- (5,9);
\draw (10,0) -- (10,6);
\draw (-10,0) -- (-10,6);
\draw[->-=.5] (6,0) -- (6,6);
\draw[->-=.5] (-6,0) -- (-6,6);
\draw[->-=.3,-<-=.8] (-4,0) -- (-4,6);
\draw[-<-=.3,->-=.8] (4,0) -- (4,6);
\draw[-<-=.5] (-4,3) -- (4,3);
\draw[fill=lightgray] (-12,0) rectangle (12,-4);
\node at (8,5) {$\scriptstyle{\cdots}$};
\node at (-8,5) {$\scriptstyle{\cdots}$};
\node at (0,-2) {$\scriptstyle{\sigma'}$};
}\ }\circ\JW{m}{n}=0.
\]
The construction of $\sigma$ and $\alpha'=\cdots{{+}{+}{-}{+}}\cdots$ say that there exists $\sigma''$ such that
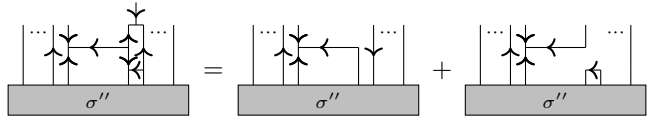
\[
\mathord{\ \tikz[baseline=-.6ex, scale=.1]{
\draw (4,6) -- (6,6);
\draw[-<-=.7] (5,6) -- (5,9);
\draw (10,0) -- (10,6);
\draw (-10,0) -- (-10,6);
\draw[->-=.5] (6,0) -- (6,6);
\draw[->-=.5] (-6,0) -- (-6,6);
\draw[->-=.3,-<-=.8] (-4,0) -- (-4,6);
\draw[-<-=.3,->-=.8] (4,0) -- (4,6);
\draw[-<-=.5] (-4,3) -- (4,3);
\draw[fill=lightgray] (-12,-2) rectangle (12,-6);
\draw (10,0) -- (10,-2);
\draw (-10,0) -- (-10,-2);
\draw (6,0) -- (6,-2);
\draw (-6,0) -- (-6,-2);
\draw (-4,0) -- (-4,-2);
\draw (4,0) -- (4,-2);
\draw[-<-=.8] (4,0) -- (6,0);
\node at (8,5) {$\scriptstyle{\cdots}$};
\node at (-8,5) {$\scriptstyle{\cdots}$};
\node at (0,-4) {$\scriptstyle{\sigma''}$};
}\ }
=
\mathord{\ \tikz[baseline=-.6ex, scale=.1]{
\draw (10,0) -- (10,6);
\draw (-10,0) -- (-10,6);
\draw[-<-=.5] (6,0) -- (6,6);
\draw[->-=.5] (-6,0) -- (-6,6);
\draw[->-=.3,-<-=.8] (-4,0) -- (-4,6);
\draw (4,0) -- (4,3);
\draw[-<-=.5] (-4,3) -- (4,3);
\draw[fill=lightgray] (-12,-2) rectangle (12,-6);
\draw (10,0) -- (10,-2);
\draw (-10,0) -- (-10,-2);
\draw (6,0) -- (6,-2);
\draw (-6,0) -- (-6,-2);
\draw (-4,0) -- (-4,-2);
\draw (4,0) -- (4,-2);
\node at (8,5) {$\scriptstyle{\cdots}$};
\node at (-8,5) {$\scriptstyle{\cdots}$};
\node at (0,-4) {$\scriptstyle{\sigma''}$};
}\ }
+
\mathord{\ \tikz[baseline=-.6ex, scale=.1]{
\draw (10,0) -- (10,6);
\draw (-10,0) -- (-10,6);
\draw[->-=.5] (-6,0) -- (-6,6);
\draw[->-=.3,-<-=.8] (-4,0) -- (-4,6);
\draw (4,3) -- (4,6);
\draw[-<-=.5] (-4,3) -- (4,3);
\draw[fill=lightgray] (-12,-2) rectangle (12,-6);
\draw (10,0) -- (10,-2);
\draw (-10,0) -- (-10,-2);
\draw (6,0) -- (6,-2);
\draw (-6,0) -- (-6,-2);
\draw (-4,0) -- (-4,-2);
\draw (4,0) -- (4,-2);
\draw[-<-=.8] (4,0) -- (6,0);
\node at (8,5) {$\scriptstyle{\cdots}$};
\node at (-8,5) {$\scriptstyle{\cdots}$};
\node at (0,-4) {$\scriptstyle{\sigma''}$};
}\ }
\]
The first term of the RHS has $t_{{+}{+}}^{-}$ and the second $d_{{-}{+}}$ on the top of $\sigma''$. 
Because of $h(\sigma'')=k-1$ and the induction hypothesis, 
the composition of the RHS with $\JW{m}{n}$ vanish.
One can prove type~{III} and type~{IV} in the same way.
\end{proof}
\begin{LEM}
Let $\alpha$, $\beta$, and $\gamma$ in $\TL$ be obtained by rearranging ${+}^{m}{-}^{n}$.
Then, ${JW}_{\beta}^{\gamma}{JW}_{\alpha}^{\beta}={JW}_{\alpha}^{\gamma}$. 
\end{LEM}
\begin{proof}
It is easy to see by Lemma~\ref{vanish} because $\tau(\beta)$ is the reflection of $\sigma(\beta)$ at the horizontal line with the oppositely directed edges.
\end{proof}

Let us introduce two types of a web, ``stair-step'' and ``triangle''.
These webs also appear in \cite{Kim06, Kim07, Yuasa17, FrohmanSikora20A}.
\begin{DEF}
For positive integers $n$ and $m$, 
$
\mathord{\ \tikz[baseline=-.6ex, scale=.1]{
\draw (-4,0) -- (-2,0);
\draw (2,0) -- (4,0);
\draw (0,-4) -- (0,-2);
\draw (0,2) -- (0,4);
\draw[fill=white] (-2,-2) rectangle (2,2);
\draw (-2,2) -- (2,-2);
\node at (-4,0) [left]{$\scriptstyle{n}$};
\node at (4,0) [right]{$\scriptstyle{n}$};
\node at (0,4) [above]{$\scriptstyle{m}$};
\node at (0,-4) [below]{$\scriptstyle{m}$};
}\ }
$
is defined by
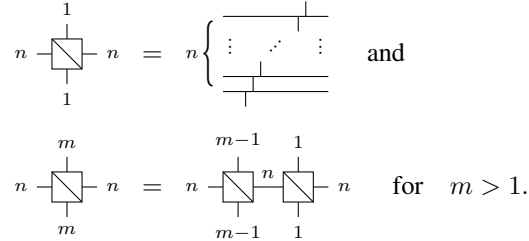
\begin{align*}
&\mathord{\ \tikz[baseline=-.6ex, scale=.1]{
\draw (-4,0) -- (-2,0);
\draw (2,0) -- (4,0);
\draw (0,-4) -- (0,-2);
\draw (0,2) -- (0,4);
\draw[fill=white] (-2,-2) rectangle (2,2);
\draw (-2,2) -- (2,-2);
\node at (-4,0) [left]{$\scriptstyle{n}$};
\node at (4,0) [right]{$\scriptstyle{n}$};
\node at (0,4) [above]{$\scriptstyle{1}$};
\node at (0,-4) [below]{$\scriptstyle{1}$};
}\ }
=
\mathord{\ \tikz[baseline=-.6ex, scale=.1]{
\draw (-7,-5) -- (7,-5);
\draw (-7,-3) -- (7,-3);
\draw (-7,5) -- (7,5);
\draw (4,5) -- (4,7);
\draw (3,3) -- (3,5);
\draw (-4,-7) -- (-4,-5);
\draw (-3,-5) -- (-3,-3);
\draw (-2,-3) -- (-2,-1);
\node[yscale=2.5] at (-9,0) {${\{}$};
\node at (-9,0) [left]{$\scriptstyle{n}$};
\node[rotate=90] at (-6,1){$\scriptstyle{\cdots}$};
\node[rotate=90] at (6,1){$\scriptstyle{\cdots}$};
\node[rotate=45] at (0,1){$\scriptstyle{\cdots}$};
}\ }\quad\text{and}\\
&\mathord{\ \tikz[baseline=-.6ex, scale=.1]{
\draw (-4,0) -- (-2,0);
\draw (2,0) -- (4,0);
\draw (0,-4) -- (0,-2);
\draw (0,2) -- (0,4);
\draw[fill=white] (-2,-2) rectangle (2,2);
\draw (-2,2) -- (2,-2);
\node at (-4,0) [left]{$\scriptstyle{n}$};
\node at (4,0) [right]{$\scriptstyle{n}$};
\node at (0,4) [above]{$\scriptstyle{m}$};
\node at (0,-4) [below]{$\scriptstyle{m}$};
}\ }
=
\mathord{\ \tikz[baseline=-.6ex, scale=.1]{
\draw (-4,0) -- (-2,0);
\draw (2,0) -- (4,0);
\draw (0,-4) -- (0,-2);
\draw (0,2) -- (0,4);
\draw[fill=white] (-2,-2) rectangle (2,2);
\draw (-2,2) -- (2,-2);
\node at (-4,0) [left]{$\scriptstyle{n}$};
\node at (0,4) [above]{$\scriptstyle{m-1}$};
\node at (0,-4) [below]{$\scriptstyle{m-1}$};
\begin{scope}[xshift=8cm]
\draw (-4,0) -- (-2,0);
\draw (2,0) -- (4,0);
\draw (0,-4) -- (0,-2);
\draw (0,2) -- (0,4);
\draw[fill=white] (-2,-2) rectangle (2,2);
\draw (-2,2) -- (2,-2);
\node at (-4,0) [above]{$\scriptstyle{n}$};
\node at (4,0) [right]{$\scriptstyle{n}$};
\node at (0,4) [above]{$\scriptstyle{1}$};
\node at (0,-4) [below]{$\scriptstyle{1}$};
\end{scope}
}\ }
\quad\text{for}\quad m>1.
\end{align*}
Specifying a direction on an edge around the box determine the all directions of edges in the box.
\end{DEF}

\begin{DEF}
For positive integer $n$,
$
\mathord{\ \tikz[baseline=-.6ex, scale=.1]{
\draw (30:5) -- (0,0);
\draw (150:5) -- (0,0);
\draw (270:5) -- (0,0);
\draw[fill=white] (-30:4) -- (90:4) -- (210:4) -- cycle;
\node at (30:5) [below]{$\scriptstyle{n}$};
\node at (150:5) [below]{$\scriptstyle{n}$};
\node at (270:5) [left]{$\scriptstyle{n}$};
}\ }
$
is defined by
$\mathord{\ \tikz[baseline=-.6ex, scale=.1]{
\draw (30:5) -- (0,0);
\draw (150:5) -- (0,0);
\draw (270:5) -- (0,0);
\draw[fill=white] (-30:4) -- (90:4) -- (210:4) -- cycle;
\node at (30:5) [below]{$\scriptstyle{1}$};
\node at (150:5) [below]{$\scriptstyle{1}$};
\node at (270:5) [left]{$\scriptstyle{1}$};
}\ }
=\mathord{\ \tikz[baseline=-.6ex, scale=.1]{
\draw (30:5) -- (0,0);
\draw (150:5) -- (0,0);
\draw (270:5) -- (0,0);
}\ }$ and 
$\mathord{\ \tikz[baseline=-.6ex, scale=.1]{
\draw (30:5) -- (0,0);
\draw (150:5) -- (0,0);
\draw (270:5) -- (0,0);
\draw[fill=white] (-30:4) -- (90:4) -- (210:4) -- cycle;
\node at (30:5) [below]{$\scriptstyle{n}$};
\node at (150:5) [below]{$\scriptstyle{n}$};
\node at (270:5) [left]{$\scriptstyle{n}$};
}\ }
=\mathord{\ \tikz[baseline=-.6ex, scale=.1]{
\draw (-10,5) -- (5,5);
\draw (-5,-4) -- (-5,5);
\draw (-10,1) -- (5,1);
\draw (0,-4) -- (0,1);
\draw[fill=white] (-30:3) -- (90:3) -- (210:3) -- cycle;
\draw[fill=white] (-6,-2) rectangle (-4,3);
\draw (-6,3) -- (-4,-2);
\node at (-10,5) [left]{$\scriptstyle{1}$};
\node at (5,5) [right]{$\scriptstyle{1}$};
\node at (-5,-4) [below]{$\scriptstyle{1}$};
\node at (-10,0) [below]{$\scriptstyle{n-1}$};
\node at (4,0) [above]{$\scriptstyle{n-1}$};
\node at (0,-4) [below]{$\scriptstyle{n-1}$};
}\ }
$
for $n>1$.
We obtain a web by specifying a direction of an edge around the triangle.
\end{DEF}

We note a graphical description of properties of $A_{2}$ clasp and useful formulas.
\begin{LEM}\label{claspformula}
\begin{align*}
&\mathord{\ \tikz[baseline=-.6ex, scale=.1]{
\draw (0,-6) -- (0,6);
\draw[fill=white] (-4,-4) rectangle (4,-2);
\draw[fill=white] (-4,4) rectangle (4,2);
\node at (0,5) [left]{${\scriptstyle \gamma}$};
\node at (0,0) [left]{${\scriptstyle \beta}$};
\node at (0,-5) [left]{${\scriptstyle \alpha}$};
}\ }
=
\mathord{\ \tikz[baseline=-.6ex, scale=.1]{
\draw (0,-6) -- (0,6);
\draw[fill=white] (-4,-1) rectangle (4,1);
\node at (0,4) [left]{${\scriptstyle \gamma}$};
\node at (0,-4) [left]{${\scriptstyle \alpha}$};
}\ }
,
\quad\mathord{\ \tikz[baseline=-.6ex, scale=.1]{
\draw (-2,1) -- (0,3);
\draw (2,1) -- (0,3);
\draw (0,6) -- (0,3);
\draw (0,-6) -- (0,0);
\draw[fill=white] (-6,-1) rectangle (6,1);
\node at (4,3) {${\scriptstyle {\dots}}$};
\node at (-4,3) {${\scriptstyle {\dots}}$};
}\ }=0
,
\quad\mathord{\ \tikz[baseline=-.6ex, scale=.1]{
\draw (-2,1) to[out=north, in=west] (0,3) to[out=east, in=north] (2,1);
\draw (0,-6) -- (0,0);
\draw[fill=white] (-6,-1) rectangle (6,1);
\node at (4,3) {${\scriptstyle {\dots}}$};
\node at (-4,3) {${\scriptstyle {\dots}}$};
}\ }=0
,\\
&\mathord{\ \tikz[baseline=-.6ex, scale=.1]{
\draw[-<-=.5] (-3,1) to[out=north, in=south west] (0,4);
\draw[->-=.5] (3,1) to[out=north, in=south east] (0,4);
\draw[-<-=.8] (0,4) to[out=north east, in=south] (3,7);
\draw[->-=.8] (0,4) to[out=north west, in=south] (-3,7);
\draw (0,-5) -- (0,0);
\draw[fill=white] (0,2) -- (-2,4) -- (0,6) -- (2,4) -- cycle;
\draw (0,2) -- (0,6);
\draw[fill=white] (-6,-1) rectangle (6,1);
\node at (5,4) {${\scriptstyle {\dots}}$};
\node at (-5,4) {${\scriptstyle {\dots}}$};
\node at (-3,7) [above]{${\scriptstyle m}$};
\node at (3,7) [above]{${\scriptstyle n}$};
}\ }
=
\mathord{\ \tikz[baseline=-.6ex, scale=.1]{
\draw[->-=.5] (-2,1) -- (-2,7);
\draw[-<-=.5] (2,1) -- (2,7);
\draw (0,-5) -- (0,0);
\draw[fill=white] (-6,-1) rectangle (6,1);
\node at (5,4) {${\scriptstyle {\dots}}$};
\node at (-5,4) {${\scriptstyle {\dots}}$};
\node at (-2,7) [above]{${\scriptstyle m}$};
\node at (2,7) [above]{${\scriptstyle n}$};
}\ }
,
\quad\mathord{\ \tikz[baseline=-.6ex, scale=.1]{
\draw[->-=.5] (-3,1) to[out=north, in=south west] (0,4);
\draw[-<-=.5] (3,1) to[out=north, in=south east] (0,4);
\draw[->-=.8] (0,4) to[out=north east, in=south] (3,7);
\draw[-<-=.8] (0,4) to[out=north west, in=south] (-3,7);
\draw (0,-5) -- (0,0);
\draw[fill=white] (0,2) -- (-2,4) -- (0,6) -- (2,4) -- cycle;
\draw (0,2) -- (0,6);
\draw[fill=white] (-6,-1) rectangle (6,1);
\node at (5,4) {${\scriptstyle {\dots}}$};
\node at (-5,4) {${\scriptstyle {\dots}}$};
\node at (-3,7) [above]{${\scriptstyle m}$};
\node at (3,7) [above]{${\scriptstyle n}$};
}\ }
=
\mathord{\ \tikz[baseline=-.6ex, scale=.1]{
\draw[-<-=.5] (-2,1) -- (-2,7);
\draw[->-=.5] (2,1) -- (2,7);
\draw (0,-5) -- (0,0);
\draw[fill=white] (-6,-1) rectangle (6,1);
\node at (5,4) {${\scriptstyle {\dots}}$};
\node at (-5,4) {${\scriptstyle {\dots}}$};
\node at (-2,7) [above]{${\scriptstyle m}$};
\node at (2,7) [above]{${\scriptstyle n}$};
}\ }
,\\
&\mathord{\ \tikz[baseline=-.6ex, scale=.1]{
\draw[-<-=.8, white, double=black, double distance=0.4pt, ultra thick] (-3,1) to[out=north, in=south] (3,7);
\draw[->-=.8, white, double=black, double distance=0.4pt, ultra thick] (3,1) to[out=north, in=south] (-3,7);
\draw (0,-5) -- (0,0);
\draw[fill=white] (-6,-1) rectangle (6,1);
\node at (5,4) {${\scriptstyle {\dots}}$};
\node at (-5,4) {${\scriptstyle {\dots}}$};
\node at (-3,7) [above]{${\scriptstyle m}$};
\node at (3,7) [above]{${\scriptstyle n}$};
}\ }
=(-q^{-\frac{1}{6}})^{mn}
\mathord{\ \tikz[baseline=-.6ex, scale=.1]{
\draw[->-=.5] (-2,1) -- (-2,7);
\draw[-<-=.5] (2,1) -- (2,7);
\draw (0,-5) -- (0,0);
\draw[fill=white] (-6,-1) rectangle (6,1);
\node at (5,4) {${\scriptstyle {\dots}}$};
\node at (-5,4) {${\scriptstyle {\dots}}$};
\node at (-2,7) [above]{${\scriptstyle m}$};
\node at (2,7) [above]{${\scriptstyle n}$};
}\ }
,
\quad\mathord{\ \tikz[baseline=-.6ex, scale=.1]{
\draw[->-=.8, white, double=black, double distance=0.4pt, ultra thick] (3,1) to[out=north, in=south] (-3,7);
\draw[-<-=.8, white, double=black, double distance=0.4pt, ultra thick] (-3,1) to[out=north, in=south] (3,7);
\draw (0,-5) -- (0,0);
\draw[fill=white] (-6,-1) rectangle (6,1);
\node at (5,4) {${\scriptstyle {\dots}}$};
\node at (-5,4) {${\scriptstyle {\dots}}$};
\node at (-3,7) [above]{${\scriptstyle m}$};
\node at (3,7) [above]{${\scriptstyle n}$};
}\ }
=(-q^{\frac{1}{6}})^{mn}
\mathord{\ \tikz[baseline=-.6ex, scale=.1]{
\draw[->-=.5] (-2,1) -- (-2,7);
\draw[-<-=.5] (2,1) -- (2,7);
\draw (0,-5) -- (0,0);
\draw[fill=white] (-6,-1) rectangle (6,1);
\node at (5,4) {${\scriptstyle {\dots}}$};
\node at (-5,4) {${\scriptstyle {\dots}}$};
\node at (-2,7) [above]{${\scriptstyle m}$};
\node at (2,7) [above]{${\scriptstyle n}$};
}\ }
,\\
&\mathord{\ \tikz[baseline=-.6ex, scale=.1]{
\draw[->-=.8, white, double=black, double distance=0.4pt, ultra thick] (-3,1) to[out=north, in=south] (3,7);
\draw[->-=.8, white, double=black, double distance=0.4pt, ultra thick] (3,1) to[out=north, in=south] (-3,7);
\draw (0,-5) -- (0,0);
\draw[fill=white] (-6,-1) rectangle (6,1);
\node at (5,4) {${\scriptstyle {\dots}}$};
\node at (-5,4) {${\scriptstyle {\dots}}$};
\node at (-3,7) [above]{${\scriptstyle m}$};
\node at (3,7) [above]{${\scriptstyle n}$};
}\ }
=q^{\frac{mn}{3}}
\mathord{\ \tikz[baseline=-.6ex, scale=.1]{
\draw[->-=.5] (-2,1) -- (-2,7);
\draw[->-=.5] (2,1) -- (2,7);
\draw (0,-5) -- (0,0);
\draw[fill=white] (-6,-1) rectangle (6,1);
\node at (5,4) {${\scriptstyle {\dots}}$};
\node at (-5,4) {${\scriptstyle {\dots}}$};
\node at (-2,7) [above]{${\scriptstyle m}$};
\node at (2,7) [above]{${\scriptstyle n}$};
}\ }
,
\quad\mathord{\ \tikz[baseline=-.6ex, scale=.1]{
\draw[->-=.8, white, double=black, double distance=0.4pt, ultra thick] (3,1) to[out=north, in=south] (-3,7);
\draw[->-=.8, white, double=black, double distance=0.4pt, ultra thick] (-3,1) to[out=north, in=south] (3,7);
\draw (0,-5) -- (0,0);
\draw[fill=white] (-6,-1) rectangle (6,1);
\node at (5,4) {${\scriptstyle {\dots}}$};
\node at (-5,4) {${\scriptstyle {\dots}}$};
\node at (-3,7) [above]{${\scriptstyle m}$};
\node at (3,7) [above]{${\scriptstyle n}$};
}\ }
=q^{-\frac{mn}{3}}
\mathord{\ \tikz[baseline=-.6ex, scale=.1]{
\draw[->-=.5] (-2,1) -- (-2,7);
\draw[->-=.5] (2,1) -- (2,7);
\draw (0,-5) -- (0,0);
\draw[fill=white] (-6,-1) rectangle (6,1);
\node at (5,4) {${\scriptstyle {\dots}}$};
\node at (-5,4) {${\scriptstyle {\dots}}$};
\node at (-2,7) [above]{${\scriptstyle m}$};
\node at (2,7) [above]{${\scriptstyle n}$};
}\ },\\
&\mathord{\ \tikz[baseline=-.6ex, scale=.1]{
\draw (-8,0) -- (8,0);
\draw (0,-7) -- (0,0);
\draw[fill=white] (-4,-3) -- (0,4) -- (4,-3) -- cycle;
\draw[fill=white] (-3,-5) rectangle (3,-4);
\draw[fill=white] (-6,-3) rectangle (-5,3);
\draw[fill=white] (6,-3) rectangle (5,3);
\node at (-8,0) [above]{${\scriptstyle n}$};
\node at (8,0) [above]{${\scriptstyle n}$};
\node at (0,-7) [right]{${\scriptstyle n}$};
}\ }
=\mathord{\ \tikz[baseline=-.6ex, scale=.1]{
\draw (-8,0) -- (6,0);
\draw (0,-7) -- (0,0);
\draw[fill=white] (-4,-3) -- (0,4) -- (4,-3) -- cycle;
\draw[fill=white] (-3,-5) rectangle (3,-4);
\draw[fill=white] (-6,-3) rectangle (-5,3);
\node at (-8,0) [above]{${\scriptstyle n}$};
\node at (6,0) [above]{${\scriptstyle n}$};
\node at (0,-7) [right]{${\scriptstyle n}$};
}\ }
\end{align*}
\end{LEM}
\begin{proof}
Graphical interpretation of the previous definitions and lemmas. 
One can prove formulas related to the braiding and the triangle by easy calculation, see \cite{Yuasa17}.
\end{proof}

\section{Twist formula}
In this section, we show new twist formulas derived from a combinatorial method related to the generating function of Young diagrams. 
We firstly explain this method which is used in the proof of full twist formulas for the Kauffman bracket and Theorem~{\ref{antifull}} in \cite{Yuasa17}.

Let $\{\sigma_{n}(k,l)\mid 0\leq k,l\leq n\}$ be a set of webs satisfying the following conditions:
\begin{enumerate}
\item $\sigma_{n}(k,l)$ is basis web if $k+l=n$,
\item $\sigma_{n}(k,l)=X(k,l)\sigma_{n}(k+1,l)+Y(k,l)\sigma_{n}(k,l+1)$,
\item $X(k,l)Y(k+1,l)=qY(k,l)X(k,l+1)$,
\end{enumerate}
where $X(k,l), Y(k,l)\in\mathbb{C}(q^{\frac{1}{6}})$.

\begin{PROP}\label{partition}
\[
 \sigma_{n}(0,0)=\sum_{k+l=n}\prod_{j=0}^{l-1}Y(0,j)\prod_{i=0}^{k-1}X(i,l)\binom{n}{k}_q\sigma_{n}(k,l).
\]\
\end{PROP}

\begin{proof}
Let us label a lattice point $(k,l)$ in $\mathbb{Z}\times\mathbb{Z}$ with $\sigma(k,l)$, and a directed edge from $(k,l)$ to $(k+1,l)$ with $X(k,l)$, and a directed edge from $(k,l)$ to $(k,l+1)$ with $Y(k,l)$.
We denote the label of an edge $e$ by $w(e)$.
In the expansion of $\sigma_{n}(0,0)$, it is seen that the coefficient of $\sigma_{n}(k,l)$, for $k+l=n$, is 
\begin{align}\label{qcalc}
 \sum_{\gamma}\left(\prod_{e\in\gamma}w(e)\right)=\sum_{a_i\in\{X,Y\}} a_{1}a_{2}\cdots a_{n},
\end{align}
where $X$ (resp. $Y$) appears $k$ (resp. $l$) times in $a_1a_2\cdots a_n$ and $\gamma$ is a path from $(0,0)$ to $(k,l)$ which consists of the above directed edges.
Although we describe $X$ and $Y$ without its coordinate in the above, it is uniquely determined. 
For example, $a_{1}a_{2}a_{3}a_{4}\cdots=XYYX\cdots=X(0,0)Y(1,0)Y(1,1)X(1,2)\cdots$. 
By the condition~(3), the RHS of (\ref{qcalc}) is 
\[
 Y^{l}X^{k}\sum q^{\#\{i<j\mid a_{i}=X, a_{j}=Y\}}=Y^{l}X^{k}\binom{n}{k}_q.
\]
\end{proof}
In \cite{Yuasa17}, we used Young diagrams instead of sequences of $X$ and $Y$. 
However, it is intrinsically the same because a lattice path from $(0,0)$ to $(k,l)$ is interpreted as a Young diagram.
The formula of Theorem~\ref{antifull} is obtained from Proposition~\ref{partition} by setting $n=d=\min\{s,t\}$ and 
\[
\sigma_{n}(k,l)=
\mathord{ \tikz[baseline=-.6ex, scale=.1]{
\draw (-6,5) -- (-9,5);
\draw (-6,-5) -- (-9,-5);
\draw (6,5) -- (9,5);
\draw (6,-5) -- (9,-5);
\draw[->-=.9, white, double=black, double distance=0.4pt, ultra thick] (-6,-5) to[out=east, in=west] (0,4);
\draw[-<-=.9, white, double=black, double distance=0.4pt, ultra thick] (-6,6) to[out=east, in=west] (0,-4);
\draw[white, double=black, double distance=0.4pt, ultra thick] (0,-4) to[out=east, in=west] (6,6);
\draw[white, double=black, double distance=0.4pt, ultra thick] (0,4) to[out=east, in=west] (6,-5);
\draw[->-=.5] (-6,-4) to[out=east, in=south] (-5,0) to[out=north, in=east] (-6,4);
\draw[-<-=.5] (6,-4) to[out=west, in=south] (5,0) to[out=north, in=west] (6,4);
\draw[->-=.5] (-6,-6) -- (6,-6);
\draw[fill=white] (-7,-7) rectangle (-6,-3);
\draw[fill=white] (-7,7) rectangle (-6,3);
\draw[fill=white] (7,-7) rectangle (6,-3);
\draw[fill=white] (7,7) rectangle (6,3);
\node at (-7,-7)[left]{$\scriptstyle{s}$};
\node at (-7,7)[left]{$\scriptstyle{t}$};
\node at (7,-7)[right]{$\scriptstyle{s}$};
\node at (7,7)[right]{$\scriptstyle{t}$};
\node at (0,-6)[below]{$\scriptstyle{k}$};
\node at (5,0)[right]{$\scriptstyle{l}$};
\node at (-5,0)[left]{$\scriptstyle{l}$};
}\ }.
\]

In a similar way, We will prove the following formula.
\begin{THM}[parallel full twist formula]\label{parallelfull}
Let $d=\min\{s,t\}$ and $\delta=\left| s-t \right|$.
\begin{align*}
\mathord{ \tikz[baseline=-.6ex, scale=.1]{
\draw (-6,5) -- (-9,5);
\draw (-6,-5) -- (-9,-5);
\draw (6,5) -- (9,5);
\draw (6,-5) -- (9,-5);
\draw[->-=.9, white, double=black, double distance=0.4pt, ultra thick] (-6,-5) to[out=east, in=west] (0,5);
\draw[->-=.9, white, double=black, double distance=0.4pt, ultra thick] (-6,5) to[out=east, in=west] (0,-5);
\draw[white, double=black, double distance=0.4pt, ultra thick] (0,-5) to[out=east, in=west] (6,5);
\draw[white, double=black, double distance=0.4pt, ultra thick] (0,5) to[out=east, in=west] (6,-5);
\draw[fill=white] (-7,-7) rectangle (-6,-3);
\draw[fill=white] (-7,7) rectangle (-6,3);
\draw[fill=white] (7,-7) rectangle (6,-3);
\draw[fill=white] (7,7) rectangle (6,3);
\node at (-7,-7)[left]{$\scriptstyle{t}$};
\node at (-7,7)[left]{$\scriptstyle{s}$};
\node at (7,-7)[right]{$\scriptstyle{t}$};
\node at (7,7)[right]{$\scriptstyle{s}$};
}\ }
&=q^{\frac{2}{3}st}
\sum_{l=0}^{\infty}
q^{l^2-\frac{l}{2}}q^{-(s+t)l}
(q)_{l}{s \choose l}_{q}{t \choose l}_{q}
\mathord{\ \tikz[baseline=-.6ex, scale=.1]{
\draw (-6,5) -- (-9,5);
\draw (-6,-5) -- (-9,-5);
\draw (6,5) -- (9,5);
\draw (6,-5) -- (9,-5);
\draw[->-=.5] (-6,6) -- (6,6);
\draw[->-=.5] (-6,-6) -- (6,-6);
\draw[->-=.5] (-6,4) to[out=east, in=north west] (0,0);
\draw[->-=.5] (-6,-4) to[out=east, in=south west] (0,0);
\draw[->-=.7] (0,0) to[out=north east, in=west] (6,4);
\draw[->-=.7] (0,0) to[out=south east, in=west] (6,-4);
\draw[fill=white] (-3,0) -- (0,3) -- (3,0) -- (0,-3) -- cycle;
\draw (0,3) -- (0,-3);
\draw[fill=white] (-7,-7) rectangle (-6,-3);
\draw[fill=white] (-7,7) rectangle (-6,3);
\draw[fill=white] (7,-7) rectangle (6,-3);
\draw[fill=white] (7,7) rectangle (6,3);
\node at (-7,-7)[left]{$\scriptstyle{s}$};
\node at (-7,7)[left]{$\scriptstyle{t}$};
\node at (7,-7)[right]{$\scriptstyle{s}$};
\node at (7,7)[right]{$\scriptstyle{t}$};
\node at (0,6)[above]{$\scriptstyle{t-l}$};
\node at (0,-6)[below]{$\scriptstyle{s-l}$};
\node at (-5,4)[below]{$\scriptstyle{l}$};
\node at (-5,-4)[above]{$\scriptstyle{l}$};
\node at (5,4)[below]{$\scriptstyle{l}$};
\node at (5,-4)[above]{$\scriptstyle{l}$};
}\ }\\
&=q^{-\frac{d(d+\delta)}{3}-\frac{d}{2}}
\sum_{k=0}^{d}
q^{k(k+\delta)+\frac{k}{2}}
\frac{(q)_{d+\delta}}{(q)_{k+\delta}}{d \choose k}_{q}
\mathord{\ \tikz[baseline=-.6ex, scale=.1]{
\draw (-6,5) -- (-9,5);
\draw (-6,-5) -- (-9,-5);
\draw (6,5) -- (9,5);
\draw (6,-5) -- (9,-5);
\draw[->-=.5] (-6,6) -- (6,6);
\draw[->-=.5] (-6,-6) -- (6,-6);
\draw[->-=.5] (-6,4) to[out=east, in=north west] (0,0);
\draw[->-=.5] (-6,-4) to[out=east, in=south west] (0,0);
\draw[->-=.7] (0,0) to[out=north east, in=west] (6,4);
\draw[->-=.7] (0,0) to[out=south east, in=west] (6,-4);
\draw[fill=white] (-3,0) -- (0,3) -- (3,0) -- (0,-3) -- cycle;
\draw (0,3) -- (0,-3);
\draw[fill=white] (-7,-7) rectangle (-6,-3);
\draw[fill=white] (-7,7) rectangle (-6,3);
\draw[fill=white] (7,-7) rectangle (6,-3);
\draw[fill=white] (7,7) rectangle (6,3);
\node at (-7,-7)[left]{$\scriptstyle{s}$};
\node at (-7,7)[left]{$\scriptstyle{t}$};
\node at (7,-7)[right]{$\scriptstyle{s}$};
\node at (7,7)[right]{$\scriptstyle{t}$};
\node at (0,6)[above]{$\scriptstyle{k+t-d}$};
\node at (0,-6)[below]{$\scriptstyle{k+s-d}$};
\node at (-3,2)[left]{$\scriptstyle{d-k}$};
\node at (-3,-2)[left]{$\scriptstyle{d-k}$};
\node at (3,2)[right]{$\scriptstyle{d-k}$};
\node at (3,-2)[right]{$\scriptstyle{d-k}$};
}\ }.
\end{align*}
\end{THM}

To prove this formula by Proposition~\ref{partition},
we set 
\begin{align*}
&\sigma_{d}(k,l)=
\mathord{ \tikz[baseline=-.6ex, scale=.1]{
\draw (-17,5) -- (-14,5);
\draw (-17,-5) -- (-14,-5);
\draw (17,5) -- (14,5);
\draw (17,-5) -- (14,-5);
\draw[->-=.5] (-14,6) -- (-6,6);
\draw[->-=.5] (-14,-6) -- (-6,-6);
\draw[->-=.5] (-6,0) -- (-11,0);
\draw[->-=.5] (-14,4) to[out=east, in=north east] (-11,0);
\draw[->-=.5] (-14,-4) to[out=east, in=south east] (-11,0);
\draw[fill=white] (-12,0) -- (-9,2) -- (-9,-2) -- cycle;
\draw[-<-=.5] (14,6) -- (6,6);
\draw[-<-=.5] (14,-6) -- (6,-6);
\draw[-<-=.5] (6,0) -- (11,0);
\draw[-<-=.5] (14,4) to[out=west, in=north west] (11,0);
\draw[-<-=.5] (14,-4) to[out=west, in=south west] (11,0);
\draw[fill=white] (12,0) -- (9,2) -- (9,-2) -- cycle;
\draw[-<-=.5] (-6,-6) -- (6,-6);
\draw[->-=.5] (-6,-4) -- (6,-4);
\draw[->-=.9, white, double=black, double distance=0.4pt, ultra thick] (-6,-1) to[out=east, in=west] (0,6);
\draw[->-=.9, white, double=black, double distance=0.4pt, ultra thick] (-6,6) to[out=east, in=west] (0,-1);
\draw[white, double=black, double distance=0.4pt, ultra thick] (0,-1) to[out=east, in=west] (6,6);
\draw[white, double=black, double distance=0.4pt, ultra thick] (0,6) to[out=east, in=west] (6,-1);
\draw[fill=white] (-7,8) rectangle (-6,4);
\draw[fill=white] (-7,-8) rectangle (-6,2);
\draw[fill=white] (7,8) rectangle (6,4);
\draw[fill=white] (7,-8) rectangle (6,2);
\node at (3,-5)[above]{$\scriptstyle{k}$};
\node at (3,-5)[below]{$\scriptstyle{l}$};
\node at (0,6)[above]{$\scriptstyle{s-k-l}$};
\draw[fill=white] (-14,-8) rectangle (-15,-2);
\draw[fill=white] (-14,8) rectangle (-15,2);
\draw[fill=white] (14,-8) rectangle (15,-2);
\draw[fill=white] (14,8) rectangle (15,2);
\node at (-10,6)[above]{$\scriptstyle{t-l}$};
\node at (-10,-6)[below]{$\scriptstyle{s-l}$};
\node at (10,6)[above]{$\scriptstyle{t-l}$};
\node at (10,-6)[below]{$\scriptstyle{s-l}$};
\node at (-10,1)[above]{$\scriptstyle{l}$};
\node at (-10,-1)[below]{$\scriptstyle{l}$};
\node at (10,1)[above]{$\scriptstyle{l}$};
\node at (10,-1)[below]{$\scriptstyle{l}$};
\node at (-17,5)[above]{$\scriptstyle{t}$};
\node at (-17,-5)[below]{$\scriptstyle{s}$};
\node at (17,5)[above]{$\scriptstyle{t}$};
\node at (17,-5)[below]{$\scriptstyle{s}$};
}\ }&&\text{if $d=s$},\\
&\sigma_{d}(k,l)=
\mathord{ \tikz[baseline=-.6ex, scale=.1]{
\draw (-17,5) -- (-14,5);
\draw (-17,-5) -- (-14,-5);
\draw (17,5) -- (14,5);
\draw (17,-5) -- (14,-5);
\draw[->-=.5] (-14,6) -- (-6,6);
\draw[->-=.5] (-14,-6) -- (-6,-6);
\draw[->-=.5] (-6,0) -- (-11,0);
\draw[->-=.5] (-14,4) to[out=east, in=north east] (-11,0);
\draw[->-=.5] (-14,-4) to[out=east, in=south east] (-11,0);
\draw[fill=white] (-12,0) -- (-9,2) -- (-9,-2) -- cycle;
\draw[-<-=.5] (14,6) -- (6,6);
\draw[-<-=.5] (14,-6) -- (6,-6);
\draw[-<-=.5] (6,0) -- (11,0);
\draw[-<-=.5] (14,4) to[out=west, in=north west] (11,0);
\draw[-<-=.5] (14,-4) to[out=west, in=south west] (11,0);
\draw[fill=white] (12,0) -- (9,2) -- (9,-2) -- cycle;
\draw[-<-=.5] (-6,6) -- (6,6);
\draw[->-=.5] (-6,4) -- (6,4);
\draw[->-=.9, white, double=black, double distance=0.4pt, ultra thick] (-6,-6) to[out=east, in=west] (0,1);
\draw[->-=.9, white, double=black, double distance=0.4pt, ultra thick] (-6,1) to[out=east, in=west] (0,-6);
\draw[white, double=black, double distance=0.4pt, ultra thick] (0,-6) to[out=east, in=west] (6,1);
\draw[white, double=black, double distance=0.4pt, ultra thick] (0,1) to[out=east, in=west] (6,-6);
\draw[fill=white] (-7,-8) rectangle (-6,-4);
\draw[fill=white] (-7,8) rectangle (-6,-2);
\draw[fill=white] (7,-8) rectangle (6,-4);
\draw[fill=white] (7,8) rectangle (6,-2);
\node at (3,5)[above]{$\scriptstyle{l}$};
\node at (3,5)[below]{$\scriptstyle{k}$};
\node at (0,-6)[below]{$\scriptstyle{t-k-l}$};
\draw[fill=white] (-14,-8) rectangle (-15,-2);
\draw[fill=white] (-14,8) rectangle (-15,2);
\draw[fill=white] (14,-8) rectangle (15,-2);
\draw[fill=white] (14,8) rectangle (15,2);
\node at (-10,6)[above]{$\scriptstyle{t-l}$};
\node at (-10,-6)[below]{$\scriptstyle{s-l}$};
\node at (10,6)[above]{$\scriptstyle{t-l}$};
\node at (10,-6)[below]{$\scriptstyle{s-l}$};
\node at (-10,1)[above]{$\scriptstyle{l}$};
\node at (-10,-1)[below]{$\scriptstyle{l}$};
\node at (10,1)[above]{$\scriptstyle{l}$};
\node at (10,-1)[below]{$\scriptstyle{l}$};
\node at (-17,5)[above]{$\scriptstyle{t}$};
\node at (-17,-5)[below]{$\scriptstyle{s}$};
\node at (17,5)[above]{$\scriptstyle{t}$};
\node at (17,-5)[below]{$\scriptstyle{s}$};
}\ }&& \text{if $d=t$.}
\end{align*}
We remark that it is easy to see that $A_2$ clasps in the middle of $\sigma_{d}(k,l)$ can be removable from the definition of $A_{2}$ clasp and Lemma~\ref{vanish}.

\begin{PROP}\label{paralleleq}
\[
 \sigma_{d}(k,l)=q^{\frac{2}{3}(d+\delta-l)}\sigma_{d}(k+1,l)+q^{-\frac{1}{6}}q^{\frac{k+\delta}{3}}q^{-\frac{2}{3}(d+\delta-l)}(1-q^{d+\delta-l})\sigma_{d}(k,l+1)
\]
\end{PROP}
We can confirm that $X(k,l)=q^{\frac{2}{3}(d+\delta-l)}$ and $Y(k,l)=q^{-\frac{1}{6}}q^{\frac{k}{3}}q^{-\frac{2}{3}(n-l)}(1-q^{n-l})$ satisfy (2) of Proposition~\ref{partition}.
Thus,
we can prove Theorem~\ref{parallelfull} by applying Proposition~\ref{partition}.
This equation is obtained as a sequel to a calculation of $A_{2}$ webs.
We will prove it in the case of $d=s$, but the otherwise is proven in the same way.

\begin{LEM}\label{lem1}
Let $n$ and $i$ are positive integers satisfying $i<n$.
\[
\mathord{ \tikz[baseline=-.6ex, scale=.1]{
\draw (-6,5) -- (-9,5);
\draw (-6,-5) -- (-9,-5);
\draw (6,5) -- (9,5);
\draw (6,-5) -- (9,-5);
\draw[->-=.5] (-6,6) -- (6,6);
\draw[->-=.9, white, double=black, double distance=0.4pt, ultra thick] (-6,-5) to[out=east, in=west] (0,4);
\draw[->-=.9, white, double=black, double distance=0.4pt, ultra thick] (-6,4) to[out=east, in=west] (0,-5);
\draw[white, double=black, double distance=0.4pt, ultra thick] (0,-5) to[out=east, in=west] (6,4);
\draw[white, double=black, double distance=0.4pt, ultra thick] (0,4) to[out=east, in=west] (6,-5);
\draw[fill=white] (-7,-7) rectangle (-6,-3);
\draw[fill=white] (-7,7) rectangle (-6,3);
\draw[fill=white] (7,-7) rectangle (6,-3);
\draw[fill=white] (7,7) rectangle (6,3);
\node at (0,6) [above]{$\scriptstyle{i}$};
\node at (0,-5) [below]{$\scriptstyle{n-i}$};
\node at (-7,-7) [left]{$\scriptstyle{1}$};
\node at (-7,7) [left]{$\scriptstyle{n}$};
\node at (7,-7) [right]{$\scriptstyle{1}$};
\node at (7,7) [right]{$\scriptstyle{n}$};
}\ }
=q^{\frac{2}{3}}
\mathord{ \tikz[baseline=-.6ex, scale=.1]{
\draw (-6,5) -- (-9,5);
\draw (-6,-5) -- (-9,-5);
\draw (6,5) -- (9,5);
\draw (6,-5) -- (9,-5);
\draw[->-=.5] (-6,6) -- (6,6);
\draw[->-=.9, white, double=black, double distance=0.4pt, ultra thick] (-6,-5) to[out=east, in=west] (0,4);
\draw[->-=.9, white, double=black, double distance=0.4pt, ultra thick] (-6,4) to[out=east, in=west] (0,-5);
\draw[white, double=black, double distance=0.4pt, ultra thick] (0,-5) to[out=east, in=west] (6,4);
\draw[white, double=black, double distance=0.4pt, ultra thick] (0,4) to[out=east, in=west] (6,-5);
\draw[fill=white] (-7,-7) rectangle (-6,-3);
\draw[fill=white] (-7,7) rectangle (-6,3);
\draw[fill=white] (7,-7) rectangle (6,-3);
\draw[fill=white] (7,7) rectangle (6,3);
\node at (0,6) [above]{$\scriptstyle{i+1}$};
\node at (0,-5) [below]{$\scriptstyle{n-i-1}$};
\node at (-7,-7) [left]{$\scriptstyle{1}$};
\node at (-7,7) [left]{$\scriptstyle{n}$};
\node at (7,-7) [right]{$\scriptstyle{1}$};
\node at (7,7) [right]{$\scriptstyle{n}$};
}\ }
+q^{-\frac{1}{2}}q^{-\frac{n}{3}}q^{\frac{i}{3}}(1-q)
\mathord{\ \tikz[baseline=-.6ex, scale=.1]{
\draw (-6,5) -- (-9,5);
\draw (-6,-5) -- (-9,-5);
\draw (6,5) -- (9,5);
\draw (6,-5) -- (9,-5);
\draw[->-=.5] (-6,6) -- (6,6);
\draw[->-=.5] (-6,4) -- (-2,0);
\draw[->-=.5] (-6,-5) -- (-2,0);
\draw[->-=.7] (2,0) -- (6,4);
\draw[->-=.7] (2,0) -- (6,-5);
\draw[-<-=.5] (-2,0) -- (2,0);
\draw[fill=white] (-7,-7) rectangle (-6,-3);
\draw[fill=white] (-7,7) rectangle (-6,3);
\draw[fill=white] (7,-7) rectangle (6,-3);
\draw[fill=white] (7,7) rectangle (6,3);
\node at (-7,-7)[left]{$\scriptstyle{1}$};
\node at (-7,7)[left]{$\scriptstyle{n}$};
\node at (7,-7)[right]{$\scriptstyle{1}$};
\node at (7,7)[right]{$\scriptstyle{n}$};
\node at (0,6)[above]{$\scriptstyle{n-1}$};
}\ }
\]
\end{LEM}

\begin{proof}
We will make a red mark on a place where we use the skein relation and Lemma~\ref{claspformula}.
\begin{align*}
&\mathord{ \tikz[baseline=-.6ex, scale=.1]{
\draw (-6,5) -- (-9,5);
\draw (-6,-5) -- (-9,-5);
\draw (6,5) -- (9,5);
\draw (6,-5) -- (9,-5);
\draw[->-=.5] (-6,6) -- (6,6);
\draw[->-=.9, white, double=black, double distance=0.4pt, ultra thick] (-6,-5) to[out=east, in=west] (0,4);
\draw[->-=.9, white, double=black, double distance=0.4pt, ultra thick] (-6,4) to[out=east, in=west] (0,-2);
\draw[->-=.9, white, double=black, double distance=0.4pt, ultra thick] (-6,2) to[out=east, in=west] (0,-5);
\draw[white, double=black, double distance=0.4pt, ultra thick] (0,-5) to[out=east, in=west] (6,2);
\draw[white, double=black, double distance=0.4pt, ultra thick] (0,-2) to[out=east, in=west] (6,4);
\draw[white, double=black, double distance=0.4pt, ultra thick] (0,4) to[out=east, in=west] (6,-5);
\draw[fill=white] (-7,-7) rectangle (-6,-3);
\draw[fill=white] (-7,7) rectangle (-6,1);
\draw[fill=white] (7,-7) rectangle (6,-3);
\draw[fill=white] (7,7) rectangle (6,1);
\node at (0,6) [above]{$\scriptstyle{i}$};
\node at (0,-2) [above]{$\scriptstyle{1}$};
\node at (0,-5) [below]{$\scriptstyle{n-i-1}$};
\node at (-7,-7) [left]{$\scriptstyle{1}$};
\node at (-7,7) [left]{$\scriptstyle{n}$};
\node at (7,-7) [right]{$\scriptstyle{1}$};
\node at (7,7) [right]{$\scriptstyle{n}$};
\fill[opacity=.2, magenta] (-3, 1) circle [radius=1];
\fill[opacity=.2, magenta] (3, 1) circle [radius=1];
}\ }\\
&=q^{\frac{2}{3}}
\mathord{ \tikz[baseline=-.6ex, scale=.1]{
\draw (-6,5) -- (-9,5);
\draw (-6,-5) -- (-9,-5);
\draw (6,5) -- (9,5);
\draw (6,-5) -- (9,-5);
\draw[->-=.5] (-6,6) -- (6,6);
\draw[->-=.9, white, double=black, double distance=0.4pt, ultra thick] (-6,-5) to[out=east, in=west] (0,4);
\draw[->-=.9, white, double=black, double distance=0.4pt, ultra thick] (-6,4) to[out=east, in=west] (0,-5);
\draw[white, double=black, double distance=0.4pt, ultra thick] (0,-5) to[out=east, in=west] (6,4);
\draw[white, double=black, double distance=0.4pt, ultra thick] (0,4) to[out=east, in=west] (6,-5);
\draw[fill=white] (-7,-7) rectangle (-6,-3);
\draw[fill=white] (-7,7) rectangle (-6,3);
\draw[fill=white] (7,-7) rectangle (6,-3);
\draw[fill=white] (7,7) rectangle (6,3);
\node at (0,6) [above]{$\scriptstyle{i+1}$};
\node at (0,-5) [below]{$\scriptstyle{n-i-1}$};
\node at (-7,-7) [left]{$\scriptstyle{1}$};
\node at (-7,7) [left]{$\scriptstyle{n}$};
\node at (7,-7) [right]{$\scriptstyle{1}$};
\node at (7,7) [right]{$\scriptstyle{n}$};
}\ }
+q^{-\frac{5}{6}}(1-q)
\mathord{ \tikz[baseline=-.6ex, scale=.1]{
\draw (-6,5) -- (-9,5);
\draw (-6,-5) -- (-9,-5);
\draw (6,5) -- (9,5);
\draw (6,-5) -- (9,-5);
\draw[->-=.5] (-6,6) -- (6,6);
\draw[->-=.4, white, double=black, double distance=0.4pt, ultra thick] (-6,-5) -- (-2,2);
\draw[->-=.9, white, double=black, double distance=0.4pt, ultra thick] (-6,2) to[out=east, in=west] (0,-5);
\draw[white, double=black, double distance=0.4pt, ultra thick] (0,-5) to[out=east, in=west] (6,2);
\draw[->-=.8, white, double=black, double distance=0.4pt, ultra thick] (2,2) -- (6,-5);
\draw[-<-=.5] (-2,2) -- (2,2);
\draw[->-=.5] (-6,4) -- (-2,2);
\draw[-<-=.5] (6,4) -- (2,2);
\draw[fill=white] (-7,-7) rectangle (-6,-3);
\draw[fill=white] (-7,7) rectangle (-6,1);
\draw[fill=white] (7,-7) rectangle (6,-3);
\draw[fill=white] (7,7) rectangle (6,1);
\node at (0,6) [above]{$\scriptstyle{i}$};
\node at (0,-5) [below]{$\scriptstyle{n-i-1}$};
\node at (-7,-7) [left]{$\scriptstyle{1}$};
\node at (-7,7) [left]{$\scriptstyle{n}$};
\node at (7,-7) [right]{$\scriptstyle{1}$};
\node at (7,7) [right]{$\scriptstyle{n}$};
\fill[opacity=.2, magenta] (-6, 1) rectangle (-5,5);
\fill[opacity=.2, magenta] (6, 1) rectangle (5,5);
}\ }\\
&=q^{\frac{2}{3}}
\mathord{ \tikz[baseline=-.6ex, scale=.1]{
\draw (-6,5) -- (-9,5);
\draw (-6,-5) -- (-9,-5);
\draw (6,5) -- (9,5);
\draw (6,-5) -- (9,-5);
\draw[->-=.5] (-6,6) -- (6,6);
\draw[->-=.9, white, double=black, double distance=0.4pt, ultra thick] (-6,-5) to[out=east, in=west] (0,4);
\draw[->-=.9, white, double=black, double distance=0.4pt, ultra thick] (-6,4) to[out=east, in=west] (0,-5);
\draw[white, double=black, double distance=0.4pt, ultra thick] (0,-5) to[out=east, in=west] (6,4);
\draw[white, double=black, double distance=0.4pt, ultra thick] (0,4) to[out=east, in=west] (6,-5);
\draw[fill=white] (-7,-7) rectangle (-6,-3);
\draw[fill=white] (-7,7) rectangle (-6,3);
\draw[fill=white] (7,-7) rectangle (6,-3);
\draw[fill=white] (7,7) rectangle (6,3);
\node at (0,6) [above]{$\scriptstyle{i+1}$};
\node at (0,-5) [below]{$\scriptstyle{n-i-1}$};
\node at (-7,-7) [left]{$\scriptstyle{1}$};
\node at (-7,7) [left]{$\scriptstyle{n}$};
\node at (7,-7) [right]{$\scriptstyle{1}$};
\node at (7,7) [right]{$\scriptstyle{n}$};
}\ }
+q^{-\frac{5}{6}}(1-q)q^{\frac{2}{3}(n-i-1)}
\mathord{ \tikz[baseline=-.6ex, scale=.1]{
\draw (-6,5) -- (-9,5);
\draw (-6,-5) -- (-9,-5);
\draw (6,5) -- (9,5);
\draw (6,-5) -- (9,-5);
\draw[->-=.5] (-6,6) -- (6,6);
\draw[white, double=black, double distance=0.4pt, ultra thick] (0,-4) to[out=east, in=west] (4,4) -- (6,4);
\draw[->-=.5] (-6,2) -- (-4,0);
\draw[->-=.5] (-6,-5) -- (-4,0);
\draw[-<-=.5, white, double=black, double distance=0.4pt, ultra thick] (-4,0) -- (4,0);
\draw[-<-=.5] (6,2) -- (4,0);
\draw[-<-=.5] (6,-5) -- (4,0);
\draw[->-=.9, white, double=black, double distance=0.4pt, ultra thick] (-6,4) -- (-4,4) to[out=east, in=west] (0,-4);
\draw[fill=white] (-7,-7) rectangle (-6,-3);
\draw[fill=white] (-7,7) rectangle (-6,1);
\draw[fill=white] (7,-7) rectangle (6,-3);
\draw[fill=white] (7,7) rectangle (6,1);
\node at (0,6) [above]{$\scriptstyle{i}$};
\node at (0,-4) [below]{$\scriptstyle{n-i-1}$};
\node at (-7,-7) [left]{$\scriptstyle{1}$};
\node at (-7,7) [left]{$\scriptstyle{n}$};
\node at (7,-7) [right]{$\scriptstyle{1}$};
\node at (7,7) [right]{$\scriptstyle{n}$};
\fill[opacity=.2, magenta] (-2, 0) circle [radius=1];
\fill[opacity=.2, magenta] (2, 0) circle [radius=1];
}\ }\\
&=q^{\frac{2}{3}}
\mathord{ \tikz[baseline=-.6ex, scale=.1]{
\draw (-6,5) -- (-9,5);
\draw (-6,-5) -- (-9,-5);
\draw (6,5) -- (9,5);
\draw (6,-5) -- (9,-5);
\draw[->-=.5] (-6,6) -- (6,6);
\draw[->-=.9, white, double=black, double distance=0.4pt, ultra thick] (-6,-5) to[out=east, in=west] (0,4);
\draw[->-=.9, white, double=black, double distance=0.4pt, ultra thick] (-6,4) to[out=east, in=west] (0,-5);
\draw[white, double=black, double distance=0.4pt, ultra thick] (0,-5) to[out=east, in=west] (6,4);
\draw[white, double=black, double distance=0.4pt, ultra thick] (0,4) to[out=east, in=west] (6,-5);
\draw[fill=white] (-7,-7) rectangle (-6,-3);
\draw[fill=white] (-7,7) rectangle (-6,3);
\draw[fill=white] (7,-7) rectangle (6,-3);
\draw[fill=white] (7,7) rectangle (6,3);
\node at (0,6) [above]{$\scriptstyle{i+1}$};
\node at (0,-5) [below]{$\scriptstyle{n-i-1}$};
\node at (-7,-7) [left]{$\scriptstyle{1}$};
\node at (-7,7) [left]{$\scriptstyle{n}$};
\node at (7,-7) [right]{$\scriptstyle{1}$};
\node at (7,7) [right]{$\scriptstyle{n}$};
}\ }
+q^{-\frac{5}{6}}(1-q)q^{-\frac{2}{3}(n-i-1)}(-q^{\frac{1}{6}})^{2(n-i-1)}
\mathord{ \tikz[baseline=-.6ex, scale=.1]{
\draw (-6,5) -- (-9,5);
\draw (-6,-5) -- (-9,-5);
\draw (6,5) -- (9,5);
\draw (6,-5) -- (9,-5);
\draw[->-=.5] (-6,6) -- (6,6);
\draw[white, double=black, double distance=0.4pt, ultra thick] (0,-4) to[out=east, in=west] (4,4) -- (6,4);
\draw[->-=.5] (-6,2) -- (-4,0);
\draw[->-=.5] (-6,-5) -- (-4,0);
\draw[-<-=.5, white, double=black, double distance=0.4pt, ultra thick] (-4,0) -- (4,0);
\draw[-<-=.5] (6,2) -- (4,0);
\draw[-<-=.5] (6,-5) -- (4,0);
\draw[->-=.9, white, double=black, double distance=0.4pt, ultra thick] (-6,4) -- (-4,4) to[out=east, in=west] (0,-4);
\draw[fill=white] (-1,-1) rectangle (-3,1);
\draw[fill=white] (1,-1) rectangle (3,1);
\draw (-1,-1) -- (-3,1);
\draw (1,-1) -- (3,1);
\draw[fill=white] (-7,-7) rectangle (-6,-3);
\draw[fill=white] (-7,7) rectangle (-6,1);
\draw[fill=white] (7,-7) rectangle (6,-3);
\draw[fill=white] (7,7) rectangle (6,1);
\node at (0,6) [above]{$\scriptstyle{i}$};
\node at (0,-4) [below]{$\scriptstyle{n-i-1}$};
\node at (-7,-7) [left]{$\scriptstyle{1}$};
\node at (-7,7) [left]{$\scriptstyle{n}$};
\node at (7,-7) [right]{$\scriptstyle{1}$};
\node at (7,7) [right]{$\scriptstyle{n}$};
\fill[opacity=.2, magenta] (-3,-4) rectangle (3,2);
}\ }\\
&=q^{\frac{2}{3}}
\mathord{ \tikz[baseline=-.6ex, scale=.1]{
\draw (-6,5) -- (-9,5);
\draw (-6,-5) -- (-9,-5);
\draw (6,5) -- (9,5);
\draw (6,-5) -- (9,-5);
\draw[->-=.5] (-6,6) -- (6,6);
\draw[->-=.9, white, double=black, double distance=0.4pt, ultra thick] (-6,-5) to[out=east, in=west] (0,4);
\draw[->-=.9, white, double=black, double distance=0.4pt, ultra thick] (-6,4) to[out=east, in=west] (0,-5);
\draw[white, double=black, double distance=0.4pt, ultra thick] (0,-5) to[out=east, in=west] (6,4);
\draw[white, double=black, double distance=0.4pt, ultra thick] (0,4) to[out=east, in=west] (6,-5);
\draw[fill=white] (-7,-7) rectangle (-6,-3);
\draw[fill=white] (-7,7) rectangle (-6,3);
\draw[fill=white] (7,-7) rectangle (6,-3);
\draw[fill=white] (7,7) rectangle (6,3);
\node at (0,6) [above]{$\scriptstyle{i+1}$};
\node at (0,-5) [below]{$\scriptstyle{n-i-1}$};
\node at (-7,-7) [left]{$\scriptstyle{1}$};
\node at (-7,7) [left]{$\scriptstyle{n}$};
\node at (7,-7) [right]{$\scriptstyle{1}$};
\node at (7,7) [right]{$\scriptstyle{n}$};
}\ }
+q^{-\frac{5}{6}}(1-q)q^{\frac{2}{3}(n-i-1)}(-q^{\frac{1}{6}})^{2(n-i-1)}
\mathord{\ \tikz[baseline=-.6ex, scale=.1]{
\draw (-6,5) -- (-9,5);
\draw (-6,-5) -- (-9,-5);
\draw (6,5) -- (9,5);
\draw (6,-5) -- (9,-5);
\draw[->-=.5] (-6,6) -- (6,6);
\draw[->-=.5] (-6,4) -- (-2,0);
\draw[->-=.5] (-6,-5) -- (-2,0);
\draw[->-=.7] (2,0) -- (6,4);
\draw[->-=.7] (2,0) -- (6,-5);
\draw[-<-=.5] (-2,0) -- (2,0);
\draw[fill=white] (-7,-7) rectangle (-6,-3);
\draw[fill=white] (-7,7) rectangle (-6,3);
\draw[fill=white] (7,-7) rectangle (6,-3);
\draw[fill=white] (7,7) rectangle (6,3);
\node at (-7,-7)[left]{$\scriptstyle{1}$};
\node at (-7,7)[left]{$\scriptstyle{n}$};
\node at (7,-7)[right]{$\scriptstyle{1}$};
\node at (7,7)[right]{$\scriptstyle{n}$};
\node at (0,6)[above]{$\scriptstyle{n-1}$};
}\ }
\end{align*}
\end{proof}

\begin{LEM}\label{lem2}
Let $n$ be a positive integer.
\[
\mathord{ \tikz[baseline=-.6ex, scale=.1]{
\draw (-6,5) -- (-9,5);
\draw (-6,-5) -- (-9,-5);
\draw (6,5) -- (9,5);
\draw (6,-5) -- (9,-5);
\draw[->-=.9, white, double=black, double distance=0.4pt, ultra thick] (-6,-5) to[out=east, in=west] (0,5);
\draw[->-=.9, white, double=black, double distance=0.4pt, ultra thick] (-6,5) to[out=east, in=west] (0,-5);
\draw[white, double=black, double distance=0.4pt, ultra thick] (0,-5) to[out=east, in=west] (6,5);
\draw[white, double=black, double distance=0.4pt, ultra thick] (0,5) to[out=east, in=west] (6,-5);
\draw[fill=white] (-7,-7) rectangle (-6,-3);
\draw[fill=white] (-7,7) rectangle (-6,3);
\draw[fill=white] (7,-7) rectangle (6,-3);
\draw[fill=white] (7,7) rectangle (6,3);
\node at (-7,-7) [left]{$\scriptstyle{1}$};
\node at (-7,7) [left]{$\scriptstyle{n}$};
\node at (7,-7) [right]{$\scriptstyle{1}$};
\node at (7,7) [right]{$\scriptstyle{n}$};
}\ }
=q^{\frac{2n}{3}}
\mathord{ \tikz[baseline=-.6ex, scale=.1]{
\draw (-6,5) -- (-9,5);
\draw (-6,-5) -- (-9,-5);
\draw (6,5) -- (9,5);
\draw (6,-5) -- (9,-5);
\draw[->-=.5] (-6,5) -- (6,5);
\draw[->-=.5] (-6,-5) -- (6,-5);
\draw[fill=white] (-7,-7) rectangle (-6,-3);
\draw[fill=white] (-7,7) rectangle (-6,3);
\draw[fill=white] (7,-7) rectangle (6,-3);
\draw[fill=white] (7,7) rectangle (6,3);
\node at (-7,-7) [left]{$\scriptstyle{1}$};
\node at (-7,7) [left]{$\scriptstyle{n}$};
\node at (7,-7) [right]{$\scriptstyle{1}$};
\node at (7,7) [right]{$\scriptstyle{n}$};
}\ }
+q^{-\frac{1}{2}}q^{-\frac{n}{3}}(1-q^{n})
\mathord{\ \tikz[baseline=-.6ex, scale=.1]{
\draw (-6,5) -- (-9,5);
\draw (-6,-5) -- (-9,-5);
\draw (6,5) -- (9,5);
\draw (6,-5) -- (9,-5);
\draw[->-=.5] (-6,6) -- (6,6);
\draw[->-=.5] (-6,4) -- (-2,0);
\draw[->-=.5] (-6,-5) -- (-2,0);
\draw[->-=.7] (2,0) -- (6,4);
\draw[->-=.7] (2,0) -- (6,-5);
\draw[-<-=.5] (-2,0) -- (2,0);
\draw[fill=white] (-7,-7) rectangle (-6,-3);
\draw[fill=white] (-7,7) rectangle (-6,3);
\draw[fill=white] (7,-7) rectangle (6,-3);
\draw[fill=white] (7,7) rectangle (6,3);
\node at (-7,-7)[left]{$\scriptstyle{1}$};
\node at (-7,7)[left]{$\scriptstyle{n}$};
\node at (7,-7)[right]{$\scriptstyle{1}$};
\node at (7,7)[right]{$\scriptstyle{n}$};
\node at (0,6)[above]{$\scriptstyle{n-1}$};
}\ }\
\]
\end{LEM}

\begin{proof}
By using Lemma~\ref{lem1}, we obtain
\begin{align*}
&q^{\frac{2i}{3}}
\mathord{ \tikz[baseline=-.6ex, scale=.1]{
\draw (-6,5) -- (-9,5);
\draw (-6,-5) -- (-9,-5);
\draw (6,5) -- (9,5);
\draw (6,-5) -- (9,-5);
\draw[->-=.5] (-6,6) -- (6,6);
\draw[->-=.9, white, double=black, double distance=0.4pt, ultra thick] (-6,-5) to[out=east, in=west] (0,4);
\draw[->-=.9, white, double=black, double distance=0.4pt, ultra thick] (-6,4) to[out=east, in=west] (0,-5);
\draw[white, double=black, double distance=0.4pt, ultra thick] (0,-5) to[out=east, in=west] (6,4);
\draw[white, double=black, double distance=0.4pt, ultra thick] (0,4) to[out=east, in=west] (6,-5);
\draw[fill=white] (-7,-7) rectangle (-6,-3);
\draw[fill=white] (-7,7) rectangle (-6,3);
\draw[fill=white] (7,-7) rectangle (6,-3);
\draw[fill=white] (7,7) rectangle (6,3);
\node at (0,6) [above]{$\scriptstyle{i}$};
\node at (0,-5) [below]{$\scriptstyle{n-i}$};
\node at (-7,-7) [left]{$\scriptstyle{1}$};
\node at (-7,7) [left]{$\scriptstyle{n}$};
\node at (7,-7) [right]{$\scriptstyle{1}$};
\node at (7,7) [right]{$\scriptstyle{n}$};
}\ }
+q^{-\frac{1}{2}}q^{-\frac{n}{3}}\left(1-q^{i}\right)
\mathord{\ \tikz[baseline=-.6ex, scale=.1]{
\draw (-6,5) -- (-9,5);
\draw (-6,-5) -- (-9,-5);
\draw (6,5) -- (9,5);
\draw (6,-5) -- (9,-5);
\draw[->-=.5] (-6,6) -- (6,6);
\draw[->-=.5] (-6,4) -- (-2,0);
\draw[->-=.5] (-6,-5) -- (-2,0);
\draw[->-=.7] (2,0) -- (6,4);
\draw[->-=.7] (2,0) -- (6,-5);
\draw[-<-=.5] (-2,0) -- (2,0);
\draw[fill=white] (-7,-7) rectangle (-6,-3);
\draw[fill=white] (-7,7) rectangle (-6,3);
\draw[fill=white] (7,-7) rectangle (6,-3);
\draw[fill=white] (7,7) rectangle (6,3);
\node at (-7,-7)[left]{$\scriptstyle{1}$};
\node at (-7,7)[left]{$\scriptstyle{n}$};
\node at (7,-7)[right]{$\scriptstyle{1}$};
\node at (7,7)[right]{$\scriptstyle{n}$};
\node at (0,6)[above]{$\scriptstyle{n-1}$};
}\ }\\
&\quad=q^{\frac{2(i+1)}{3}}
\mathord{ \tikz[baseline=-.6ex, scale=.1]{
\draw (-6,5) -- (-9,5);
\draw (-6,-5) -- (-9,-5);
\draw (6,5) -- (9,5);
\draw (6,-5) -- (9,-5);
\draw[->-=.5] (-6,6) -- (6,6);
\draw[->-=.9, white, double=black, double distance=0.4pt, ultra thick] (-6,-5) to[out=east, in=west] (0,4);
\draw[->-=.9, white, double=black, double distance=0.4pt, ultra thick] (-6,4) to[out=east, in=west] (0,-5);
\draw[white, double=black, double distance=0.4pt, ultra thick] (0,-5) to[out=east, in=west] (6,4);
\draw[white, double=black, double distance=0.4pt, ultra thick] (0,4) to[out=east, in=west] (6,-5);
\draw[fill=white] (-7,-7) rectangle (-6,-3);
\draw[fill=white] (-7,7) rectangle (-6,3);
\draw[fill=white] (7,-7) rectangle (6,-3);
\draw[fill=white] (7,7) rectangle (6,3);
\node at (0,6) [above]{$\scriptstyle{i+1}$};
\node at (0,-5) [below]{$\scriptstyle{n-i-1}$};
\node at (-7,-7) [left]{$\scriptstyle{1}$};
\node at (-7,7) [left]{$\scriptstyle{n}$};
\node at (7,-7) [right]{$\scriptstyle{1}$};
\node at (7,7) [right]{$\scriptstyle{n}$};
}\ }
+q^{-\frac{1}{2}}q^{-\frac{n}{3}}\left(1-q^{i+1}\right)
\mathord{\ \tikz[baseline=-.6ex, scale=.1]{
\draw (-6,5) -- (-9,5);
\draw (-6,-5) -- (-9,-5);
\draw (6,5) -- (9,5);
\draw (6,-5) -- (9,-5);
\draw[->-=.5] (-6,6) -- (6,6);
\draw[->-=.5] (-6,4) -- (-2,0);
\draw[->-=.5] (-6,-5) -- (-2,0);
\draw[->-=.7] (2,0) -- (6,4);
\draw[->-=.7] (2,0) -- (6,-5);
\draw[-<-=.5] (-2,0) -- (2,0);
\draw[fill=white] (-7,-7) rectangle (-6,-3);
\draw[fill=white] (-7,7) rectangle (-6,3);
\draw[fill=white] (7,-7) rectangle (6,-3);
\draw[fill=white] (7,7) rectangle (6,3);
\node at (-7,-7)[left]{$\scriptstyle{1}$};
\node at (-7,7)[left]{$\scriptstyle{n}$};
\node at (7,-7)[right]{$\scriptstyle{1}$};
\node at (7,7)[right]{$\scriptstyle{n}$};
\node at (0,6)[above]{$\scriptstyle{n-1}$};
}\ }.
\end{align*}
We repeatedly apply the above equation to the RHS of Lemma~\ref{lem1} at $i=0$.
\end{proof}

\begin{LEM}\label{lem3}
Let $i$ and $j$ be positive integers.
\[
\mathord{ \tikz[baseline=-.6ex, scale=.1]{
\draw (-6,5) -- (-9,5);
\draw (-6,-5) -- (-9,-5);
\draw (6,5) -- (9,5);
\draw (6,-5) -- (9,-5);
\draw[->-=.9, white, double=black, double distance=0.4pt, ultra thick] (-6,-5) to[out=east, in=west] (0,5);
\draw[->-=.9, white, double=black, double distance=0.4pt, ultra thick] (-6,5) to[out=east, in=west] (0,-5);
\draw[white, double=black, double distance=0.4pt, ultra thick] (0,-5) to[out=east, in=west] (6,5);
\draw[white, double=black, double distance=0.4pt, ultra thick] (0,5) to[out=east, in=west] (6,-5);
\draw[fill=white] (-7,-7) rectangle (-6,-3);
\draw[fill=white] (-7,7) rectangle (-6,3);
\draw[fill=white] (7,-7) rectangle (6,-3);
\draw[fill=white] (7,7) rectangle (6,3);
\node at (0,5) [above]{$\scriptstyle{i}$};
\node at (0,-5) [below]{$\scriptstyle{j}$};
\node at (-7,-7) [left]{$\scriptstyle{i}$};
\node at (-7,7) [left]{$\scriptstyle{j}$};
\node at (7,-7) [right]{$\scriptstyle{i}$};
\node at (7,7) [right]{$\scriptstyle{j}$};
}\ }
=q^{\frac{2j}{3}}
\mathord{ \tikz[baseline=-.6ex, scale=.1]{
\draw (-6,5) -- (-9,5);
\draw (-6,-5) -- (-9,-5);
\draw (6,5) -- (9,5);
\draw (6,-5) -- (9,-5);
\draw[->-=.5] (-6,-6) -- (6,-6);
\draw[->-=.9, white, double=black, double distance=0.4pt, ultra thick] (-6,-4) to[out=east, in=west] (0,5);
\draw[->-=.9, white, double=black, double distance=0.4pt, ultra thick] (-6,5) to[out=east, in=west] (0,-4);
\draw[white, double=black, double distance=0.4pt, ultra thick] (0,-4) to[out=east, in=west] (6,5);
\draw[white, double=black, double distance=0.4pt, ultra thick] (0,5) to[out=east, in=west] (6,-4);
\draw[fill=white] (-7,-7) rectangle (-6,-3);
\draw[fill=white] (-7,7) rectangle (-6,3);
\draw[fill=white] (7,-7) rectangle (6,-3);
\draw[fill=white] (7,7) rectangle (6,3);
\node at (0,5) [above]{$\scriptstyle{i-1}$};
\node at (0,-4) [above]{$\scriptstyle{j}$};
\node at (0,-6) [below]{$\scriptstyle{1}$};
\node at (-7,-7) [left]{$\scriptstyle{i}$};
\node at (-7,7) [left]{$\scriptstyle{j}$};
\node at (7,-7) [right]{$\scriptstyle{i}$};
\node at (7,7) [right]{$\scriptstyle{j}$};
}\ }
+q^{-\frac{1}{6}}q^{-\frac{i+j}{3}}(1-q^{j})
\mathord{ \tikz[baseline=-.6ex, scale=.1]{
\draw (-17,5) -- (-14,5);
\draw (-17,-5) -- (-14,-5);
\draw (17,5) -- (14,5);
\draw (17,-5) -- (14,-5);
\draw[->-=.5] (-14,6) -- (-6,6);
\draw[->-=.5] (-14,-6) -- (-6,-6);
\draw[->-=.5] (-7,0) -- (-11,0);
\draw[->-=.5] (-14,4) -- (-11,0);
\draw[->-=.5] (-14,-4) -- (-11,0);
\draw[-<-=.5] (14,6) -- (6,6);
\draw[-<-=.5] (14,-6) -- (6,-6);
\draw[-<-=.5] (7,0) -- (11,0);
\draw[-<-=.5] (14,4) -- (11,0);
\draw[-<-=.5] (14,-4) -- (11,0);
\draw[-<-=.5] (-6,-6) -- (6,-6);
\draw[->-=.9, white, double=black, double distance=0.4pt, ultra thick] (-6,-1) to[out=east, in=west] (0,6);
\draw[->-=.9, white, double=black, double distance=0.4pt, ultra thick] (-6,6) to[out=east, in=west] (0,-1);
\draw[white, double=black, double distance=0.4pt, ultra thick] (0,-1) to[out=east, in=west] (6,6);
\draw[white, double=black, double distance=0.4pt, ultra thick] (0,6) to[out=east, in=west] (6,-1);
\draw[fill=white] (-7,8) rectangle (-6,4);
\draw[fill=white] (-7,-8) rectangle (-6,2);
\draw[fill=white] (7,8) rectangle (6,4);
\draw[fill=white] (7,-8) rectangle (6,2);
\node at (3,-5)[below]{$\scriptstyle{1}$};
\node at (0,6)[above]{$\scriptstyle{i-1}$};
\draw[fill=white] (-14,-8) rectangle (-15,-2);
\draw[fill=white] (-14,8) rectangle (-15,2);
\draw[fill=white] (14,-8) rectangle (15,-2);
\draw[fill=white] (14,8) rectangle (15,2);
\node at (-10,6)[above]{$\scriptstyle{j-1}$};
\node at (-10,-6)[below]{$\scriptstyle{i-1}$};
\node at (10,6)[above]{$\scriptstyle{j-1}$};
\node at (10,-6)[below]{$\scriptstyle{i-1}$};
\node at (-17,5)[above]{$\scriptstyle{j}$};
\node at (-17,-5)[below]{$\scriptstyle{i}$};
\node at (17,5)[above]{$\scriptstyle{j}$};
\node at (17,-5)[below]{$\scriptstyle{i}$};
}\ }
\]
\end{LEM}

\begin{proof}
By Lemma~\ref{lem2},
\begin{align*}
\mathord{ \tikz[baseline=-.6ex, scale=.1]{
\draw (-6,5) -- (-9,5);
\draw (-6,-5) -- (-9,-5);
\draw (6,5) -- (9,5);
\draw (6,-5) -- (9,-5);
\draw[->-=.9, white, double=black, double distance=0.4pt, ultra thick] (-6,-6) to[out=east, in=west] (0,1);
\draw[->-=.9, white, double=black, double distance=0.4pt, ultra thick] (-6,-3) to[out=east, in=west] (0,5);
\draw[->-=.9, white, double=black, double distance=0.4pt, ultra thick] (-6,5) to[out=east, in=west] (0,-6);
\draw[white, double=black, double distance=0.4pt, ultra thick] (0,-6) to[out=east, in=west] (6,5);
\draw[white, double=black, double distance=0.4pt, ultra thick] (0,5) to[out=east, in=west] (6,-3);
\draw[white, double=black, double distance=0.4pt, ultra thick] (0,1) to[out=east, in=west] (6,-6);
\draw[fill=white] (-7,7) rectangle (-6,3);
\draw[fill=white] (-7,-7) rectangle (-6,-2);
\draw[fill=white] (7,7) rectangle (6,3);
\draw[fill=white] (7,-7) rectangle (6,-2);
\node at (0,-6) [above]{$\scriptstyle{j}$};
\node at (0,1) [above]{$\scriptstyle{1}$};
\node at (0,5) [above]{$\scriptstyle{i-1}$};
\node at (-7,7) [left]{$\scriptstyle{j}$};
\node at (-7,-7) [left]{$\scriptstyle{i}$};
\node at (7,7) [right]{$\scriptstyle{j}$};
\node at (7,-7) [right]{$\scriptstyle{i}$};
}\ }
&=q^{\frac{2j}{3}}
\mathord{ \tikz[baseline=-.6ex, scale=.1]{
\draw (-6,5) -- (-9,5);
\draw (-6,-5) -- (-9,-5);
\draw (6,5) -- (9,5);
\draw (6,-5) -- (9,-5);
\draw[->-=.5] (-6,-6) -- (6,-6);
\draw[->-=.9, white, double=black, double distance=0.4pt, ultra thick] (-6,-4) to[out=east, in=west] (0,5);
\draw[->-=.9, white, double=black, double distance=0.4pt, ultra thick] (-6,5) to[out=east, in=west] (0,-4);
\draw[white, double=black, double distance=0.4pt, ultra thick] (0,-4) to[out=east, in=west] (6,5);
\draw[white, double=black, double distance=0.4pt, ultra thick] (0,5) to[out=east, in=west] (6,-4);
\draw[fill=white] (-7,-7) rectangle (-6,-3);
\draw[fill=white] (-7,7) rectangle (-6,3);
\draw[fill=white] (7,-7) rectangle (6,-3);
\draw[fill=white] (7,7) rectangle (6,3);
\node at (0,5) [above]{$\scriptstyle{i-1}$};
\node at (0,-4) [above]{$\scriptstyle{j}$};
\node at (0,-6) [below]{$\scriptstyle{1}$};
\node at (-7,-7) [left]{$\scriptstyle{i}$};
\node at (-7,7) [left]{$\scriptstyle{j}$};
\node at (7,-7) [right]{$\scriptstyle{i}$};
\node at (7,7) [right]{$\scriptstyle{j}$};
}\ }
+q^{-\frac{1}{2}}q^{-\frac{j}{3}}\left(1-q^{j}\right)
\mathord{ \tikz[baseline=-.6ex, scale=.1]{
\draw (-6,5) -- (-9,5);
\draw (-6,-5) -- (-9,-5);
\draw (6,5) -- (9,5);
\draw (6,-5) -- (9,-5);
\draw (0,2) to[out=east, in=west] (6,6);
\draw (0,0) -- (3,0);
\draw[->-=.3, white, double=black, double distance=0.4pt, ultra thick] (3,0) -- (6,-6);
\draw[white, double=black, double distance=0.4pt, ultra thick] (-6,-3) -- (-3,-3) to[out=east, in=west] (0,6);
\draw[->-=.5, white, double=black, double distance=0.4pt, ultra thick] (-6,6) to[out=east, in=west] (0,2);
\draw[-<-=.9, white, double=black, double distance=0.4pt, ultra thick] (-3,0) -- (0,0);
\draw[->-=.2, white, double=black, double distance=0.4pt, ultra thick] (0,6) to[out=east, in=west] (3,-3) -- (6,-3);
\draw[->-=.8, white, double=black, double distance=0.4pt, ultra thick] (-6,-6) -- (-3,0);
\draw[-<-=.7] (6,4) -- (3,0);
\draw[->-=.8] (-6,4) -- (-3,0);
\draw[fill=white] (-7,-7) rectangle (-6,-2);
\draw[fill=white] (-7,7) rectangle (-6,2);
\draw[fill=white] (7,-7) rectangle (6,-2);
\draw[fill=white] (7,7) rectangle (6,2);
\node at (2,6) [above]{$\scriptstyle{i-1}$};
\node at (-4,6) [above]{$\scriptstyle{j-1}$};
\node at (-7,-7) [left]{$\scriptstyle{i}$};
\node at (-7,7) [left]{$\scriptstyle{j}$};
\node at (7,-7) [right]{$\scriptstyle{i}$};
\node at (7,7) [right]{$\scriptstyle{j}$};
\fill[opacity=.2, magenta] (-4,-3) circle[radius=1];
\fill[opacity=.2, magenta] (4,-3) circle[radius=1];
}\ }\\
&=q^{\frac{2j}{3}}
\mathord{ \tikz[baseline=-.6ex, scale=.1]{
\draw (-6,5) -- (-9,5);
\draw (-6,-5) -- (-9,-5);
\draw (6,5) -- (9,5);
\draw (6,-5) -- (9,-5);
\draw[->-=.5] (-6,-6) -- (6,-6);
\draw[->-=.9, white, double=black, double distance=0.4pt, ultra thick] (-6,-4) to[out=east, in=west] (0,5);
\draw[->-=.9, white, double=black, double distance=0.4pt, ultra thick] (-6,5) to[out=east, in=west] (0,-4);
\draw[white, double=black, double distance=0.4pt, ultra thick] (0,-4) to[out=east, in=west] (6,5);
\draw[white, double=black, double distance=0.4pt, ultra thick] (0,5) to[out=east, in=west] (6,-4);
\draw[fill=white] (-7,-7) rectangle (-6,-3);
\draw[fill=white] (-7,7) rectangle (-6,3);
\draw[fill=white] (7,-7) rectangle (6,-3);
\draw[fill=white] (7,7) rectangle (6,3);
\node at (0,5) [above]{$\scriptstyle{i-1}$};
\node at (0,-4) [above]{$\scriptstyle{j}$};
\node at (0,-6) [below]{$\scriptstyle{1}$};
\node at (-7,-7) [left]{$\scriptstyle{i}$};
\node at (-7,7) [left]{$\scriptstyle{j}$};
\node at (7,-7) [right]{$\scriptstyle{i}$};
\node at (7,7) [right]{$\scriptstyle{j}$};
}\ }
+q^{-\frac{1}{2}}q^{-\frac{j}{3}}\left(1-q^{j}\right)q^{-\frac{2}{3}(i-1)}
\mathord{ \tikz[baseline=-.6ex, scale=.1]{
\draw (-6,5) -- (-9,5);
\draw (-6,-5) -- (-9,-5);
\draw (6,5) -- (9,5);
\draw (6,-5) -- (9,-5);
\draw (0,2) to[out=east, in=west] (6,6);
\draw (0,0) -- (4,0);
\draw[->-=.7, white, double=black, double distance=0.4pt, ultra thick] (4,0) -- (6,-3);
\draw[white, double=black, double distance=0.4pt, ultra thick] (-6,-6) -- (-4,-6) to[out=east, in=west] (0,6);
\draw[->-=.5, white, double=black, double distance=0.4pt, ultra thick] (-6,6) to[out=east, in=west] (0,2);
\draw[-<-=.9, white, double=black, double distance=0.4pt, ultra thick] (-4,0) -- (0,0);
\draw[->-=.2, white, double=black, double distance=0.4pt, ultra thick] (0,6) to[out=east, in=west] (4,-6) -- (6,-6);
\draw[->-=.6, white, double=black, double distance=0.4pt, ultra thick] (-6,-3) -- (-4,0);
\draw[-<-=.6] (6,3) -- (4,0);
\draw[->-=.6] (-6,3) -- (-4,0);
\draw[fill=white] (-7,-7) rectangle (-6,-2);
\draw[fill=white] (-7,7) rectangle (-6,2);
\draw[fill=white] (7,-7) rectangle (6,-2);
\draw[fill=white] (7,7) rectangle (6,2);
\node at (2,6) [above]{$\scriptstyle{i-1}$};
\node at (-4,6) [above]{$\scriptstyle{j-1}$};
\node at (-7,-7) [left]{$\scriptstyle{i}$};
\node at (-7,7) [left]{$\scriptstyle{j}$};
\node at (7,-7) [right]{$\scriptstyle{i}$};
\node at (7,7) [right]{$\scriptstyle{j}$};
}\ }.
\end{align*}
As we remarked before Proposition~\ref{paralleleq}, 
the definition of $A_{2}$ clasp and its property says that
\[
\mathord{ \tikz[baseline=-.6ex, scale=.1]{
\draw (-6,5) -- (-9,5);
\draw (-6,-5) -- (-9,-5);
\draw (6,5) -- (9,5);
\draw (6,-5) -- (9,-5);
\draw (0,2) to[out=east, in=west] (6,6);
\draw (0,0) -- (4,0);
\draw[->-=.7, white, double=black, double distance=0.4pt, ultra thick] (4,0) -- (6,-3);
\draw[white, double=black, double distance=0.4pt, ultra thick] (-6,-6) -- (-4,-6) to[out=east, in=west] (0,6);
\draw[->-=.5, white, double=black, double distance=0.4pt, ultra thick] (-6,6) to[out=east, in=west] (0,2);
\draw[-<-=.9, white, double=black, double distance=0.4pt, ultra thick] (-4,0) -- (0,0);
\draw[->-=.2, white, double=black, double distance=0.4pt, ultra thick] (0,6) to[out=east, in=west] (4,-6) -- (6,-6);
\draw[->-=.6, white, double=black, double distance=0.4pt, ultra thick] (-6,-3) -- (-4,0);
\draw[-<-=.6] (6,3) -- (4,0);
\draw[->-=.6] (-6,3) -- (-4,0);
\draw[fill=white] (-7,-7) rectangle (-6,-2);
\draw[fill=white] (-7,7) rectangle (-6,2);
\draw[fill=white] (7,-7) rectangle (6,-2);
\draw[fill=white] (7,7) rectangle (6,2);
\node at (2,6) [above]{$\scriptstyle{i-1}$};
\node at (-4,6) [above]{$\scriptstyle{j-1}$};
\node at (-7,-7) [left]{$\scriptstyle{i}$};
\node at (-7,7) [left]{$\scriptstyle{j}$};
\node at (7,-7) [right]{$\scriptstyle{i}$};
\node at (7,7) [right]{$\scriptstyle{j}$};
}\ }
=
\mathord{ \tikz[baseline=-.6ex, scale=.1]{
\draw (-8,5) -- (-11,5);
\draw (-8,-5) -- (-11,-5);
\draw (8,5) -- (11,5);
\draw (8,-5) -- (11,-5);
\draw (0,2) to[out=east, in=west] (8,6);
\draw (0,0) -- (6,0);
\draw[->-=.7, white, double=black, double distance=0.4pt, ultra thick] (6,0) -- (8,-3);
\draw[white, double=black, double distance=0.4pt, ultra thick] (-8,-6) -- (-5,-6) to[out=east, in=west] (0,6);
\draw[->-=.5, white, double=black, double distance=0.4pt, ultra thick] (-8,6) to[out=east, in=west] (0,2);
\draw[-<-=.9, white, double=black, double distance=0.4pt, ultra thick] (-6,0) -- (0,0);
\draw[->-=.2, white, double=black, double distance=0.4pt, ultra thick] (0,6) to[out=east, in=west] (5,-6) -- (8,-6);
\draw[->-=.6, white, double=black, double distance=0.4pt, ultra thick] (-8,-3) -- (-6,0);
\draw[-<-=.6] (8,3) -- (6,0);
\draw[->-=.6] (-8,3) -- (-6,0);
\draw[fill=white] (-9,-7) rectangle (-8,-2);
\draw[fill=white] (-9,7) rectangle (-8,2);
\draw[fill=white] (9,-7) rectangle (8,-2);
\draw[fill=white] (9,7) rectangle (8,2);
\draw[fill=white] (-5,-7) rectangle (-4,1);
\draw[fill=white] (5,-7) rectangle (4,1);
\node at (2,6) [above]{$\scriptstyle{i-1}$};
\node at (-5,6) [above]{$\scriptstyle{j-1}$};
\node at (-9,-7) [left]{$\scriptstyle{i}$};
\node at (-9,7) [left]{$\scriptstyle{j}$};
\node at (9,-7) [right]{$\scriptstyle{i}$};
\node at (9,7) [right]{$\scriptstyle{j}$};
\node at (0,0) [below]{$\scriptstyle{1}$};
\fill[opacity=.2, magenta] (-2,0) circle[radius=1];
\fill[opacity=.2, magenta] (2,0) circle[radius=1];
}\ }
=(-q^{\frac{1}{6}})^{2(i-1)}
\mathord{ \tikz[baseline=-.6ex, scale=.1]{
\draw (-8,5) -- (-11,5);
\draw (-8,-5) -- (-11,-5);
\draw (8,5) -- (11,5);
\draw (8,-5) -- (11,-5);
\draw (0,2) to[out=east, in=west] (8,6);
\draw[-<-=.5] (-4,-5) -- (4,-5);
\draw (-6,0) -- (-4,0);
\draw (6,0) -- (4,0);
\draw[->-=.5] (-8,-6) -- (-4,-6);
\draw[-<-=.5] (8,-6) -- (4,-6);
\draw[->-=.7, white, double=black, double distance=0.4pt, ultra thick] (6,0) -- (8,-3);
\draw[white, double=black, double distance=0.4pt, ultra thick] (-4,-1) to[out=east, in=west] (0,6);
\draw[->-=.5, white, double=black, double distance=0.4pt, ultra thick] (-8,6) to[out=east, in=west] (0,2);
\draw[->-=.3, white, double=black, double distance=0.4pt, ultra thick] (0,6) to[out=east, in=west] (4,-1);
\draw[->-=.6, white, double=black, double distance=0.4pt, ultra thick] (-8,-3) -- (-6,0);
\draw[-<-=.6] (8,3) -- (6,0);
\draw[->-=.6] (-8,3) -- (-6,0);
\draw[fill=white] (-9,-7) rectangle (-8,-2);
\draw[fill=white] (-9,7) rectangle (-8,2);
\draw[fill=white] (9,-7) rectangle (8,-2);
\draw[fill=white] (9,7) rectangle (8,2);
\draw[fill=white] (-5,-7) rectangle (-4,1);
\draw[fill=white] (5,-7) rectangle (4,1);
\node at (2,6) [above]{$\scriptstyle{i-1}$};
\node at (-5,6) [above]{$\scriptstyle{j-1}$};
\node at (-9,-7) [left]{$\scriptstyle{i}$};
\node at (-9,7) [left]{$\scriptstyle{j}$};
\node at (9,-7) [right]{$\scriptstyle{i}$};
\node at (9,7) [right]{$\scriptstyle{j}$};
\node at (-6,-6) [below]{$\scriptstyle{i-1}$};
\node at (6,-6) [below]{$\scriptstyle{i-1}$};
\node at (0,-5) [below]{$\scriptstyle{1}$};
}\ }.\]
\end{proof}

\begin{proof}[Proof of Propositoin~\ref{paralleleq}]
Apply Lemma~\ref{lem3} to $\sigma_{d}(k,l)$.
We have to confirm the web in the second term is equal to $\sigma_{d}(k,l+1)$. 
It is deformed as follows.
\begin{align}
\mathord{ \tikz[baseline=-.6ex, scale=.1, yshift=2cm]{
\draw (-19,4) -- (-17,4);
\draw (-19,-8) -- (-17,-8);
\draw (19,4) -- (17,4);
\draw (19,-8) -- (17,-8);
\draw[-<-=.5] (-9,5) -- (-16,5);
\draw[->-=.5] (9,5) -- (16,5);
\draw[->-=.5] (-9,-2) -- (-13,-2);
\draw[->-=.5] (-16,2) to[out=east, in=north west] (-13,-2);
\draw[->-=.5] (-16,-6) to[out=east, in=south west] (-13,-2);
\draw[fill=white] (-12,-4) -- (-12,0) -- (-15,-2) -- cycle;
\draw[->-=.5] (9,-2) -- (13,-2);
\draw[->-=.5] (16,2) to[out=west, in=north east] (13,-2);
\draw[->-=.5] (16,-6) to[out=west, in=south east] (13,-2);
\draw[fill=white] (12,-4) -- (12,0) -- (15,-2) -- cycle;
\draw[->-=.5] (-16,-9) -- (-8,-9);
\draw[-<-=.5] (16,-9) -- (8,-9);
\draw (0,2) to[out=east, in=west] (8,6);
\draw[-<-=.5] (-4,-5) -- (4,-5);
\draw[->-=.5] (-8,-6) -- (-4,-6);
\draw[-<-=.5] (8,-6) -- (4,-6);
\draw[white, double=black, double distance=0.4pt, ultra thick] (-4,-1) to[out=east, in=west] (0,6);
\draw[->-=.5, white, double=black, double distance=0.4pt, ultra thick] (-8,6) to[out=east, in=west] (0,2);
\draw[->-=.3, white, double=black, double distance=0.4pt, ultra thick] (0,6) to[out=east, in=west] (4,-1);
\draw (-7,0) -- (-4,0);
\draw (7,0) -- (4,0);
\draw[-<-=.7] (8,-3) -- (7,-3) -- (7,0);
\draw[->-=.7] (-8,-3) -- (-7,-3) -- (-7,0);
\draw[-<-=.7] (8,3) -- (7,3) -- (7,0);
\draw[->-=.7] (-8,3) -- (-7,3) -- (-7,0);
\draw[->-=.5] (-8,-8) -- (8,-8);
\draw[-<-=.5] (-8,-10) -- (8,-10);
\draw[fill=white] (-9,-11) rectangle (-8,0);
\draw[fill=white] (-9,7) rectangle (-8,2);
\draw[fill=white] (9,-11) rectangle (8,0);
\draw[fill=white] (9,7) rectangle (8,2);
\draw[fill=white] (-5,-7) rectangle (-4,1);
\draw[fill=white] (5,-7) rectangle (4,1);
\node at (-2,-6) [above]{$\scriptstyle{1}$};
\node at (2,-9) [above]{$\scriptstyle{k}$};
\node at (-2,-9) [below]{$\scriptstyle{l}$};
\draw[fill=white] (-17,-11) rectangle (-16,-4);
\draw[fill=white] (-17,7) rectangle (-16,0);
\draw[fill=white] (17,-11) rectangle (16,-4);
\draw[fill=white] (17,7) rectangle (16,0);
}\ }
&=
\mathord{ \tikz[baseline=-.6ex, scale=.1, yshift=2cm]{
\draw (-19,4) -- (-17,4);
\draw (-19,-8) -- (-17,-8);
\draw (19,4) -- (17,4);
\draw (19,-8) -- (17,-8);
\draw[->-=.5] (9,5) -- (16,5);
\draw[->-=.7] (-9,-4) -- (-9,-2) -- (-13,-2);
\draw[->-=.5] (-16,2) to[out=east, in=north west] (-13,-2);
\draw[->-=.5] (-16,-6) to[out=east, in=south west] (-13,-2);
\draw[fill=white] (-12,-4) -- (-12,0) -- (-15,-2) -- cycle;
\draw[->-=.5] (9,-2) -- (13,-2);
\draw[->-=.5] (16,2) to[out=west, in=north east] (13,-2);
\draw[->-=.5] (16,-6) to[out=west, in=south east] (13,-2);
\draw[fill=white] (12,-4) -- (12,0) -- (15,-2) -- cycle;
\draw[->-=.5] (-16,-9) to[out=east, in=west] (-10,-7);
\draw[-<-=.5] (16,-9) -- (8,-9);
\draw (0,2) to[out=east, in=west] (8,6);
\draw[-<-=.5] (-4,-5) -- (4,-5);
\draw[->-=.5] (-8,-6) -- (-4,-6);
\draw[-<-=.5] (8,-6) -- (4,-6);
\draw[white, double=black, double distance=0.4pt, ultra thick] (-4,-1) to[out=east, in=west] (0,6);
\draw[->-=.5, white, double=black, double distance=0.4pt, ultra thick] (-16,6) -- (-8,6) to[out=east, in=west] (0,2);
\draw[->-=.3, white, double=black, double distance=0.4pt, ultra thick] (0,6) to[out=east, in=west] (4,-1);
\draw (-7,0) -- (-4,0);
\draw (7,0) -- (4,0);
\draw[-<-=.7] (8,-3) -- (7,-3) -- (7,0);
\draw[->-=.7] (-8,-5) -- (-7,-5) -- (-7,0);
\draw[-<-=.7] (8,3) -- (7,3) -- (7,0);
\draw[->-=.5] (-16,4) -- (-7,4) -- (-7,0);
\draw[->-=.5] (-8,-8) -- (8,-8);
\draw[-<-=.3] (-9,-9) -- (-9,-10) -- (8,-10);
\draw[fill=white] (9,-11) rectangle (8,0);
\draw[fill=white] (9,7) rectangle (8,2);
\draw[fill=white] (-5,-7) rectangle (-4,1);
\draw[fill=white] (5,-7) rectangle (4,1);
\node at (-2,-6) [above]{$\scriptstyle{1}$};
\node at (2,-9) [above]{$\scriptstyle{k}$};
\node at (-2,-9) [below]{$\scriptstyle{l}$};
\draw[fill=white] (-10,-4) rectangle (-8,-9);
\draw (-10,-4) -- (-8,-9);
\draw[fill=white] (-17,-11) rectangle (-16,-4);
\draw[fill=white] (-17,7) rectangle (-16,0);
\draw[fill=white] (17,-11) rectangle (16,-4);
\draw[fill=white] (17,7) rectangle (16,0);
}\ }
=
\mathord{ \tikz[baseline=-.6ex, scale=.1, yshift=2cm]{
\draw (-19,4) -- (-17,4);
\draw (-19,-8) -- (-17,-8);
\draw (19,4) -- (17,4);
\draw (19,-8) -- (17,-8);
\draw[->-=.5] (-16,4) -- (-10,4);
\draw[-<-=.5] (-10,4) to[out=east, in=west] (-4,0);
\draw[->-=.5] (9,5) -- (16,5);
\draw (-8,-5) -- (-8,-2) -- (-13,-2);
\draw[->-=.5] (-16,2) to[out=east, in=north west] (-13,-2);
\draw[->-=.5] (-16,-6) to[out=east, in=south west] (-13,-2);
\draw[fill=white] (-12,-4) -- (-12,0) -- (-15,-2) -- cycle;
\draw[->-=.5] (9,-2) -- (13,-2);
\draw[->-=.5] (16,2) to[out=west, in=north east] (13,-2);
\draw[->-=.5] (16,-6) to[out=west, in=south east] (13,-2);
\draw[fill=white] (12,-4) -- (12,0) -- (15,-2) -- cycle;
\draw[->-=.2, ->-=.9] (-16,-8) -- (-10,-5) -- (-10,4);
\draw[->-=.5] (-16,-10) to[out=east, in=west] (-9,-8);
\draw[-<-=.5] (16,-9) -- (8,-9);
\draw (0,2) to[out=east, in=west] (8,6);
\draw[-<-=.5] (-4,-5) -- (4,-5);
\draw[->-=.5] (-7,-6) -- (-4,-6);
\draw[-<-=.5] (8,-6) -- (4,-6);
\draw[white, double=black, double distance=0.4pt, ultra thick] (-4,-1) to[out=east, in=west] (0,6);
\draw[->-=.5, white, double=black, double distance=0.4pt, ultra thick] (-16,6) -- (-8,6) to[out=east, in=west] (0,2);
\draw[->-=.3, white, double=black, double distance=0.4pt, ultra thick] (0,6) to[out=east, in=west] (4,-1);
\draw (7,0) -- (4,0);
\draw[-<-=.7] (8,-3) -- (7,-3) -- (7,0);
\draw[-<-=.7] (8,3) -- (7,3) -- (7,0);
\draw[->-=.5] (-8,-8) -- (8,-8);
\draw[-<-=.3] (-8,-9) -- (-8,-10) -- (8,-10);
\draw[fill=white] (9,-11) rectangle (8,0);
\draw[fill=white] (9,7) rectangle (8,2);
\draw[fill=white] (-5,-7) rectangle (-4,1);
\draw[fill=white] (5,-7) rectangle (4,1);
\node at (-2,-6) [above]{$\scriptstyle{1}$};
\node at (2,-9) [above]{$\scriptstyle{k}$};
\node at (-2,-9) [below]{$\scriptstyle{l}$};
\draw[fill=white] (-11,-4) rectangle (-9,0);
\draw[fill=white] (-11,-4) -- (-9,0);
\draw[fill=white] (-9,-5) rectangle (-7,-9);
\draw (-9,-5) -- (-7,-9);
\draw[fill=white] (-17,-11) rectangle (-16,-4);
\draw[fill=white] (-17,7) rectangle (-16,0);
\draw[fill=white] (17,-11) rectangle (16,-4);
\draw[fill=white] (17,7) rectangle (16,0);
}\ }\label{uniontri}\\ 
&=\mathord{ \tikz[baseline=-.6ex, scale=.1, yshift=2cm]{
\draw (-19,4) -- (-17,4);
\draw (-19,-8) -- (-17,-8);
\draw (19,4) -- (17,4);
\draw (19,-8) -- (17,-8);
\draw[-<-=.5] (-12,0) -- (-4,0);
\draw[->-=.5] (9,5) -- (16,5);
\draw[->-=.3] (-8,-5) -- (-8,-2) -- (-13,-2);
\draw[->-=.5] (-16,2) to[out=east, in=north west] (-13,-2);
\draw[->-=.5] (-16,-6) to[out=east, in=south west] (-13,-2);
\draw[fill=white] (-12,-5) -- (-12,1) -- (-16,-2) -- cycle;
\draw[->-=.5] (9,-2) -- (13,-2);
\draw[->-=.5] (16,2) to[out=west, in=north east] (13,-2);
\draw[->-=.5] (16,-6) to[out=west, in=south east] (13,-2);
\draw[fill=white] (12,-4) -- (12,0) -- (15,-2) -- cycle;
\draw[->-=.5] (-16,-9) to[out=east, in=west] (-9,-7);
\draw[-<-=.5] (16,-9) -- (8,-9);
\draw (0,2) to[out=east, in=west] (8,6);
\draw[-<-=.5] (-4,-5) -- (4,-5);
\draw[->-=.5] (-7,-6) -- (-4,-6);
\draw[-<-=.5] (8,-6) -- (4,-6);
\draw[white, double=black, double distance=0.4pt, ultra thick] (-4,-1) to[out=east, in=west] (0,6);
\draw[->-=.5, white, double=black, double distance=0.4pt, ultra thick] (-16,6) -- (-8,6) to[out=east, in=west] (0,2);
\draw[->-=.3, white, double=black, double distance=0.4pt, ultra thick] (0,6) to[out=east, in=west] (4,-1);
\draw (7,0) -- (4,0);
\draw[-<-=.7] (8,-3) -- (7,-3) -- (7,0);
\draw[-<-=.7] (8,3) -- (7,3) -- (7,0);
\draw[->-=.5] (-8,-8) -- (8,-8);
\draw[-<-=.3] (-8,-9) -- (-8,-10) -- (8,-10);
\draw[fill=white] (9,-11) rectangle (8,0);
\draw[fill=white] (9,7) rectangle (8,2);
\draw[fill=white] (-5,-7) rectangle (-4,1);
\draw[fill=white] (5,-7) rectangle (4,1);
\node at (-2,-6) [above]{$\scriptstyle{1}$};
\node at (2,-9) [above]{$\scriptstyle{k}$};
\node at (-2,-9) [below]{$\scriptstyle{l}$};
\draw[fill=white] (-9,-5) rectangle (-7,-9);
\draw (-9,-5) -- (-7,-9);
\draw[fill=white] (-17,-11) rectangle (-16,-4);
\draw[fill=white] (-17,7) rectangle (-16,0);
\draw[fill=white] (17,-11) rectangle (16,-4);
\draw[fill=white] (17,7) rectangle (16,0);
}\ }
=\mathord{ \tikz[baseline=-.6ex, scale=.1, yshift=2cm]{
\draw (-19,4) -- (-17,4);
\draw (-19,-8) -- (-17,-8);
\draw (19,4) -- (17,4);
\draw (19,-8) -- (17,-8);
\draw[-<-=.5] (-10,0) -- (-4,0);
\draw[->-=.5] (9,5) -- (16,5);
\draw (-12,-2) -- (-11,-2);
\draw[->-=.5] (-8,-5) -- (-8,-3) -- (-10,-3);
\draw[->-=.5] (-16,2) to[out=east, in=north west] (-13,-2);
\draw[->-=.5] (-16,-6) to[out=east, in=south west] (-13,-2);
\draw[fill=white] (-12,-5) -- (-12,1) -- (-16,-2) -- cycle;
\draw[->-=.5] (9,-2) -- (13,-2);
\draw[->-=.5] (16,2) to[out=west, in=north east] (13,-2);
\draw[->-=.5] (16,-6) to[out=west, in=south east] (13,-2);
\draw[fill=white] (12,-4) -- (12,0) -- (15,-2) -- cycle;
\draw[->-=.5] (-16,-9) to[out=east, in=west] (-9,-7);
\draw[-<-=.5] (16,-9) -- (8,-9);
\draw (0,2) to[out=east, in=west] (8,6);
\draw[-<-=.5] (-4,-5) -- (4,-5);
\draw[->-=.5] (-7,-6) -- (-4,-6);
\draw[-<-=.5] (8,-6) -- (4,-6);
\draw[white, double=black, double distance=0.4pt, ultra thick] (-4,-1) to[out=east, in=west] (0,6);
\draw[->-=.5, white, double=black, double distance=0.4pt, ultra thick] (-16,6) -- (-8,6) to[out=east, in=west] (0,2);
\draw[->-=.3, white, double=black, double distance=0.4pt, ultra thick] (0,6) to[out=east, in=west] (4,-1);
\draw (7,0) -- (4,0);
\draw[-<-=.7] (8,-3) -- (7,-3) -- (7,0);
\draw[-<-=.7] (8,3) -- (7,3) -- (7,0);
\draw[->-=.5] (-8,-8) -- (8,-8);
\draw[-<-=.3] (-8,-9) -- (-8,-10) -- (8,-10);
\draw[fill=white] (9,-11) rectangle (8,0);
\draw[fill=white] (9,7) rectangle (8,2);
\draw[fill=white] (-5,-7) rectangle (-4,1);
\draw[fill=white] (5,-7) rectangle (4,1);
\node at (-2,-6) [above]{$\scriptstyle{1}$};
\node at (2,-9) [above]{$\scriptstyle{k}$};
\node at (-2,-9) [below]{$\scriptstyle{l}$};
\draw[fill=white] (-11,-9) rectangle (-10,1);
\draw[fill=white] (-9,-5) rectangle (-7,-9);
\draw (-9,-5) -- (-7,-9);
\draw[fill=white] (-17,-11) rectangle (-16,-4);
\draw[fill=white] (-17,7) rectangle (-16,0);
\draw[fill=white] (17,-11) rectangle (16,-4);
\draw[fill=white] (17,7) rectangle (16,0);
}\ }\notag\\
&=\mathord{ \tikz[baseline=-.6ex, scale=.1, yshift=2cm]{
\draw (-19,4) -- (-17,4);
\draw (-19,-8) -- (-17,-8);
\draw (19,4) -- (17,4);
\draw (19,-8) -- (17,-8);
\draw[-<-=.5] (-10,0) -- (-4,0);
\draw[->-=.5] (9,5) -- (16,5);
\draw (-11,-2) -- (-12,-2);
\draw[->-=.5] (-16,2) to[out=east, in=north west] (-13,-2);
\draw[->-=.5] (-16,-6) to[out=east, in=south west] (-13,-2);
\draw[fill=white] (-12,-5) -- (-12,1) -- (-16,-2) -- cycle;
\draw[->-=.5] (9,-2) -- (13,-2);
\draw[->-=.5] (16,2) to[out=west, in=north east] (13,-2);
\draw[->-=.5] (16,-6) to[out=west, in=south east] (13,-2);
\draw[fill=white] (12,-4) -- (12,0) -- (15,-2) -- cycle;
\draw[->-=.5] (-16,-9) to[out=east, in=west] (-10,-7);
\draw[-<-=.5] (16,-9) -- (8,-9);
\draw (0,2) to[out=east, in=west] (8,6);
\draw[-<-=.5] (-4,-5) -- (4,-5);
\draw[->-=.5] (-10,-6) -- (-4,-6);
\draw[-<-=.5] (8,-6) -- (4,-6);
\draw[white, double=black, double distance=0.4pt, ultra thick] (-4,-1) to[out=east, in=west] (0,6);
\draw[->-=.5, white, double=black, double distance=0.4pt, ultra thick] (-16,6) -- (-8,6) to[out=east, in=west] (0,2);
\draw[->-=.3, white, double=black, double distance=0.4pt, ultra thick] (0,6) to[out=east, in=west] (4,-1);
\draw (7,0) -- (4,0);
\draw[-<-=.7] (8,-3) -- (7,-3) -- (7,0);
\draw[-<-=.7] (8,3) -- (7,3) -- (7,0);
\draw[->-=.5] (-10,-8) -- (8,-8);
\draw[-<-=.3] (-10,-10) -- (8,-10);
\draw[fill=white] (9,-11) rectangle (8,0);
\draw[fill=white] (9,7) rectangle (8,2);
\draw[fill=white] (-5,-7) rectangle (-4,1);
\draw[fill=white] (5,-7) rectangle (4,1);
\node at (-2,-6) [above]{$\scriptstyle{1}$};
\node at (2,-9) [above]{$\scriptstyle{k}$};
\node at (-2,-9) [below]{$\scriptstyle{l}$};
\draw[fill=white] (-11,-11) rectangle (-10,1);
\draw[fill=white] (-17,-11) rectangle (-16,-4);
\draw[fill=white] (-17,7) rectangle (-16,0);
\draw[fill=white] (17,-11) rectangle (16,-4);
\draw[fill=white] (17,7) rectangle (16,0);
}\ }
=\mathord{ \tikz[baseline=-.6ex, scale=.1, yshift=2cm]{
\draw (-19,4) -- (-17,4);
\draw (-19,-8) -- (-17,-8);
\draw (19,4) -- (17,4);
\draw (19,-8) -- (17,-8);
\draw[->-=.5] (9,5) -- (16,5);
\draw (-11,-2) -- (-12,-2);
\draw[->-=.5] (-16,2) to[out=east, in=north west] (-13,-2);
\draw[->-=.5] (-16,-6) to[out=east, in=south west] (-13,-2);
\draw[fill=white] (-12,-5) -- (-12,1) -- (-16,-2) -- cycle;
\draw[->-=.5] (9,-2) -- (13,-2);
\draw[->-=.5] (16,2) to[out=west, in=north east] (13,-2);
\draw[->-=.5] (16,-6) to[out=west, in=south east] (13,-2);
\draw[fill=white] (12,-4) -- (12,0) -- (15,-2) -- cycle;
\draw[->-=.5] (-16,-9) to[out=east, in=west] (-10,-7);
\draw[-<-=.5] (16,-9) -- (8,-9);
\draw (0,2) to[out=east, in=west] (8,6);
\draw[-<-=.5] (-10,-5) -- (4,-5);
\draw[-<-=.5] (8,-6) -- (4,-6);
\draw[white, double=black, double distance=0.4pt, ultra thick] (-10,-1) -- (-4,-1) to[out=east, in=west] (0,6);
\draw[->-=.5, white, double=black, double distance=0.4pt, ultra thick] (-16,6) -- (-8,6) to[out=east, in=west] (0,2);
\draw[->-=.3, white, double=black, double distance=0.4pt, ultra thick] (0,6) to[out=east, in=west] (4,-1);
\draw (7,0) -- (4,0);
\draw[-<-=.7] (8,-3) -- (7,-3) -- (7,0);
\draw[-<-=.7] (8,3) -- (7,3) -- (7,0);
\draw[->-=.5] (-10,-8) -- (8,-8);
\draw[-<-=.3] (-10,-10) -- (8,-10);
\draw[fill=white] (9,-11) rectangle (8,0);
\draw[fill=white] (9,7) rectangle (8,2);
\draw[fill=white] (5,-7) rectangle (4,1);
\node at (-2,-6) [above]{$\scriptstyle{1}$};
\node at (2,-9) [above]{$\scriptstyle{k}$};
\node at (-2,-9) [below]{$\scriptstyle{l}$};
\draw[fill=white] (-11,-11) rectangle (-10,1);
\draw[fill=white] (-17,-11) rectangle (-16,-4);
\draw[fill=white] (-17,7) rectangle (-16,0);
\draw[fill=white] (17,-11) rectangle (16,-4);
\draw[fill=white] (17,7) rectangle (16,0);
}\ }.\notag
\end{align}
In the first equation, 
we remove an $A_{2}$ clasp by using its definition and property of the triangle in Lemma~\ref{claspformula}.
One can also apply the same deformation on the right side of the web and obtain the following equation:
\[
\mathord{ \tikz[baseline=-.6ex, scale=.1, yshift=2cm]{
\draw (-19,4) -- (-17,4);
\draw (-19,-8) -- (-17,-8);
\draw (19,4) -- (17,4);
\draw (19,-8) -- (17,-8);
\draw[-<-=.5] (-9,5) -- (-16,5);
\draw[->-=.5] (9,5) -- (16,5);
\draw[->-=.5] (-9,-2) -- (-13,-2);
\draw[->-=.5] (-16,2) to[out=east, in=north west] (-13,-2);
\draw[->-=.5] (-16,-6) to[out=east, in=south west] (-13,-2);
\draw[fill=white] (-12,-4) -- (-12,0) -- (-15,-2) -- cycle;
\draw[->-=.5] (9,-2) -- (13,-2);
\draw[->-=.5] (16,2) to[out=west, in=north east] (13,-2);
\draw[->-=.5] (16,-6) to[out=west, in=south east] (13,-2);
\draw[fill=white] (12,-4) -- (12,0) -- (15,-2) -- cycle;
\draw[->-=.5] (-16,-9) -- (-8,-9);
\draw[-<-=.5] (16,-9) -- (8,-9);
\draw (0,2) to[out=east, in=west] (8,6);
\draw[-<-=.5] (-4,-5) -- (4,-5);
\draw[->-=.5] (-8,-6) -- (-4,-6);
\draw[-<-=.5] (8,-6) -- (4,-6);
\draw[white, double=black, double distance=0.4pt, ultra thick] (-4,-1) to[out=east, in=west] (0,6);
\draw[->-=.5, white, double=black, double distance=0.4pt, ultra thick] (-8,6) to[out=east, in=west] (0,2);
\draw[->-=.3, white, double=black, double distance=0.4pt, ultra thick] (0,6) to[out=east, in=west] (4,-1);
\draw (-7,0) -- (-4,0);
\draw (7,0) -- (4,0);
\draw[-<-=.7] (8,-3) -- (7,-3) -- (7,0);
\draw[->-=.7] (-8,-3) -- (-7,-3) -- (-7,0);
\draw[-<-=.7] (8,3) -- (7,3) -- (7,0);
\draw[->-=.7] (-8,3) -- (-7,3) -- (-7,0);
\draw[->-=.5] (-8,-8) -- (8,-8);
\draw[-<-=.5] (-8,-10) -- (8,-10);
\draw[fill=white] (-9,-11) rectangle (-8,0);
\draw[fill=white] (-9,7) rectangle (-8,2);
\draw[fill=white] (9,-11) rectangle (8,0);
\draw[fill=white] (9,7) rectangle (8,2);
\draw[fill=white] (-5,-7) rectangle (-4,1);
\draw[fill=white] (5,-7) rectangle (4,1);
\node at (-2,-6) [above]{$\scriptstyle{1}$};
\node at (2,-9) [above]{$\scriptstyle{k}$};
\node at (-2,-9) [below]{$\scriptstyle{l}$};
\draw[fill=white] (-17,-11) rectangle (-16,-4);
\draw[fill=white] (-17,7) rectangle (-16,0);
\draw[fill=white] (17,-11) rectangle (16,-4);
\draw[fill=white] (17,7) rectangle (16,0);
}\ }
=\mathord{ \tikz[baseline=-.6ex, scale=.1]{
\draw (-17,5) -- (-14,5);
\draw (-17,-5) -- (-14,-5);
\draw (17,5) -- (14,5);
\draw (17,-5) -- (14,-5);
\draw[->-=.5] (-14,6) -- (-6,6);
\draw[->-=.5] (-14,-6) -- (-6,-6);
\draw[->-=.5] (-6,0) -- (-11,0);
\draw[->-=.5] (-14,4) to[out=east, in=north east] (-11,0);
\draw[->-=.5] (-14,-4) to[out=east, in=south east] (-11,0);
\draw[fill=white] (-12,0) -- (-9,2) -- (-9,-2) -- cycle;
\draw[-<-=.5] (14,6) -- (6,6);
\draw[-<-=.5] (14,-6) -- (6,-6);
\draw[-<-=.5] (6,0) -- (11,0);
\draw[-<-=.5] (14,4) to[out=west, in=north west] (11,0);
\draw[-<-=.5] (14,-4) to[out=west, in=south west] (11,0);
\draw[fill=white] (12,0) -- (9,2) -- (9,-2) -- cycle;
\draw[-<-=.5] (-6,-6) -- (6,-6);
\draw[->-=.5] (-6,-4) -- (6,-4);
\draw[->-=.9, white, double=black, double distance=0.4pt, ultra thick] (-6,-1) to[out=east, in=west] (0,6);
\draw[->-=.9, white, double=black, double distance=0.4pt, ultra thick] (-6,6) to[out=east, in=west] (0,-1);
\draw[white, double=black, double distance=0.4pt, ultra thick] (0,-1) to[out=east, in=west] (6,6);
\draw[white, double=black, double distance=0.4pt, ultra thick] (0,6) to[out=east, in=west] (6,-1);
\draw[fill=white] (-7,8) rectangle (-6,4);
\draw[fill=white] (-7,-8) rectangle (-6,2);
\draw[fill=white] (7,8) rectangle (6,4);
\draw[fill=white] (7,-8) rectangle (6,2);
\node at (3,-5)[above]{$\scriptstyle{k}$};
\node at (3,-5)[below]{$\scriptstyle{l+1}$};
\draw[fill=white] (-14,-8) rectangle (-15,-2);
\draw[fill=white] (-14,8) rectangle (-15,2);
\draw[fill=white] (14,-8) rectangle (15,-2);
\draw[fill=white] (14,8) rectangle (15,2);
}\ }
=\sigma_{d}(k,l+1).
\]
\end{proof}

\begin{THM}[parallel $m$-full twist formula]\label{parallelmfull}
Let $\underline{k}=(k_{1},\dots,k_{m})$ be an $m$-tuple of integers, $k_{0}=d=\min\{s,t\}$, and $\delta=\left| s-t \right|$.
\begin{align*}
\mathord{\ \tikz[baseline=-.6ex]{
\begin{scope}[xshift=-1cm]
\draw 
(-.5,.4) -- +(-.2,0)
(-.5,-.4) -- +(-.2,0);
\draw[->-=1, white, double=black, double distance=0.4pt, ultra thick] 
(-.5,-.4) to[out=east, in=west] (.0,.4);
\draw[->-=1, white, double=black, double distance=0.4pt, ultra thick] 
(-.5,.4) to[out=east, in=west] (.0,-.4);
\draw[white, double=black, double distance=0.4pt, ultra thick] 
(0,-.4) to[out=east, in=west] (.5,.4);
\draw[white, double=black, double distance=0.4pt, ultra thick] 
(0,.4) to[out=east, in=west] (.5,-.4);
\draw[fill=white] (-.5,-.6) rectangle +(-.1,.4);
\draw[fill=white] (-.5,.6) rectangle +(-.1,-.4);
\node at (-.5,-.6)[left]{$\scriptstyle{s}$};
\node at (-.5,.6)[left]{$\scriptstyle{t}$};
\end{scope}
\node at (.0,.0){$\cdots$};
\node at (.0,-.4)[below]{$\scriptstyle{m\text{ full twists}}$};
\begin{scope}[xshift=1cm]
\draw
(.5,-.4) -- +(.2,0)
(.5,.4) -- +(.2,0);
\draw[->-=1, white, double=black, double distance=0.4pt, ultra thick] 
(-.5,-.4) to[out=east, in=west] (.0,.4);
\draw[->-=1, white, double=black, double distance=0.4pt, ultra thick] 
(-.5,.4) to[out=east, in=west] (.0,-.4);
\draw[white, double=black, double distance=0.4pt, ultra thick] 
(0,-.4) to[out=east, in=west] (.5,.4);
\draw[white, double=black, double distance=0.4pt, ultra thick] 
(0,.4) to[out=east, in=west] (.5,-.4);
\draw[fill=white] (.5,-.6) rectangle +(.1,.4);
\draw[fill=white] (.5,.6) rectangle +(.1,-.4);
\node at (.5,-.6)[right]{$\scriptstyle{s}$};
\node at (.5,.6)[right]{$\scriptstyle{t}$};
\end{scope}
}\ }
&=q^{-\frac{m}{3}k_{0}(k_{0}+\delta)-mk_{0}}
\sum_{k_{0}\geq k_{1}\geq\cdots\geq k_{m}\geq 0}D(\underline{k})
\mathord{\ \tikz[baseline=-.6ex, scale=.1]{
\draw (-6,5) -- (-9,5);
\draw (-6,-5) -- (-9,-5);
\draw (6,5) -- (9,5);
\draw (6,-5) -- (9,-5);
\draw[->-=.5] (-6,6) -- (6,6);
\draw[->-=.5] (-6,-6) -- (6,-6);
\draw[->-=.5] (-6,4) to[out=east, in=north west] (0,0);
\draw[->-=.5] (-6,-4) to[out=east, in=south west] (0,0);
\draw[->-=.7] (0,0) to[out=north east, in=west] (6,4);
\draw[->-=.7] (0,0) to[out=south east, in=west] (6,-4);
\draw[fill=white] (-3,0) -- (0,3) -- (3,0) -- (0,-3) -- cycle;
\draw (0,3) -- (0,-3);
\draw[fill=white] (-7,-7) rectangle (-6,-3);
\draw[fill=white] (-7,7) rectangle (-6,3);
\draw[fill=white] (7,-7) rectangle (6,-3);
\draw[fill=white] (7,7) rectangle (6,3);
\node at (-7,-7)[left]{$\scriptstyle{s}$};
\node at (-7,7)[left]{$\scriptstyle{t}$};
\node at (7,-7)[right]{$\scriptstyle{s}$};
\node at (7,7)[right]{$\scriptstyle{t}$};
\node at (0,6)[above]{$\scriptstyle{t-d+k_{m}}$};
\node at (0,-6)[below]{$\scriptstyle{s-d+k_{m}}$};
\node at (-3,4)[below left]{$\scriptstyle{d-k_{m}}$};
\node at (-3,-4)[above left]{$\scriptstyle{d-k_{m}}$};
\node at (3,4)[below right]{$\scriptstyle{d-k_{m}}$};
\node at (3,-4)[above right]{$\scriptstyle{d-k_{m}}$};
}\ },
\end{align*}
where 
\[
D(\underline{k})=q^{\sum_{i=1}^{m}k_{i}(k_{i}+\delta)+k_{i}}q^{\frac{1}{2}(k_{0}-k_{m})}\frac{(q)_{k_{0}+\delta}}{(q)_{k_{m}+\delta}}{d \choose k_{1}',k_{2}',\dots,k_{m}',k_{m}}_{q}
\]
and $k_{i+1}'=k_i-k_{i+1}$ for $i=0,1,\dots,m-1$.
\end{THM}

For a positive integer $m$, 
we denote a right-handed $m$ half twist by
\[
\mathord{\ \tikz[baseline=-.6ex, scale=.1]{
\draw (-5,2) -- (5,2);
\draw (-5,-2) -- (5,-2);
\draw[fill=white] (-3,-3) rectangle (3,3);
\node at (0,0){$\small{m}$};
}\ }
=
\mathord{\ \tikz[baseline=-.6ex, scale=.1]{
\begin{scope}[xshift=-5cm]
\draw (-5,2) -- +(-2,0);
\draw (-5,-2) -- +(-2,0);
\draw[white, double=black, double distance=0.4pt, ultra thick] 
(-5,-2) to[out=east, in=west] (0,2);
\draw[white, double=black, double distance=0.4pt, ultra thick] 
(-5,2) to[out=east, in=west] (0,-2);
\end{scope}
\node at (0,0){$\cdots$};
\node at (0,-2)[below]{$\scriptstyle{m\text{ crossings}}$};
\begin{scope}[xshift=5cm]
\draw
(5,-2) -- +(2,0)
(5,2) -- +(2,0);
\draw[white, double=black, double distance=0.4pt, ultra thick] 
(0,-2) to[out=east, in=west] (5,2);
\draw[white, double=black, double distance=0.4pt, ultra thick] 
(0,2) to[out=east, in=west] (5,-2);
\end{scope}
}\ }
\]

\begin{proof}
By sliding the right-hand triangle,
\begin{align*}
&\mathord{\ \tikz[baseline=-.6ex, scale=.1]{
\draw (-6,5) -- (-9,5);
\draw (-6,-5) -- (-9,-5);
\draw[->-=.5] (-6,6) -- (6,6);
\draw[->-=.5] (-6,-6) -- (6,-6);
\draw[->-=.5] (-6,4) to[out=east, in=north west] (0,0);
\draw[->-=.5] (-6,-4) to[out=east, in=south west] (0,0);
\draw[->-=.7] (0,0) to[out=north east, in=west] (6,4);
\draw[->-=.7] (0,0) to[out=south east, in=west] (6,-4);
\draw[fill=white] (-3,0) -- (0,3) -- (3,0) -- (0,-3) -- cycle;
\draw (0,3) -- (0,-3);
\begin{scope}[xshift=16cm]
\draw[white, double=black, double distance=0.4pt, ultra thick] (-6,-5) to[out=east, in=west] (0,5);
\draw[white, double=black, double distance=0.4pt, ultra thick] (-6,5) to[out=east, in=west] (0,-5);
\draw[white, double=black, double distance=0.4pt, ultra thick] (0,-5) to[out=east, in=west] (6,5);
\draw[white, double=black, double distance=0.4pt, ultra thick] (0,5) to[out=east, in=west] (6,-5);
\draw[->-=.7] (-9,5) -- (-6,5);
\draw[->-=.7] (-9,-5) -- (-6,-5);
\draw[->-=.5] (6,5) -- (9,5);
\draw[->-=.5] (6,-5) -- (9,-5);
\draw (10,5) -- (12,5);
\draw (10,-5) -- (12,-5);
\draw[fill=white] (9,-7) rectangle (10,-3);
\draw[fill=white] (9,7) rectangle (10,3);
\node at (10,-7)[right]{$\scriptstyle{s}$};
\node at (10,7)[right]{$\scriptstyle{t}$};
\end{scope}
\draw[fill=white] (-7,-7) rectangle (-6,-3);
\draw[fill=white] (-7,7) rectangle (-6,3);
\draw[fill=white] (7,-7) rectangle (6,-3);
\draw[fill=white] (7,7) rectangle (6,3);
\node at (-7,-7)[left]{$\scriptstyle{s}$};
\node at (-7,7)[left]{$\scriptstyle{t}$};
\node at (7,-7)[right]{$\scriptstyle{s}$};
\node at (7,7)[right]{$\scriptstyle{t}$};
\node at (0,6)[above]{$\scriptstyle{k+(t-d)}$};
\node at (0,-6)[below]{$\scriptstyle{k+(s-d)}$};
\node at (-3,4)[below left]{$\scriptstyle{d-k}$};
\node at (-3,-4)[above left]{$\scriptstyle{d-k}$};
\node at (3,4)[below right]{$\scriptstyle{d-k}$};
\node at (3,-4)[above right]{$\scriptstyle{d-k}$};
}\ }\\
&=\left((-1)^{d-k}q^{-\frac{1}{6}(d-k)^{2}-\frac{1}{2}(d-k)}q^{-\frac{1}{3}(d-k)(k+s-d)}q^{-\frac{1}{3}(d-k)(k+t-d)}\right)^{2}\\
&\quad\times\mathord{\ \tikz[baseline=-.6ex, scale=.1]{
\begin{scope}[xshift=0cm]
\draw[white, double=black, double distance=0.4pt, ultra thick] (-9,-6) to[out=east, in=south] (-2,0) to[out=north, in=west](0,6);
\draw[-<-=.2, white, double=black, double distance=0.4pt, ultra thick] (-9,0) -- (0,0);
\draw[white, double=black, double distance=0.4pt, ultra thick] (-9,6) to[out=east, in=west] (0,-6);
\draw[white, double=black, double distance=0.4pt, ultra thick] (0,-6) to[out=east, in=west] (9,6);
\draw[-<-=.8, white, double=black, double distance=0.4pt, ultra thick] (0,0) -- (9,0);
\draw[white, double=black, double distance=0.4pt, ultra thick] (0,6) to[out=east, in=north] (2,0) to[out=south, in=west] (9,-6);
\end{scope}
\begin{scope}[xshift=-9cm]
\draw (-6,5) -- (-9,5);
\draw (-6,-5) -- (-9,-5);
\draw[->-=.5] (-6,6) -- (0,6);
\draw[->-=.5] (-6,-6) -- (0,-6);
\draw[->-=.5] (-6,4) to[out=east, in=north west] (0,0);
\draw[->-=.5] (-6,-4) to[out=east, in=south west] (0,0);
\draw[fill=white] (-3,0) -- (0,3) -- (0,-3) -- cycle;
\draw[fill=white] (-7,-7) rectangle (-6,-3);
\draw[fill=white] (-7,7) rectangle (-6,3);
\node at (-3,4)[below left]{$\scriptstyle{d-k}$};
\node at (-3,-4)[above left]{$\scriptstyle{d-k}$};
\node at (-7,-7)[left]{$\scriptstyle{s}$};
\node at (-7,7)[left]{$\scriptstyle{t}$};
\end{scope}
\begin{scope}[xshift=9cm]
\draw (6,5) -- (9,5);
\draw (6,-5) -- (9,-5);
\draw[->-=.5] (0,6) -- (6,6);
\draw[->-=.5] (0,-6) -- (6,-6);
\draw[->-=.7] (0,0) to[out=north east, in=west] (6,4);
\draw[->-=.7] (0,0) to[out=south east, in=west] (6,-4);
\draw[fill=white] (3,0) -- (0,3) -- (0,-3) -- cycle;
\node at (7,-7)[right]{$\scriptstyle{s}$};
\node at (7,7)[right]{$\scriptstyle{t}$};
\draw[fill=white] (7,-7) rectangle (6,-3);
\draw[fill=white] (7,7) rectangle (6,3);
\node at (3,4)[below right]{$\scriptstyle{d-k}$};
\node at (3,-4)[above right]{$\scriptstyle{d-k}$};
\end{scope}
\node at (0,6)[above]{$\scriptstyle{k+(s-d)}$};
\node at (0,-6)[below]{$\scriptstyle{k+(t-d)}$};
}\ }
\end{align*}
and
\begin{align*}
\mathord{\ \tikz[baseline=-.6ex, scale=.1]{
\begin{scope}[xshift=0cm]
\draw[white, double=black, double distance=0.4pt, ultra thick] (-9,-6) to[out=east, in=south] (-2,0) to[out=north, in=west](0,6);
\draw[-<-=.2, white, double=black, double distance=0.4pt, ultra thick] (-9,0) -- (0,0);
\draw[white, double=black, double distance=0.4pt, ultra thick] (-9,6) to[out=east, in=west] (0,-6);
\draw[white, double=black, double distance=0.4pt, ultra thick] (0,-6) to[out=east, in=west] (9,6);
\draw[-<-=.8, white, double=black, double distance=0.4pt, ultra thick] (0,0) -- (9,0);
\draw[white, double=black, double distance=0.4pt, ultra thick] (0,6) to[out=east, in=north] (2,0) to[out=south, in=west] (9,-6);
\end{scope}
\begin{scope}[xshift=-9cm]
\draw (-6,5) -- (-9,5);
\draw (-6,-5) -- (-9,-5);
\draw[->-=.5] (-6,6) -- (0,6);
\draw[->-=.5] (-6,-6) -- (0,-6);
\draw[->-=.5] (-6,4) to[out=east, in=north west] (0,0);
\draw[->-=.5] (-6,-4) to[out=east, in=south west] (0,0);
\draw[fill=white] (-3,0) -- (0,3) -- (0,-3) -- cycle;
\draw[fill=white] (-7,-7) rectangle (-6,-3);
\draw[fill=white] (-7,7) rectangle (-6,3);
\node at (-3,4)[below left]{$\scriptstyle{d-k}$};
\node at (-3,-4)[above left]{$\scriptstyle{d-k}$};
\node at (-7,-7)[left]{$\scriptstyle{s}$};
\node at (-7,7)[left]{$\scriptstyle{t}$};
\end{scope}
\begin{scope}[xshift=9cm]
\draw (6,5) -- (9,5);
\draw (6,-5) -- (9,-5);
\draw[->-=.5] (0,6) -- (6,6);
\draw[->-=.5] (0,-6) -- (6,-6);
\draw[->-=.7] (0,0) to[out=north east, in=west] (6,4);
\draw[->-=.7] (0,0) to[out=south east, in=west] (6,-4);
\draw[fill=white] (3,0) -- (0,3) -- (0,-3) -- cycle;
\node at (7,-7)[right]{$\scriptstyle{s}$};
\node at (7,7)[right]{$\scriptstyle{t}$};
\draw[fill=white] (7,-7) rectangle (6,-3);
\draw[fill=white] (7,7) rectangle (6,3);
\node at (3,4)[below right]{$\scriptstyle{d-k}$};
\node at (3,-4)[above right]{$\scriptstyle{d-k}$};
\end{scope}
\node at (0,6)[above]{$\scriptstyle{k+(s-d)}$};
\node at (0,-6)[below]{$\scriptstyle{k+(t-d)}$};
}\ }
&=\mathord{\ \tikz[baseline=-.6ex, scale=.1]{
\begin{scope}[xshift=0cm]
\draw[white, double=black, double distance=0.4pt, ultra thick] (-9,-6) to[out=east, in=south] (-2,0) to[out=north, in=west](0,6);
\draw[-<-=.2, white, double=black, double distance=0.4pt, ultra thick] (-9,0) -- (0,0);
\draw[white, double=black, double distance=0.4pt, ultra thick] (-9,6) to[out=east, in=west] (0,-6);
\draw[white, double=black, double distance=0.4pt, ultra thick] (0,-6) to[out=east, in=west] (9,6);
\draw[white, double=black, double distance=0.4pt, ultra thick] (0,0) -- (9,0);
\draw[white, double=black, double distance=0.4pt, ultra thick] (0,6) to[out=east, in=north] (2,0) to[out=south, in=west] (9,-6);
\end{scope}
\begin{scope}[xshift=-9cm]
\draw (-6,5) -- (-9,5);
\draw (-6,-5) -- (-9,-5);
\draw[->-=.9] (-6,6) -- (0,6);
\draw[->-=.5] (-6,-6) -- (0,-6);
\draw[->-=.5] (-6,4) to[out=east, in=north west] (0,0);
\draw[->-=.5] (-6,-4) to[out=east, in=south west] (0,0);
\draw[fill=white] (-3,0) -- (0,3) -- (0,-3) -- cycle;
\draw[fill=white] (-7,-7) rectangle (-6,-3);
\draw[fill=white] (-7,7) rectangle (-6,3);
\node at (-3,4)[below left]{$\scriptstyle{d-k}$};
\node at (-3,-4)[above left]{$\scriptstyle{d-k}$};
\node at (-7,-7)[left]{$\scriptstyle{s}$};
\node at (-7,7)[left]{$\scriptstyle{t}$};
\end{scope}
\begin{scope}[xshift=9cm]
\draw (6,5) -- (9,5);
\draw (6,-5) -- (9,-5);
\draw[->-=.5] (0,6) -- (6,6);
\draw[->-=.5] (0,-6) -- (6,-6);
\draw[->-=.7] (0,0) to[out=north east, in=west] (6,4);
\draw[->-=.7] (0,0) to[out=south east, in=west] (6,-4);
\draw[fill=white] (3,0) -- (0,3) -- (0,-3) -- cycle;
\node at (7,-7)[right]{$\scriptstyle{s}$};
\node at (7,7)[right]{$\scriptstyle{t}$};
\draw[fill=white] (-1,-7) rectangle (-2,3);
\draw[fill=white] (7,-7) rectangle (6,-3);
\draw[fill=white] (7,7) rectangle (6,3);
\node at (3,4)[below right]{$\scriptstyle{d-k}$};
\node at (3,-4)[above right]{$\scriptstyle{d-k}$};
\end{scope}
\node at (0,6)[above]{$\scriptstyle{k+(s-d)}$};
\node at (0,-6)[below]{$\scriptstyle{k+(t-d)}$};
}\ }\\
&=\left(-q^{\frac{1}{6}}\right)^{2(k+s-d)(d-k)}
\mathord{\ \tikz[baseline=-.6ex, scale=.1]{
\begin{scope}[xshift=0cm]
\draw[white, double=black, double distance=0.4pt, ultra thick] (-9,-6) to[out=east, in=west](0,-2);
\draw[-<-=.2, white, double=black, double distance=0.4pt, ultra thick] (-9,0) -- (0,0);
\draw[white, double=black, double distance=0.4pt, ultra thick] (-9,6) to[out=east, in=west] (0,-6);
\draw[white, double=black, double distance=0.4pt, ultra thick] (0,-6) to[out=east, in=west] (9,6);
\draw[white, double=black, double distance=0.4pt, ultra thick] (0,0) -- (9,0);
\draw[white, double=black, double distance=0.4pt, ultra thick] (0,-2) to[out=east, in=west] (9,-6);
\end{scope}
\begin{scope}[xshift=-9cm]
\draw (-6,5) -- (-9,5);
\draw (-6,-5) -- (-9,-5);
\draw[->-=.5] (-6,6) -- (0,6);
\draw[->-=.5] (-6,-6) -- (0,-6);
\draw[->-=.5] (-6,4) to[out=east, in=north west] (0,0);
\draw[->-=.5] (-6,-4) to[out=east, in=south west] (0,0);
\draw[fill=white] (-3,0) -- (0,3) -- (0,-3) -- cycle;
\draw[fill=white] (-7,-7) rectangle (-6,-3);
\draw[fill=white] (-7,7) rectangle (-6,3);
\node at (-3,4)[below left]{$\scriptstyle{d-k}$};
\node at (-3,-4)[above left]{$\scriptstyle{d-k}$};
\node at (-7,-7)[left]{$\scriptstyle{s}$};
\node at (-7,7)[left]{$\scriptstyle{t}$};
\end{scope}
\begin{scope}[xshift=9cm]
\draw (6,5) -- (9,5);
\draw (6,-5) -- (9,-5);
\draw[->-=.5] (0,6) -- (6,6);
\draw[->-=.5] (0,-6) -- (6,-6);
\draw[->-=.7] (0,0) to[out=north east, in=west] (6,4);
\draw[->-=.7] (0,0) to[out=south east, in=west] (6,-4);
\draw[fill=white] (3,0) -- (0,3) -- (0,-3) -- cycle;
\node at (7,-7)[right]{$\scriptstyle{s}$};
\node at (7,7)[right]{$\scriptstyle{t}$};
\draw[fill=white] (-1,-7) rectangle (-2,3);
\draw[fill=white] (7,-7) rectangle (6,-3);
\draw[fill=white] (7,7) rectangle (6,3);
\node at (3,4)[below right]{$\scriptstyle{d-k}$};
\node at (3,-4)[above right]{$\scriptstyle{d-k}$};
\end{scope}
\node at (-9,6)[above]{$\scriptstyle{k+(t-d)}$};
\node at (-9,-6)[below]{$\scriptstyle{k+(s-d)}$};
}\ }\\
&=q^{\frac{1}{3}(d-k)(k+s-d)}
\mathord{\ \tikz[baseline=-.6ex, scale=.1]{
\begin{scope}[xshift=0cm]
\draw[white, double=black, double distance=0.4pt, ultra thick] (-9,-6) to[out=east, in=west](0,-2);
\draw[-<-=.9, white, double=black, double distance=0.4pt, ultra thick] (-9,0) -- (0,0);
\draw[white, double=black, double distance=0.4pt, ultra thick] (-9,6) to[out=east, in=west] (0,-6);
\draw[white, double=black, double distance=0.4pt, ultra thick] (0,-6) to[out=east, in=west] (9,6);
\draw[white, double=black, double distance=0.4pt, ultra thick] (0,0) -- (9,0);
\draw[white, double=black, double distance=0.4pt, ultra thick] (0,-2) to[out=east, in=west] (9,-6);
\end{scope}
\begin{scope}[xshift=-9cm]
\draw (-6,5) -- (-9,5);
\draw (-6,-5) -- (-9,-5);
\draw[->-=.5] (-6,6) -- (0,6);
\draw[->-=.5] (-6,-6) -- (0,-6);
\draw[->-=.5] (-6,4) to[out=east, in=north west] (0,0);
\draw[->-=.5] (-6,-4) to[out=east, in=south west] (0,0);
\draw[fill=white] (-3,0) -- (0,3) -- (0,-3) -- cycle;
\draw[fill=white] (-7,-7) rectangle (-6,-3);
\draw[fill=white] (-7,7) rectangle (-6,3);
\node at (-3,4)[below left]{$\scriptstyle{d-k}$};
\node at (-3,-4)[above left]{$\scriptstyle{d-k}$};
\node at (-7,-7)[left]{$\scriptstyle{s}$};
\node at (-7,7)[left]{$\scriptstyle{t}$};
\end{scope}
\begin{scope}[xshift=9cm]
\draw (6,5) -- (9,5);
\draw (6,-5) -- (9,-5);
\draw[->-=.5] (0,6) -- (6,6);
\draw[->-=.5] (0,-6) -- (6,-6);
\draw[->-=.7] (0,0) to[out=north east, in=west] (6,4);
\draw[->-=.7] (0,0) to[out=south east, in=west] (6,-4);
\draw[fill=white] (3,0) -- (0,3) -- (0,-3) -- cycle;
\node at (7,-7)[right]{$\scriptstyle{s}$};
\node at (7,7)[right]{$\scriptstyle{t}$};
\draw[fill=white] (7,-7) rectangle (6,-3);
\draw[fill=white] (7,7) rectangle (6,3);
\node at (3,4)[below right]{$\scriptstyle{d-k}$};
\node at (3,-4)[above right]{$\scriptstyle{d-k}$};
\end{scope}
\node at (-9,6)[above]{$\scriptstyle{k+(t-d)}$};
\node at (-9,-6)[below]{$\scriptstyle{k+(s-d)}$};
}\ }.
\end{align*}
Let $m$ be positive integers and $\{k_{i}\}_{i}$ are non-negative integers satisfying $k_{i}\geq k_{i+1}$ for $0\leq i\leq m-1$. 
We set $m_{i}=(m-i)$, $k_{0}=d$, $t_{i}=k_{i}+(t-d)$, and $s_{i}=k_{i}+(s-d)$. 
Then, as the above calculation, we obtain
\begin{align*}
&\mathord{\ \tikz[baseline=-.6ex, scale=.1]{
\draw (-6,5) -- (-9,5);
\draw (-6,-5) -- (-9,-5);
\draw[->-=.5] (-6,6) -- (6,6);
\draw[->-=.5] (-6,-6) -- (6,-6);
\draw[->-=.5] (-6,4) to[out=east, in=north west] (0,0);
\draw[->-=.5] (-6,-4) to[out=east, in=south west] (0,0);
\draw[->-=.7] (0,0) to[out=north east, in=west] (6,4);
\draw[->-=.7] (0,0) to[out=south east, in=west] (6,-4);
\draw[fill=white] (-3,0) -- (0,3) -- (3,0) -- (0,-3) -- cycle;
\draw (0,3) -- (0,-3);
\draw[->-=.5] (7,5) -- (12,5);
\draw[->-=.5] (7,-5) -- (12,-5);
\draw[->-=.5] (20,5) -- (25,5);
\draw[->-=.5] (20,-5) -- (25,-5);
\draw (26,5) -- (28,5);
\draw (26,-5) -- (28,-5);
\draw[fill=white] (12,-6) rectangle (20,6);
\draw[fill=white] (-7,-7) rectangle (-6,-3);
\draw[fill=white] (-7,7) rectangle (-6,3);
\draw[fill=white] (7,-7) rectangle (6,-3);
\draw[fill=white] (7,7) rectangle (6,3);
\draw[fill=white] (25,-7) rectangle (26,-3);
\draw[fill=white] (25,7) rectangle (26,3);
\node at (16,0){$\scriptstyle{2m_{i}}$};
\node at (-7,-7)[left]{$\scriptstyle{s_{i-1}}$};
\node at (-7,7)[left]{$\scriptstyle{t_{i-1}}$};
\node at (7,-7)[right]{$\scriptstyle{s_{i-1}}$};
\node at (7,7)[right]{$\scriptstyle{t_{i-1}}$};
\node at (26,-7)[right]{$\scriptstyle{s_{i-1}}$};
\node at (26,7)[right]{$\scriptstyle{t_{i-1}}$};
\node at (0,6)[above]{$\scriptstyle{t_{i}}$};
\node at (0,-6)[below]{$\scriptstyle{s_{i}}$};
\node at (-3,4)[below left]{$\scriptstyle{k_{i-1}-k_{i}}$};
\node at (-3,-4)[above left]{$\scriptstyle{k_{i-1}-k_{i}}$};
}\ }\\
&\quad=q^{-\frac{m_{i}}{3}(k_{i-1}-k_{i})^{2}-m_{i}(k_{i-1}-k_{i})}q^{-\frac{m_{i}}{3}(k_{i-1}-k_{i})(k_{i}+s-d)}q^{-\frac{m_{i}-1}{3}(k_{i-1}-k_{i})(k_{i}+t-d)}\\
&\qquad\times q^{-\frac{2}{3}(k_{i-1}-k_{i})(k_{i}+t-d)}
\mathord{\ \tikz[baseline=-.6ex, scale=.1]{
\begin{scope}[xshift=0cm]
\draw[white, double=black, double distance=0.4pt, ultra thick] (-9,-6) to[out=east, in=west](-4,-6);
\draw[-<-=.9, white, double=black, double distance=0.4pt, ultra thick] (-9,0) -- (0,0);
\draw[white, double=black, double distance=0.4pt, ultra thick] (-9,6) to[out=east, in=west] (-4,-3);
\draw[white, double=black, double distance=0.4pt, ultra thick] (4,-3) to[out=east, in=west] (9,6);
\draw[white, double=black, double distance=0.4pt, ultra thick] (0,0) -- (9,0);
\draw[white, double=black, double distance=0.4pt, ultra thick] (4,-6) to[out=east, in=west] (9,-6);
\draw[fill=white] (-4,-7) rectangle (4,-2);
\node at (0,-5){$\scriptstyle{2m_{i}}$};
\end{scope}
\begin{scope}[xshift=-9cm]
\draw (-6,5) -- (-9,5);
\draw (-6,-5) -- (-9,-5);
\draw[->-=.5] (-6,6) -- (0,6);
\draw[->-=.5] (-6,-6) -- (0,-6);
\draw[->-=.5] (-6,4) to[out=east, in=north west] (0,0);
\draw[->-=.5] (-6,-4) to[out=east, in=south west] (0,0);
\draw[fill=white] (-3,0) -- (0,3) -- (0,-3) -- cycle;
\draw[fill=white] (-7,-7) rectangle (-6,-3);
\draw[fill=white] (-7,7) rectangle (-6,3);
\node at (-3,4)[below left]{$\scriptstyle{k_{i-1}-k_{i}}$};
\node at (-3,-4)[above left]{$\scriptstyle{k_{i-1}-k_{i}}$};
\node at (-7,-7)[left]{$\scriptstyle{s_{i-1}}$};
\node at (-7,7)[left]{$\scriptstyle{t_{i-1}}$};
\end{scope}
\begin{scope}[xshift=9cm]
\draw (6,5) -- (9,5);
\draw (6,-5) -- (9,-5);
\draw[->-=.5] (0,6) -- (6,6);
\draw[->-=.5] (0,-6) -- (6,-6);
\draw[->-=.7] (0,0) to[out=north east, in=west] (6,4);
\draw[->-=.7] (0,0) to[out=south east, in=west] (6,-4);
\draw[fill=white] (3,0) -- (0,3) -- (0,-3) -- cycle;
\node at (7,-7)[right]{$\scriptstyle{s_{i-1}}$};
\node at (7,7)[right]{$\scriptstyle{t_{i-1}}$};
\draw[fill=white] (7,-7) rectangle (6,-3);
\draw[fill=white] (7,7) rectangle (6,3);
\node at (3,4)[below right]{$\scriptstyle{k_{i-1}-k_{i}}$};
\node at (3,-4)[above right]{$\scriptstyle{k_{i-1}-k_{i}}$};
\end{scope}
\node at (-11,6)[above]{$\scriptstyle{t_{i}}$};
\node at (-9,-6)[below]{$\scriptstyle{s_{i}}$};
}\ }\\
&\quad=q^{-\frac{m_{i}}{3}(k_{i-1}-k_{i})^{2}-m_{i}(k_{i-1}-k_{i})}q^{-\frac{m_{i}}{3}(k_{i-1}-k_{i})(k_{i}+s-d)}q^{-\frac{m_{i}}{3}(k_{i-1}-k_{i})(k_{i}+t-d)}\\
&\qquad\times\mathord{\ \tikz[baseline=-.6ex, scale=.1]{
\begin{scope}[xshift=0cm]
\draw (-9,0) -- (-8,0);
\draw (9,0) -- (8,0);
\draw[-<-=.5, white, double=black, double distance=0.4pt, ultra thick] (-7,4) -- (7,4);
\draw[->-=.6, white, double=black, double distance=0.4pt, ultra thick] (-7,-3) -- (-4,-3);
\draw[->-=.6, white, double=black, double distance=0.4pt, ultra thick] (4,-3) -- (7,-3);
\draw[white, double=black, double distance=0.4pt, ultra thick] (-9,-6) -- (9,-6);
\draw[fill=white] (-4,-7) rectangle (4,-2);
\draw[fill=white] (-8,-4) rectangle (-7,7);
\draw[fill=white] (8,-4) rectangle (7,7);
\node at (0,-5) {$\scriptstyle{2m_{i}}$};
\end{scope}
\begin{scope}[xshift=-9cm]
\draw (-6,5) -- (-9,5);
\draw (-6,-5) -- (-9,-5);
\draw[->-=.5] (-6,6) -- (1,6);
\draw[->-=.5] (-6,-6) -- (0,-6);
\draw[->-=.5] (-6,4) to[out=east, in=north west] (0,0);
\draw[->-=.5] (-6,-4) to[out=east, in=south west] (0,0);
\draw[fill=white] (-3,0) -- (0,3) -- (0,-3) -- cycle;
\draw[fill=white] (-7,-7) rectangle (-6,-3);
\draw[fill=white] (-7,7) rectangle (-6,3);
\node at (-3,4)[below left]{$\scriptstyle{k_{i-1}-k_{i}}$};
\node at (-3,-4)[above left]{$\scriptstyle{k_{i-1}-k_{i}}$};
\node at (-7,-7)[left]{$\scriptstyle{s_{i-1}}$};
\node at (-7,7)[left]{$\scriptstyle{t_{i-1}}$};
\end{scope}
\begin{scope}[xshift=9cm]
\draw (6,5) -- (9,5);
\draw (6,-5) -- (9,-5);
\draw[->-=.5] (-1,6) -- (6,6);
\draw[->-=.5] (0,-6) -- (6,-6);
\draw[->-=.7] (0,0) to[out=north east, in=west] (6,4);
\draw[->-=.7] (0,0) to[out=south east, in=west] (6,-4);
\draw[fill=white] (3,0) -- (0,3) -- (0,-3) -- cycle;
\node at (7,-7)[right]{$\scriptstyle{s_{i-1}}$};
\node at (7,7)[right]{$\scriptstyle{t_{i-1}}$};
\draw[fill=white] (7,-7) rectangle (6,-3);
\draw[fill=white] (7,7) rectangle (6,3);
\node at (3,4)[below right]{$\scriptstyle{k_{i-1}-k_{i}}$};
\node at (3,-4)[above right]{$\scriptstyle{k_{i-1}-k_{i}}$};
\end{scope}
\node at (-11,6)[above]{$\scriptstyle{t_{i}}$};
\node at (-9,-6)[below]{$\scriptstyle{s_{i}}$};
\node at (0,4)[below]{$\scriptstyle{k_{i-1}-k_{i}}$};
}\ }\\
&\quad=q^{-\frac{m_{i}}{3}(k_{i-1}-k_{i})^{2}-m_{i}(k_{i-1}-k_{i})}q^{-\frac{m_{i}}{3}(k_{i-1}-k_{i})k_{i}}q^{-\frac{m_{i}}{3}(k_{i-1}-k_{i})(k_{i}+\delta)}\\
&\qquad\times\mathord{\ \tikz[baseline=-.6ex, scale=.1]{
\begin{scope}[xshift=0cm]
\draw (-9,0) -- (-8,0);
\draw (9,0) -- (8,0);
\draw[-<-=.5, white, double=black, double distance=0.4pt, ultra thick] (-7,4) -- (7,4);
\draw[->-=.6, white, double=black, double distance=0.4pt, ultra thick] (-7,-3) -- (-4,-3);
\draw[->-=.6, white, double=black, double distance=0.4pt, ultra thick] (4,-3) -- (7,-3);
\draw[white, double=black, double distance=0.4pt, ultra thick] (-9,-6) -- (9,-6);
\draw[fill=white] (-4,-7) rectangle (4,-2);
\draw[fill=white] (-8,-4) rectangle (-7,7);
\draw[fill=white] (8,-4) rectangle (7,7);
\node at (0,-5) {$\scriptstyle{2m_{i}}$};
\end{scope}
\begin{scope}[xshift=-9cm]
\draw (-6,5) -- (-9,5);
\draw (-6,-5) -- (-9,-5);
\draw[->-=.5] (-6,6) -- (1,6);
\draw[->-=.5] (-6,-6) -- (0,-6);
\draw[->-=.5] (-6,4) to[out=east, in=north west] (0,0);
\draw[->-=.5] (-6,-4) to[out=east, in=south west] (0,0);
\draw[fill=white] (-3,0) -- (0,3) -- (0,-3) -- cycle;
\draw[fill=white] (-7,-7) rectangle (-6,-3);
\draw[fill=white] (-7,7) rectangle (-6,3);
\node at (-3,4)[below left]{$\scriptstyle{k_{i-1}-k_{i}}$};
\node at (-3,-4)[above left]{$\scriptstyle{k_{i-1}-k_{i}}$};
\node at (-7,-7)[left]{$\scriptstyle{t_{i-1}}$};
\node at (-7,7)[left]{$\scriptstyle{t_{i-1}}$};
\end{scope}
\begin{scope}[xshift=9cm]
\draw (6,5) -- (9,5);
\draw (6,-5) -- (9,-5);
\draw[->-=.5] (-1,6) -- (6,6);
\draw[->-=.5] (0,-6) -- (6,-6);
\draw[->-=.7] (0,0) to[out=north east, in=west] (6,4);
\draw[->-=.7] (0,0) to[out=south east, in=west] (6,-4);
\draw[fill=white] (3,0) -- (0,3) -- (0,-3) -- cycle;
\node at (7,-7)[right]{$\scriptstyle{t_{i-1}}$};
\node at (7,7)[right]{$\scriptstyle{t_{i-1}}$};
\draw[fill=white] (7,-7) rectangle (6,-3);
\draw[fill=white] (7,7) rectangle (6,3);
\node at (3,4)[below right]{$\scriptstyle{k_{i-1}-k_{i}}$};
\node at (3,-4)[above right]{$\scriptstyle{k_{i-1}-k_{i}}$};
\end{scope}
\node at (-11,6)[above]{$\scriptstyle{t_{i}}$};
\node at (-9,-6)[below]{$\scriptstyle{s_{i}}$};
\node at (0,4)[below]{$\scriptstyle{k_{i-1}-k_{i}}$};
}\ }\\
&\quad=q^{-\frac{m_{i}}{3}k_{i-1}(k_{i-1}+\delta)-m_{i}k_{i-1}}q^{\frac{m_{i}}{3}k_{i}(k_{i}+\delta)+m_{i}k_{i}}
\mathord{\ \tikz[baseline=-.6ex, scale=.1]{
\begin{scope}[xshift=0cm]
\draw (-9,0) -- (-8,0);
\draw (9,0) -- (8,0);
\draw[-<-=.5, white, double=black, double distance=0.4pt, ultra thick] (-7,4) -- (7,4);
\draw[->-=.6, white, double=black, double distance=0.4pt, ultra thick] (-7,-3) -- (-4,-3);
\draw[->-=.6, white, double=black, double distance=0.4pt, ultra thick] (4,-3) -- (7,-3);
\draw[white, double=black, double distance=0.4pt, ultra thick] (-9,-6) -- (9,-6);
\draw[fill=white] (-4,-7) rectangle (4,-2);
\draw[fill=white] (-8,-4) rectangle (-7,7);
\draw[fill=white] (8,-4) rectangle (7,7);
\node at (0,-5) {$\scriptstyle{2m_{i}}$};
\end{scope}
\begin{scope}[xshift=-9cm]
\draw (-6,5) -- (-9,5);
\draw (-6,-5) -- (-9,-5);
\draw[->-=.5] (-6,6) -- (1,6);
\draw[->-=.5] (-6,-6) -- (0,-6);
\draw[->-=.5] (-6,4) to[out=east, in=north west] (0,0);
\draw[->-=.5] (-6,-4) to[out=east, in=south west] (0,0);
\draw[fill=white] (-3,0) -- (0,3) -- (0,-3) -- cycle;
\draw[fill=white] (-7,-7) rectangle (-6,-3);
\draw[fill=white] (-7,7) rectangle (-6,3);
\node at (-3,4)[below left]{$\scriptstyle{k_{i-1}-k_{i}}$};
\node at (-3,-4)[above left]{$\scriptstyle{k_{i-1}-k_{i}}$};
\node at (-7,-7)[left]{$\scriptstyle{t_{i-1}}$};
\node at (-7,7)[left]{$\scriptstyle{t_{i-1}}$};
\end{scope}
\begin{scope}[xshift=9cm]
\draw (6,5) -- (9,5);
\draw (6,-5) -- (9,-5);
\draw[->-=.5] (-1,6) -- (6,6);
\draw[->-=.5] (0,-6) -- (6,-6);
\draw[->-=.7] (0,0) to[out=north east, in=west] (6,4);
\draw[->-=.7] (0,0) to[out=south east, in=west] (6,-4);
\draw[fill=white] (3,0) -- (0,3) -- (0,-3) -- cycle;
\node at (7,-7)[right]{$\scriptstyle{t_{i-1}}$};
\node at (7,7)[right]{$\scriptstyle{t_{i-1}}$};
\draw[fill=white] (7,-7) rectangle (6,-3);
\draw[fill=white] (7,7) rectangle (6,3);
\node at (3,4)[below right]{$\scriptstyle{k_{i-1}-k_{i}}$};
\node at (3,-4)[above right]{$\scriptstyle{k_{i-1}-k_{i}}$};
\end{scope}
\node at (-11,6)[above]{$\scriptstyle{t_{i}}$};
\node at (-9,-6)[below]{$\scriptstyle{s_{i}}$};
\node at (0,4)[below]{$\scriptstyle{k_{i-1}-k_{i}}$};
}\ }.
\end{align*}
Let $\tau_{0}$ be the parallel $m$ full twist in the LHS of Theorem~\ref{parallelfull} and
\[
\tau_{i}
=\mathord{\ \tikz[baseline=-.6ex, scale=.1]{
\begin{scope}[xshift=0cm]
\draw (-9,0) -- (-8,0);
\draw (9,0) -- (8,0);
\draw[-<-=.5, white, double=black, double distance=0.4pt, ultra thick] (-7,4) -- (7,4);
\draw[->-=.6, white, double=black, double distance=0.4pt, ultra thick] (-7,-3) -- (-4,-3);
\draw[->-=.6, white, double=black, double distance=0.4pt, ultra thick] (4,-3) -- (7,-3);
\draw[white, double=black, double distance=0.4pt, ultra thick] (-9,-6) -- (9,-6);
\draw[fill=white] (-4,-7) rectangle (4,-2);
\draw[fill=white] (-8,-4) rectangle (-7,7);
\draw[fill=white] (8,-4) rectangle (7,7);
\node at (0,-5) {$\scriptstyle{2m_{i}}$};
\end{scope}
\begin{scope}[xshift=-9cm]
\draw (-6,5) -- (-9,5);
\draw (-6,-5) -- (-9,-5);
\draw[->-=.5] (-6,6) -- (1,6);
\draw[->-=.5] (-6,-6) -- (0,-6);
\draw[->-=.5] (-6,4) to[out=east, in=north west] (0,0);
\draw[->-=.5] (-6,-4) to[out=east, in=south west] (0,0);
\draw[fill=white] (-3,0) -- (0,3) -- (0,-3) -- cycle;
\draw[fill=white] (-7,-7) rectangle (-6,-3);
\draw[fill=white] (-7,7) rectangle (-6,3);
\node at (-3,4)[below left]{$\scriptstyle{k_{i-1}-k_{i}}$};
\node at (-3,-4)[above left]{$\scriptstyle{k_{i-1}-k_{i}}$};
\node at (-7,-7)[left]{$\scriptstyle{s_{i-1}}$};
\node at (-7,7)[left]{$\scriptstyle{t_{i-1}}$};
\end{scope}
\begin{scope}[xshift=9cm]
\draw (6,5) -- (9,5);
\draw (6,-5) -- (9,-5);
\draw[->-=.5] (-1,6) -- (6,6);
\draw[->-=.5] (0,-6) -- (6,-6);
\draw[->-=.7] (0,0) to[out=north east, in=west] (6,4);
\draw[->-=.7] (0,0) to[out=south east, in=west] (6,-4);
\draw[fill=white] (3,0) -- (0,3) -- (0,-3) -- cycle;
\node at (7,-7)[right]{$\scriptstyle{s_{i-1}}$};
\node at (7,7)[right]{$\scriptstyle{t_{i-1}}$};
\draw[fill=white] (7,-7) rectangle (6,-3);
\draw[fill=white] (7,7) rectangle (6,3);
\node at (3,4)[below right]{$\scriptstyle{k_{i-1}-k_{i}}$};
\node at (3,-4)[above right]{$\scriptstyle{k_{i-1}-k_{i}}$};
\end{scope}
\node at (-11,6)[above]{$\scriptstyle{t_{i}}$};
\node at (-9,-6)[below]{$\scriptstyle{s_{i}}$};
\node at (0,4)[below]{$\scriptstyle{k_{i-1}-k_{i}}$};
}\ }
\]
for $i=1,2,\dots,m-1$.
Then,
we apply Theorem~\ref{parallelfull}, the index of summation is $k_{i+1}$, to the left-most full twist in the box
$
\mathord{\ \tikz[baseline=-.6ex, scale=.1]{
\node[draw, rectangle] {$\scriptstyle{2m_{i}}$};
}\ }$ 
of $\tau_{i}$ and the above deformation.
\begin{align*}
\tau_{i}
&=q^{-\frac{1}{3}k_{i}(k_{i}+\delta)-\frac{k_{i}}{2}}
\sum_{k_{i+1}=0}^{k_{i}}
q^{k_{i+1}(k_{i+1}+\delta)+\frac{k_{i+1}}{2}}
\frac{(q)_{k_{i}+\delta}}{(q)_{k_{i+1}+\delta}}{k_{i} \choose k_{i+1}}_{q}\\
&\times q^{-\frac{m_{i+1}}{3}k_{i}(k_{i}+\delta)-m_{i+1}k_{i}}q^{\frac{m_{i+1}}{3}k_{i+1}(k_{i+1}+\delta)+m_{i+1}k_{i+1}}
\mathord{\ \tikz[baseline=-.6ex, scale=.1]{
\begin{scope}[xshift=0cm]
\draw (-9,0) -- (-8,0);
\draw (9,0) -- (8,0);
\draw[-<-=.5, white, double=black, double distance=0.4pt, ultra thick] (-7,4) -- (7,4);
\draw[->-=.6, white, double=black, double distance=0.4pt, ultra thick] (-7,-3) -- (-4,-3);
\draw[->-=.6, white, double=black, double distance=0.4pt, ultra thick] (4,-3) -- (7,-3);
\draw[white, double=black, double distance=0.4pt, ultra thick] (-9,-6) -- (9,-6);
\draw[fill=lightgray] (-4,-7) rectangle (4,-2);
\draw[fill=white] (-8,-4) rectangle (-7,7);
\draw[fill=white] (8,-4) rectangle (7,7);
\node at (0,-5) {$\scriptstyle{\tau_{i+1}}$};
\end{scope}
\begin{scope}[xshift=-9cm]
\draw (-6,5) -- (-9,5);
\draw (-6,-5) -- (-9,-5);
\draw[->-=.5] (-6,6) -- (1,6);
\draw[->-=.5] (-6,-6) -- (0,-6);
\draw[->-=.5] (-6,4) to[out=east, in=north west] (0,0);
\draw[->-=.5] (-6,-4) to[out=east, in=south west] (0,0);
\draw[fill=white] (-3,0) -- (0,3) -- (0,-3) -- cycle;
\draw[fill=white] (-7,-7) rectangle (-6,-3);
\draw[fill=white] (-7,7) rectangle (-6,3);
\node at (-3,4)[below left]{$\scriptstyle{k_{i-1}-k_{i}}$};
\node at (-3,-4)[above left]{$\scriptstyle{k_{i-1}-k_{i}}$};
\node at (-7,-7)[left]{$\scriptstyle{t_{i-1}}$};
\node at (-7,7)[left]{$\scriptstyle{t_{i-1}}$};
\end{scope}
\begin{scope}[xshift=9cm]
\draw (6,5) -- (9,5);
\draw (6,-5) -- (9,-5);
\draw[->-=.5] (-1,6) -- (6,6);
\draw[->-=.5] (0,-6) -- (6,-6);
\draw[->-=.7] (0,0) to[out=north east, in=west] (6,4);
\draw[->-=.7] (0,0) to[out=south east, in=west] (6,-4);
\draw[fill=white] (3,0) -- (0,3) -- (0,-3) -- cycle;
\node at (7,-7)[right]{$\scriptstyle{t_{i-1}}$};
\node at (7,7)[right]{$\scriptstyle{t_{i-1}}$};
\draw[fill=white] (7,-7) rectangle (6,-3);
\draw[fill=white] (7,7) rectangle (6,3);
\node at (3,4)[below right]{$\scriptstyle{k_{i-1}-k_{i}}$};
\node at (3,-4)[above right]{$\scriptstyle{k_{i-1}-k_{i}}$};
\end{scope}
\node at (-11,6)[above]{$\scriptstyle{t_{i}}$};
\node at (-9,-6)[below]{$\scriptstyle{s_{i}}$};
\node at (0,4)[below]{$\scriptstyle{k_{i-1}-k_{i}}$};
}\ }\\
&=\sum_{k_{i+1}=0}^{k_{i}}
q^{-\frac{m_{i}}{3}k_{i}(k_{i}+\delta)+\frac{m_{i+1}}{3}k_{i+1}(k_{i+1}+\delta)}
q^{-\frac{1}{2}(k_{i}-k_{i+1})}q^{-m_{i+1}(k_{i}-k_{i+1})}
q^{k_{i+1}(k_{i+1}+\delta)}\\
&\quad\times 
\frac{(q)_{k_{i}+\delta}}{(q)_{k_{i+1}+\delta}}{k_{i} \choose k_{i+1}}_{q}
\mathord{\ \tikz[baseline=-.6ex, scale=.1]{
\begin{scope}[xshift=0cm]
\draw (-9,0) -- (-8,0);
\draw (9,0) -- (8,0);
\draw[-<-=.5, white, double=black, double distance=0.4pt, ultra thick] (-7,4) -- (7,4);
\draw[->-=.6, white, double=black, double distance=0.4pt, ultra thick] (-7,-3) -- (-4,-3);
\draw[->-=.6, white, double=black, double distance=0.4pt, ultra thick] (4,-3) -- (7,-3);
\draw[white, double=black, double distance=0.4pt, ultra thick] (-9,-6) -- (9,-6);
\draw[fill=lightgray] (-4,-7) rectangle (4,-2);
\draw[fill=white] (-8,-4) rectangle (-7,7);
\draw[fill=white] (8,-4) rectangle (7,7);
\node at (0,-5) {$\scriptstyle{\tau_{i+1}}$};
\end{scope}
\begin{scope}[xshift=-9cm]
\draw (-6,5) -- (-9,5);
\draw (-6,-5) -- (-9,-5);
\draw[->-=.5] (-6,6) -- (1,6);
\draw[->-=.5] (-6,-6) -- (0,-6);
\draw[->-=.5] (-6,4) to[out=east, in=north west] (0,0);
\draw[->-=.5] (-6,-4) to[out=east, in=south west] (0,0);
\draw[fill=white] (-3,0) -- (0,3) -- (0,-3) -- cycle;
\draw[fill=white] (-7,-7) rectangle (-6,-3);
\draw[fill=white] (-7,7) rectangle (-6,3);
\node at (-3,4)[below left]{$\scriptstyle{k_{i-1}-k_{i}}$};
\node at (-3,-4)[above left]{$\scriptstyle{k_{i-1}-k_{i}}$};
\node at (-7,-7)[left]{$\scriptstyle{t_{i-1}}$};
\node at (-7,7)[left]{$\scriptstyle{t_{i-1}}$};
\end{scope}
\begin{scope}[xshift=9cm]
\draw (6,5) -- (9,5);
\draw (6,-5) -- (9,-5);
\draw[->-=.5] (-1,6) -- (6,6);
\draw[->-=.5] (0,-6) -- (6,-6);
\draw[->-=.7] (0,0) to[out=north east, in=west] (6,4);
\draw[->-=.7] (0,0) to[out=south east, in=west] (6,-4);
\draw[fill=white] (3,0) -- (0,3) -- (0,-3) -- cycle;
\node at (7,-7)[right]{$\scriptstyle{t_{i-1}}$};
\node at (7,7)[right]{$\scriptstyle{t_{i-1}}$};
\draw[fill=white] (7,-7) rectangle (6,-3);
\draw[fill=white] (7,7) rectangle (6,3);
\node at (3,4)[below right]{$\scriptstyle{k_{i-1}-k_{i}}$};
\node at (3,-4)[above right]{$\scriptstyle{k_{i-1}-k_{i}}$};
\end{scope}
\node at (-11,6)[above]{$\scriptstyle{t_{i}}$};
\node at (-9,-6)[below]{$\scriptstyle{s_{i}}$};
\node at (0,4)[below]{$\scriptstyle{k_{i-1}-k_{i}}$};
}\ }\\
&=\sum_{k_{i+1}=0}^{k_{i}}
q^{-\frac{m_{i}}{3}k_{i}(k_{i}+\delta)+\frac{m_{i+1}}{3}k_{i+1}(k_{i+1}+\delta)}
q^{-\frac{1}{2}(k_{i}-k_{i+1})}q^{-m_{i}k_{i}+m_{i+1}k_{i+1}}\\
&\quad\times 
q^{k_{i}}q^{k_{i+1}(k_{i+1}+\delta)}
\frac{(q)_{k_{i}+\delta}}{(q)_{k_{i+1}+\delta}}{k_{i} \choose k_{i+1}}_{q}
\mathord{\ \tikz[baseline=-.6ex, scale=.1]{
\begin{scope}[xshift=0cm]
\draw (-9,0) -- (-8,0);
\draw (9,0) -- (8,0);
\draw[-<-=.5, white, double=black, double distance=0.4pt, ultra thick] (-7,4) -- (7,4);
\draw[->-=.6, white, double=black, double distance=0.4pt, ultra thick] (-7,-3) -- (-4,-3);
\draw[->-=.6, white, double=black, double distance=0.4pt, ultra thick] (4,-3) -- (7,-3);
\draw[white, double=black, double distance=0.4pt, ultra thick] (-9,-6) -- (9,-6);
\draw[fill=lightgray] (-4,-7) rectangle (4,-2);
\draw[fill=white] (-8,-4) rectangle (-7,7);
\draw[fill=white] (8,-4) rectangle (7,7);
\node at (0,-5) {$\scriptstyle{\tau_{i+1}}$};
\end{scope}
\begin{scope}[xshift=-9cm]
\draw (-6,5) -- (-9,5);
\draw (-6,-5) -- (-9,-5);
\draw[->-=.5] (-6,6) -- (1,6);
\draw[->-=.5] (-6,-6) -- (0,-6);
\draw[->-=.5] (-6,4) to[out=east, in=north west] (0,0);
\draw[->-=.5] (-6,-4) to[out=east, in=south west] (0,0);
\draw[fill=white] (-3,0) -- (0,3) -- (0,-3) -- cycle;
\draw[fill=white] (-7,-7) rectangle (-6,-3);
\draw[fill=white] (-7,7) rectangle (-6,3);
\node at (-3,4)[below left]{$\scriptstyle{k_{i-1}-k_{i}}$};
\node at (-3,-4)[above left]{$\scriptstyle{k_{i-1}-k_{i}}$};
\node at (-7,-7)[left]{$\scriptstyle{t_{i-1}}$};
\node at (-7,7)[left]{$\scriptstyle{t_{i-1}}$};
\end{scope}
\begin{scope}[xshift=9cm]
\draw (6,5) -- (9,5);
\draw (6,-5) -- (9,-5);
\draw[->-=.5] (-1,6) -- (6,6);
\draw[->-=.5] (0,-6) -- (6,-6);
\draw[->-=.7] (0,0) to[out=north east, in=west] (6,4);
\draw[->-=.7] (0,0) to[out=south east, in=west] (6,-4);
\draw[fill=white] (3,0) -- (0,3) -- (0,-3) -- cycle;
\node at (7,-7)[right]{$\scriptstyle{t_{i-1}}$};
\node at (7,7)[right]{$\scriptstyle{t_{i-1}}$};
\draw[fill=white] (7,-7) rectangle (6,-3);
\draw[fill=white] (7,7) rectangle (6,3);
\node at (3,4)[below right]{$\scriptstyle{k_{i-1}-k_{i}}$};
\node at (3,-4)[above right]{$\scriptstyle{k_{i-1}-k_{i}}$};
\end{scope}
\node at (-11,6)[above]{$\scriptstyle{t_{i}}$};
\node at (-9,-6)[below]{$\scriptstyle{s_{i}}$};
\node at (0,4)[below]{$\scriptstyle{k_{i-1}-k_{i}}$};
}\ },
\end{align*}
where the shaded box with $\tau_{i+1}$ means replacement of the box with $\tau_{i+1}$.
We remark that $m_{0}=m$, $t_{0}=t$, $s_{0}=s$, and $\min\{s_{i-1},t_{i-1}\}=k_{i-1}$.
By using a similar way to (\ref{uniontri}), 
we can confirm that 
\[
 \mathord{\ \tikz[baseline=-.6ex, scale=.1]{
\begin{scope}[xshift=0cm]
\draw (-9,0) -- (-8,0);
\draw (9,0) -- (8,0);
\draw[-<-=.5, white, double=black, double distance=0.4pt, ultra thick] (-7,4) -- (7,4);
\draw[->-=.6, white, double=black, double distance=0.4pt, ultra thick] (-7,-3) -- (-4,-3);
\draw[->-=.6, white, double=black, double distance=0.4pt, ultra thick] (4,-3) -- (7,-3);
\draw[white, double=black, double distance=0.4pt, ultra thick] (-9,-6) -- (9,-6);
\draw[fill=lightgray] (-4,-7) rectangle (4,-2);
\draw[fill=white] (-8,-4) rectangle (-7,7);
\draw[fill=white] (8,-4) rectangle (7,7);
\node at (0,-5) {$\scriptstyle{\tau_{i+1}}$};
\end{scope}
\begin{scope}[xshift=-9cm]
\draw (-6,5) -- (-9,5);
\draw (-6,-5) -- (-9,-5);
\draw[->-=.5] (-6,6) -- (1,6);
\draw[->-=.5] (-6,-6) -- (0,-6);
\draw[->-=.5] (-6,4) to[out=east, in=north west] (0,0);
\draw[->-=.5] (-6,-4) to[out=east, in=south west] (0,0);
\draw[fill=white] (-3,0) -- (0,3) -- (0,-3) -- cycle;
\draw[fill=white] (-7,-7) rectangle (-6,-3);
\draw[fill=white] (-7,7) rectangle (-6,3);
\node at (-3,4)[below left]{$\scriptstyle{k_{i-1}-k_{i}}$};
\node at (-3,-4)[above left]{$\scriptstyle{k_{i-1}-k_{i}}$};
\node at (-7,-7)[left]{$\scriptstyle{t_{i-1}}$};
\node at (-7,7)[left]{$\scriptstyle{t_{i-1}}$};
\end{scope}
\begin{scope}[xshift=9cm]
\draw (6,5) -- (9,5);
\draw (6,-5) -- (9,-5);
\draw[->-=.5] (-1,6) -- (6,6);
\draw[->-=.5] (0,-6) -- (6,-6);
\draw[->-=.7] (0,0) to[out=north east, in=west] (6,4);
\draw[->-=.7] (0,0) to[out=south east, in=west] (6,-4);
\draw[fill=white] (3,0) -- (0,3) -- (0,-3) -- cycle;
\node at (7,-7)[right]{$\scriptstyle{t_{i-1}}$};
\node at (7,7)[right]{$\scriptstyle{t_{i-1}}$};
\draw[fill=white] (7,-7) rectangle (6,-3);
\draw[fill=white] (7,7) rectangle (6,3);
\node at (3,4)[below right]{$\scriptstyle{k_{i-1}-k_{i}}$};
\node at (3,-4)[above right]{$\scriptstyle{k_{i-1}-k_{i}}$};
\end{scope}
\node at (-11,6)[above]{$\scriptstyle{t_{i}}$};
\node at (-9,-6)[below]{$\scriptstyle{s_{i}}$};
\node at (0,4)[below]{$\scriptstyle{k_{i-1}-k_{i}}$};
}\ }
=\mathord{\ \tikz[baseline=-.6ex, scale=.12]{
\begin{scope}[xshift=0cm]
\draw (-9,0) -- (-8,0);
\draw (9,0) -- (8,0);
\draw[-<-=.5, white, double=black, double distance=0.4pt, ultra thick] (-7,4) -- (7,4);
\draw[->-=.6, white, double=black, double distance=0.4pt, ultra thick] (-7,-3) -- (-4,-3);
\draw[->-=.6, white, double=black, double distance=0.4pt, ultra thick] (4,-3) -- (7,-3);
\draw[white, double=black, double distance=0.4pt, ultra thick] (-9,-6) -- (9,-6);
\draw[fill=white] (-4,-7) rectangle (4,-2);
\draw[fill=white] (-8,-4) rectangle (-7,7);
\draw[fill=white] (8,-4) rectangle (7,7);
\node at (0,-5) {$\scriptstyle{2m_{i+1}}$};
\end{scope}
\begin{scope}[xshift=-9cm]
\draw (-6,5) -- (-9,5);
\draw (-6,-5) -- (-9,-5);
\draw[->-=.5] (-6,6) -- (1,6);
\draw[->-=.5] (-6,-6) -- (0,-6);
\draw[->-=.5] (-6,4) to[out=east, in=north west] (0,0);
\draw[->-=.5] (-6,-4) to[out=east, in=south west] (0,0);
\draw[fill=white] (-3,0) -- (0,3) -- (0,-3) -- cycle;
\draw[fill=white] (-7,-7) rectangle (-6,-3);
\draw[fill=white] (-7,7) rectangle (-6,3);
\node at (-3,4)[below left]{$\scriptstyle{k_{i-1}-k_{i+1}}$};
\node at (-3,-4)[above left]{$\scriptstyle{k_{i-1}-k_{i+1}}$};
\node at (-7,-7)[left]{$\scriptstyle{s_{i-1}}$};
\node at (-7,7)[left]{$\scriptstyle{t_{i-1}}$};
\end{scope}
\begin{scope}[xshift=9cm]
\draw (6,5) -- (9,5);
\draw (6,-5) -- (9,-5);
\draw[->-=.5] (-1,6) -- (6,6);
\draw[->-=.5] (0,-6) -- (6,-6);
\draw[->-=.7] (0,0) to[out=north east, in=west] (6,4);
\draw[->-=.7] (0,0) to[out=south east, in=west] (6,-4);
\draw[fill=white] (3,0) -- (0,3) -- (0,-3) -- cycle;
\node at (7,-7)[right]{$\scriptstyle{s_{i-1}}$};
\node at (7,7)[right]{$\scriptstyle{t_{i-1}}$};
\draw[fill=white] (7,-7) rectangle (6,-3);
\draw[fill=white] (7,7) rectangle (6,3);
\node at (3,4)[below right]{$\scriptstyle{k_{i-1}-k_{i+1}}$};
\node at (3,-4)[above right]{$\scriptstyle{k_{i-1}-k_{i+1}}$};
\end{scope}
\node at (-11,6)[above]{$\scriptstyle{t_{i+1}}$};
\node at (-9,-6)[below]{$\scriptstyle{s_{i+1}}$};
\node at (0,4)[below]{$\scriptstyle{k_{i-1}-k_{i+1}}$};
}\ }.
\]
Consequently, 
\begin{align*}
\tau_{0}&=\sum_{k_{0}\geq k_{1}\geq\cdots\geq k_{m}\geq 0}
\prod_{i=0}^{m-1}
q^{-\frac{m_{i}}{3}k_{i}(k_{i}+\delta)+\frac{m_{i+1}}{3}k_{i+1}(k_{i+1}+\delta)}
q^{-\frac{1}{2}(k_{i}-k_{i+1})}q^{-m_{i}k_{i}+m_{i+1}k_{i+1}}\\
&\quad\times 
q^{k_{i+1}(k_{i+1}+\delta)+k_{i}}
\frac{(q)_{k_{i}+\delta}}{(q)_{k_{i+1}+\delta}}{k_{i} \choose k_{i+1}}_{q}
\mathord{\ \tikz[baseline=-.6ex, scale=.1]{
\draw (-6,5) -- (-9,5);
\draw (-6,-5) -- (-9,-5);
\draw (6,5) -- (9,5);
\draw (6,-5) -- (9,-5);
\draw[->-=.5] (-6,6) -- (6,6);
\draw[->-=.5] (-6,-6) -- (6,-6);
\draw[->-=.5] (-6,4) to[out=east, in=north west] (0,0);
\draw[->-=.5] (-6,-4) to[out=east, in=south west] (0,0);
\draw[->-=.7] (0,0) to[out=north east, in=west] (6,4);
\draw[->-=.7] (0,0) to[out=south east, in=west] (6,-4);
\draw[fill=white] (-3,0) -- (0,3) -- (3,0) -- (0,-3) -- cycle;
\draw (0,3) -- (0,-3);
\draw[fill=white] (-7,-7) rectangle (-6,-3);
\draw[fill=white] (-7,7) rectangle (-6,3);
\draw[fill=white] (7,-7) rectangle (6,-3);
\draw[fill=white] (7,7) rectangle (6,3);
\node at (-7,-7)[left]{$\scriptstyle{s}$};
\node at (-7,7)[left]{$\scriptstyle{t}$};
\node at (7,-7)[right]{$\scriptstyle{s}$};
\node at (7,7)[right]{$\scriptstyle{t}$};
\node at (0,6)[above]{$\scriptstyle{t-d+k_{m}}$};
\node at (0,-6)[below]{$\scriptstyle{s-d+k_{m}}$};
\node at (-3,4)[below left]{$\scriptstyle{d-k_{m}}$};
\node at (-3,-4)[above left]{$\scriptstyle{d-k_{m}}$};
\node at (3,4)[below right]{$\scriptstyle{d-k_{m}}$};
\node at (3,-4)[above right]{$\scriptstyle{d-k_{m}}$};
}\ }\\
&=\sum_{k_{0}\geq k_{1}\geq\cdots\geq k_{m}\geq 0}
q^{-\frac{m}{3}k_{0}(k_{0}+\delta)}
q^{-\frac{1}{2}(k_{0}-k_{m})}q^{-mk_{0}}
q^{\sum_{i=0}^{m-1}k_{i+1}(k_{i+1}+\delta)+k_{i}}\\
&\quad\times 
\frac{(q)_{k_{0}+\delta}}{(q)_{k_{m}+\delta}}\frac{(q)_{k_{0}}}{(q)_{k_{0}-k_{1}}(q)_{k_{1}-k_{2}}\dots (q)_{k_{m-1}-k_{m}}(q)_{k_{m}}}
\mathord{\ \tikz[baseline=-.6ex, scale=.1]{
\draw (-6,5) -- (-9,5);
\draw (-6,-5) -- (-9,-5);
\draw (6,5) -- (9,5);
\draw (6,-5) -- (9,-5);
\draw[->-=.5] (-6,6) -- (6,6);
\draw[->-=.5] (-6,-6) -- (6,-6);
\draw[->-=.5] (-6,4) to[out=east, in=north west] (0,0);
\draw[->-=.5] (-6,-4) to[out=east, in=south west] (0,0);
\draw[->-=.7] (0,0) to[out=north east, in=west] (6,4);
\draw[->-=.7] (0,0) to[out=south east, in=west] (6,-4);
\draw[fill=white] (-3,0) -- (0,3) -- (3,0) -- (0,-3) -- cycle;
\draw (0,3) -- (0,-3);
\draw[fill=white] (-7,-7) rectangle (-6,-3);
\draw[fill=white] (-7,7) rectangle (-6,3);
\draw[fill=white] (7,-7) rectangle (6,-3);
\draw[fill=white] (7,7) rectangle (6,3);
\node at (-7,-7)[left]{$\scriptstyle{s}$};
\node at (-7,7)[left]{$\scriptstyle{t}$};
\node at (7,-7)[right]{$\scriptstyle{s}$};
\node at (7,7)[right]{$\scriptstyle{t}$};
\node at (0,6)[above]{$\scriptstyle{t-d+k_{m}}$};
\node at (0,-6)[below]{$\scriptstyle{s-d+k_{m}}$};
\node at (-3,4)[below left]{$\scriptstyle{d-k_{m}}$};
\node at (-3,-4)[above left]{$\scriptstyle{d-k_{m}}$};
\node at (3,4)[below right]{$\scriptstyle{d-k_{m}}$};
\node at (3,-4)[above right]{$\scriptstyle{d-k_{m}}$};
}\ }\\
&=q^{-\frac{m}{3}k_{0}(k_{0}+\delta)-mk_{0}}\sum_{k_{0}\geq k_{1}\geq\cdots\geq k_{m}\geq 0}
q^{\frac{1}{2}(k_{0}-k_{m})}
q^{\sum_{i=1}^{m}k_{i}(k_{i}+\delta)+k_{i}}\\
&\quad\times 
\frac{(q)_{k_{0}+\delta}}{(q)_{k_{m}+\delta}}\frac{(q)_{k_{0}}}{(q)_{k_{0}-k_{1}}(q)_{k_{1}-k_{2}}\cdots (q)_{k_{m-1}-k_{m}}(q)_{k_{m}}}
\mathord{\ \tikz[baseline=-.6ex, scale=.1]{
\draw (-6,5) -- (-9,5);
\draw (-6,-5) -- (-9,-5);
\draw (6,5) -- (9,5);
\draw (6,-5) -- (9,-5);
\draw[->-=.5] (-6,6) -- (6,6);
\draw[->-=.5] (-6,-6) -- (6,-6);
\draw[->-=.5] (-6,4) to[out=east, in=north west] (0,0);
\draw[->-=.5] (-6,-4) to[out=east, in=south west] (0,0);
\draw[->-=.7] (0,0) to[out=north east, in=west] (6,4);
\draw[->-=.7] (0,0) to[out=south east, in=west] (6,-4);
\draw[fill=white] (-3,0) -- (0,3) -- (3,0) -- (0,-3) -- cycle;
\draw (0,3) -- (0,-3);
\draw[fill=white] (-7,-7) rectangle (-6,-3);
\draw[fill=white] (-7,7) rectangle (-6,3);
\draw[fill=white] (7,-7) rectangle (6,-3);
\draw[fill=white] (7,7) rectangle (6,3);
\node at (-7,-7)[left]{$\scriptstyle{s}$};
\node at (-7,7)[left]{$\scriptstyle{t}$};
\node at (7,-7)[right]{$\scriptstyle{s}$};
\node at (7,7)[right]{$\scriptstyle{t}$};
\node at (0,6)[above]{$\scriptstyle{t-d+k_{m}}$};
\node at (0,-6)[below]{$\scriptstyle{s-d+k_{m}}$};
\node at (-3,4)[below left]{$\scriptstyle{d-k_{m}}$};
\node at (-3,-4)[above left]{$\scriptstyle{d-k_{m}}$};
\node at (3,4)[below right]{$\scriptstyle{d-k_{m}}$};
\node at (3,-4)[above right]{$\scriptstyle{d-k_{m}}$};
}\ }
\end{align*}
\end{proof}

\section{Application to the $\mathfrak{sl}_3$ tail}
In this section, 
we discuss the $\mathfrak{sl}_3$ tail of $(2,2m)$-torus link with one-row colorings.
A tail of a link $L$ is a limit of the $\mathfrak{sl}_2$ colored Jones polynomials $\{J_{n}^{\mathfrak{sl}_2}(L)\}_{n}$. 
The existence of the tail is proven for adequate links in \cite{Armond13}.
The tail of $(2,2m)$ torus link is given by the false theta series and its explicit formula was obtained in \cite{Hajij16}.
The tail or such stability of the $\mathfrak{g}$ colored Jones polynomials of torus knots in a case where $\mathfrak{g}$ is the rank $2$ simple Lie algebra was shown by Garoufalidis and Vuong~\cite{GaroufalidisVuong17}. 
In the same paper, 
They gave explicit formulas of the tail for $(2,2m+1)$- and $(4,5)$-torus knots in the case of $\mathfrak{g}=\mathfrak{sl}_3$.
The author also gave two explicit formulas of the $\mathfrak{sl}_3$ tail for the $(2,2m)$-torus link with anti-parallel orientation in \cite{Yuasa18} and it gave Andrews-Gordon type identities for a ``diagonal summand'' of the $\mathfrak{sl}_3$ theta series, see also Section~9 in \cite{BringmannKaszianMilas19}.

In the following, we give an explicit formula of one-row colored $\mathfrak{sl}_3$ Jones polynomial $J_{(n,0)}^{\mathfrak{sl}_3}(T_{\rightrightarrows}(2,2m))$ for the $(2,2m)$-torus link with parallel orientation and its $\mathfrak{sl}_3$ tail.
We use the following normalization of $\mathfrak{sl}_3$ colored Jones polynomial for an oriented framed link.
Let $D=D_{1}\sqcup D_{2}\sqcup \cdots \sqcup D_{p}$ be a link diagram of a framed link $L$ with $p$ components through the blackboard framing. 
A coloring $\omega_{i}$ of $D_{i}$ is given by an irreducible representations of $\mathfrak{sl}_{3}$ with the highest weight $(m_{i},n_{i})$ for $p=1,2,\dots,p$.
One can identify the coloring $\omega_{i}$ and $(m_{i},n_{i})\in\mathbb{Z}_{{}\geq 0}\times \mathbb{Z}_{{}\geq 0}$.
Let $D(\omega)$ be the $A_2$ web obtained by replacing $D_{i}$ with $JW_{\omega_{i}}$ where the orientation of the strand colored by $m$ coincides with the orientation of $D_{i}$ for all $i=1,2,\dots,p$. 
See Figure~{\ref{CJP}}.

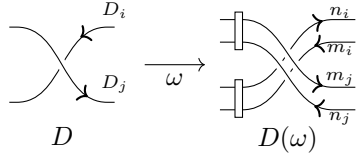
\begin{figure}
\begin{tikzpicture}[scale=.1]
\draw[-<-=.8, white, double=black, double distance=0.4pt, ultra thick] (-7,-5) -- (-6,-5) to[out=east, in=west] (6,5) -- (7,5);
\draw[->-=.8, white, double=black, double distance=0.4pt, ultra thick] (-7,5) -- (-6,5) to[out=east, in=west] (6,-5) -- (7,-5);
\node at (7,-5) [above]{$\scriptstyle{D_{j}}$};
\node at (7,5) [above]{$\scriptstyle{D_{i}}$};
\node at (0,-7) [below]{$D$};
\begin{scope}[xshift=15cm]
\draw[->] (-4,0) -- (4,0);
\node at (0,0) [below]{$\omega$};
\end{scope}
\begin{scope}[xshift=30cm]
\draw[-<-=.8, white, double=black, double distance=0.4pt, ultra thick] (-9,-6) -- (-6,-6) to[out=east, in=west] (6,3) -- (9,3);
\draw[->-=.8, white, double=black, double distance=0.4pt, ultra thick] (-9,-3) -- (-6,-3) to[out=east, in=west] (6,6) -- (9,6);
\draw[->-=.8, white, double=black, double distance=0.4pt, ultra thick] (-9,6) -- (-6,6) to[out=east, in=west] (6,-3) -- (9,-3);
\draw[-<-=.8, white, double=black, double distance=0.4pt, ultra thick] (-9,3) -- (-6,3) to[out=east, in=west] (6,-6) -- (9,-6);
\draw[fill=white] (-7,7) rectangle (-6,2);
\draw[fill=white] (-7,-7) rectangle (-6,-2);
\node at (7,5) [above]{$\scriptstyle{n_{i}}$};
\node at (7,4) [below]{$\scriptstyle{m_{i}}$};
\node at (7,-5) [below]{$\scriptstyle{n_{j}}$};
\node at (7,-4) [above]{$\scriptstyle{m_{j}}$};
\node at (0,-7) [below]{$D(\omega)$};
\end{scope}
\end{tikzpicture}
\caption{The construction of $A_2$ web from a link diagram $D$ and a coloring $\omega$.}
\label{CJP}
\end{figure}

\begin{DEF}
$J_{\omega}^{\mathfrak{sl}_3}(L)=D(\omega)\in\mathbb{C}(q^{\frac{1}{6}})$.
\end{DEF}

\begin{RMK}
The above normalization of the $\mathfrak{sl}_3$ colored Jones polynomial of a link is different from in \cite{Yuasa18}.
\end{RMK}

\begin{LEM}\label{closure}
Let $s$ and $t$ be positive integers and $d=\min\{s,t\}$.
For any integer $0\leq k\leq d$,
\[
\mathord{\ \tikz[baseline=-.6ex, scale=.1]{
\draw[rounded corners, -<-=.5] (-6,5) -- (-10,5) -- (-10,9) -- (10,9) -- (10,5) -- (6,5);
\draw[rounded corners, -<-=.5] (-6,-5) -- (-12,-5) -- (-12,11) -- (12,11) -- (12,-5) -- (6,-5);
\draw[->-=.5] (-6,6) -- (6,6);
\draw[->-=.5] (-6,-6) -- (6,-6);
\draw[->-=.5] (-6,4) to[out=east, in=north west] (0,0);
\draw[->-=.5] (-6,-4) to[out=east, in=south west] (0,0);
\draw[->-=.7] (0,0) to[out=north east, in=west] (6,4);
\draw[->-=.7] (0,0) to[out=south east, in=west] (6,-4);
\draw[fill=white] (-3,0) -- (0,3) -- (3,0) -- (0,-3) -- cycle;
\draw (0,3) -- (0,-3);
\draw[fill=white] (-7,-7) rectangle (-6,-3);
\draw[fill=white] (-7,7) rectangle (-6,3);
\draw[fill=white] (7,-7) rectangle (6,-3);
\draw[fill=white] (7,7) rectangle (6,3);
\node at (-10,7)[right]{$\scriptstyle{t}$};
\node at (-12,7)[left]{$\scriptstyle{s}$};
\node at (-3,4)[below left]{$\scriptstyle{d-k}$};
\node at (-3,-4)[above left]{$\scriptstyle{d-k}$};
\node at (3,4)[below right]{$\scriptstyle{d-k}$};
\node at (3,-4)[above right]{$\scriptstyle{d-k}$};
}\ }
=[d-k+1]\frac{\Delta(s,0)\Delta(t,0)}{\Delta(d-k,0)}
\]
where $\Delta$ is an oriented circle.
\end{LEM}
\begin{proof}
One can calculate the above web by using the following formulas (see, for example, \cite{Yuasa17}).
\begin{align*}
\mathord{\ \tikz[baseline=-.6ex, scale=.1]{
\draw[->-=.3] (-12,0) -- (-5,0);
\draw[-<-=.3] (12,0) -- (5,0);
\draw[-<-=.7] (-5,0) to[out=north, in=west] (0,4);
\draw[->-=.7] (5,0) to[out=north, in=east] (0,4);
\draw[-<-=.7] (-5,0) to[out=south, in=west] (0,-4);
\draw[->-=.7] (5,0) to[out=south, in=east] (0,-4);
\draw[fill=white] (0,2) rectangle (1,6);
\draw[fill=white] (0,-2) rectangle (1,-6);
\draw[fill=white, xshift=-5cm] (0:2) -- (120:2) -- (240:2) -- cycle;
\draw[fill=white, xshift=5cm] (60:2) -- (180:2) -- (300:2) -- cycle;
\draw[fill=white] (-8,-2) rectangle (-7,2);
\draw[fill=white] (8,-2) rectangle (7,2);
\node at (-2,3)[above]{$\scriptstyle{n}$};
\node at (-2,-3)[below]{$\scriptstyle{n}$};
\node at (-8,0)[above left]{$\scriptstyle{n}$};
\node at (8,0)[above right]{$\scriptstyle{n}$};
}\ }
&=
\left[n+1\right]
\mathord{\ \tikz[baseline=-.6ex,scale=.1]{
\draw[->-=.8] (-5,0) -- (5,0);
\draw[fill=white] (0,-2) rectangle (1,2);
\node at (-2,0)[above]{$\scriptstyle{n}$};
}\ },\\
\mathord{\ \tikz[baseline=-.6ex, scale=.1]{
\draw[-<-=.2] (-5,1) -- (5,1);
\draw[-<-=.3, rounded corners] (-3,-1) -- (-3,-5) -- (3,-5) -- (3,-1) -- cycle;
\draw[fill=white] (0,-3) rectangle (1,3);
\node at (-5,1)[above]{$\scriptstyle{n-k}$};
\node at (0,-5)[below]{$\scriptstyle{k}$};
}\ }
=\frac{\Delta(n,0)}{\Delta(n-k,0)}
\mathord{\ \tikz[baseline=-.6ex, scale=.1]{
\draw[-<-=.3] (-4,0) -- (4,0);
\draw[fill=white] (-1,-2) rectangle (0,2);
\node at (4,0)[above]{$\scriptstyle{n-k}$};
}\ }
&=\frac{\left[n+1\right]\left[n+2\right]}{\left[n-k+1\right]\left[n-k+2\right]}
\mathord{\ \tikz[baseline=-.6ex, scale=.1]{
\draw[-<-=.3] (-4,0) -- (4,0);
\draw[fill=white] (-1,-2) rectangle (0,2);
\node at (4,0)[above]{$\scriptstyle{n-k}$};
}\ }
\end{align*}
\end{proof}

\begin{THM}\label{paralleltorus}
Let 
$T_{\rightrightarrows}(2,2m)=
\mathord{\ \tikz[baseline=-.6ex, scale=.1, yshift=-2cm]{
\draw[->-=.7, rounded corners] (-5,2) -- (5,2) -- (5,5) -- (-5,5) -- cycle;
\draw[->-=.7, rounded corners] (-8,-2) -- (8,-2) -- (8,8) -- (-8,8) -- cycle;
\draw[fill=white] (-3,-3) rectangle (3,3);
\node at (0,0){$\scriptstyle{2m}$};
}\ }
$
be the $(2,2m)$-torus link with parallel orientation. 
For any one-row colorings $\omega_1=(s,0)$ and $\omega_2=(t,0)$,
\begin{align*}
&J_{\omega}^{\mathfrak{sl}_3}(T(2,2m))=q^{-\frac{m}{3}d(d+\delta)-md}\\
&\quad\times\hspace{-1em}\sum_{d\geq k_{1}\geq\cdots\geq k_{m}\geq 0}\hspace{-1em}
q^{\sum_{i=1}^{m}k_{i}(k_{i}+\delta)+k_{i}}q^{d-k_{m}}
\frac{(q)_{d+\delta}}{(q)_{k_{m}+\delta}}\frac{(q)_{d}}{(q)_{d-k_{1}}(q)_{k_{1}-k_{2}}\cdots (q)_{k_{m-1}-k_{m}}(q)_{k_{m}}}\\
&\quad\times\frac{1-q^{2}}{1-q^{d-k_{m}+1}}\Delta(s,0)\Delta(t,0)
\end{align*}
where $d=\min\{s,t\}$ and $\delta=\left| s-t \right|$.
\end{THM}
\begin{proof}
One can easily obtain the above formula from Theorem~{\ref{parallelmfull}}, Lemma~{\ref{closure}}, and
\[
\Delta(n,0)=\frac{\left[n+1\right]\left[n+2\right]}{\left[2\right]}=q^{-n}\frac{(1-q^{n+1})(1-q^{n+2})}{(1-q)(1-q^{2})}.
\]
\end{proof}

Let us consider a family of formal power series $\{f_{n}(q)\in\mathbb{Z}[[q]]\mid n\in\mathbb{Z}_{{}\geq 0}\}$ and $F(q)\in\mathbb{Z}[[q]]$.
Then, we define $\lim_{n\to\infty}f_n(q)=F(q)$ if $f_n(q)=F(q)$ in $\mathbb{Z}[[q]]/q^{n+1}\mathbb{Z}[[q]]$ for any $n$.

We review an explicit formula of the $\mathfrak{sl}_3$ tail of the $(2,2m)$-torus link with anti-parallel orientation
$T_{\leftrightarrows}(2,2m)=
\mathord{\ \tikz[baseline=-.6ex, scale=.1, yshift=-2cm]{
\draw[-<-=.7, rounded corners] (-5,2) -- (5,2) -- (5,5) -- (-5,5) -- cycle;
\draw[->-=.7, rounded corners] (-8,-2) -- (8,-2) -- (8,8) -- (-8,8) -- cycle;
\draw[fill=white] (-3,-3) rectangle (3,3);
\node at (0,0){$\scriptstyle{2m}$};
}\ }
$.
\begin{THM}[The $\mathfrak{sl}_3$ colored Jones polynomial of $T_{\leftrightarrows}(2,2m)$~\cite{Yuasa17}]\label{antiparalleltorus}
For any one-row colorings $\omega_1=(s,0)$ and $\omega_2=(t,0)$,
\begin{align*}
&J_{\omega}^{\mathfrak{sl}_3}(T_{\leftrightarrows}(2,2m))=q^{-\frac{2m}{3}d(d+\delta)-2md}\\
&\quad\times\hspace{-1em}\sum_{d\geq k_{1}\geq\cdots\geq k_{m}\geq 0}\hspace{-1em}
q^{\sum_{i=1}^{m}k_{i}(k_{i}+\delta)+2k_{i}}q^{2d-2k_{m}}
\frac{(q)_{d+\delta}}{(q)_{k_{m}+\delta}}\frac{(q)_{d}}{(q)_{d-k_{1}}(q)_{k_{1}-k_{2}}\cdots (q)_{k_{m-1}-k_{m}}(q)_{k_{m}}}\\
&\quad\times\frac{(1-q)(1-q^{2})}{(1-q^{d-k_{m}+1})(1-q^{d-k_{m}+2})}\Delta(s,0)\Delta(t,0)
\end{align*}
where $d=\min\{s,t\}$ and $\delta=\left| s-t \right|$.
\end{THM}

\begin{THM}[The $\mathfrak{sl}_3$ tail of $T_{\leftrightarrows}(2,2m)$~\cite{Yuasa18}]\label{antiparalleltail}
Let $\omega_{n}$ be a special coloring of $T_{\leftrightarrows}(2,2m)$ such that $s=t=n$.
Then,
\[
\lim_{n\to\infty}f_n(q)=\frac{1}{(1-q)(1-q^{2})}\sum_{k_{1}\geq k_{2}\geq\cdots\geq k_{m}\geq 0}
\frac{q^{-2k_{m}}q^{\sum_{i=1}^{m}k_{i}^2+2k_{i}}}{(q)_{k_{1}-k_{2}}(q)_{k_{2}-k_{3}}\cdots (q)_{k_{m-1}-k_{m}}(q)_{k_{m}}}
\]
where $f_{n}(q)=q^{\frac{2m}{3}n^2+2mn}J_{\omega_{n}}^{\mathfrak{sl}_3}(T_{\leftrightarrows}(2,2m))$.
\end{THM}

We obtain an explicit formula of the $\mathfrak{sl}_3$ tail of $T_{\rightrightarrows}(2,2m)$ from Theorem~\ref{paralleltorus}
\begin{THM}[The $\mathfrak{sl}_3$ tail of $T_{\rightrightarrows}(2,2m)$]
Let $\omega_{n}$ be a special coloring of $T_{\rightrightarrows}(2,2m)$ such that $s=t=n$. 
Then,
\[
 \lim_{n\to\infty}f_n(q)=\frac{1}{(1-q)^{2}(1-q^{2})}\sum_{k_{1}\geq k_{2}\geq\cdots\geq k_{m}\geq 0}
\frac{q^{-k_{m}}q^{\sum_{i=1}^{m}k_{i}^2+k_{i}}}{(q)_{k_{1}-k_{2}}(q)_{k_{2}-k_{3}}\cdots (q)_{k_{m-1}-k_{m}}(q)_{k_{m}}}
\]
where $f_n(q)=q^{\frac{m}{3}n^2+mn}q^{n}J_{\omega_{n}}^{\mathfrak{sl}_3}(T_{\rightrightarrows}(2,2m))$.
\end{THM}
\begin{proof}
It is easy to see that $q^{k_{1}^2+k_{1}}\frac{(q)_{n}}{(q)_{n-k_{1}}}=q^{k_1^2+k_1}\in\mathbb{Z}[[q]]/q^{n+1}\mathbb{Z}[[q]]$ and $q^{k_{m}^2}\frac{(q)_{n}}{(q)_{k_{m}}}=q^{k_{m}^2}\in\mathbb{Z}[[q]]/q^{n+1}\mathbb{Z}[[q]]$.
\end{proof}

\subsection *{Acknowledgment}
This work was supported by Grant-in-Aid for JSPS Fellows Grant Number 19J00252 and Grant-in-Aid for Early-Career Scientists Grant Number 19K14528.

\bibliographystyle{amsalpha}
\bibliography{sl3theta_1row}
\end{document}